\documentclass[11pt,a4paper]{article}
\usepackage[english]{babel}
\usepackage{amsmath,					
			amsthm,
    		amsrefs,
			amssymb
			}
\usepackage[mathscr]{euscript}

\usepackage[all]{xy}
\usepackage{tikz-cd} 

\usepackage{authblk}   

\usepackage[hmargin=2.2cm,vmargin=2.5cm]{geometry}		
\numberwithin{equation}{section}

\usepackage[colorlinks]{hyperref}			

\newcommand{\smooth}{$C^\infty$}

\newcommand{\bigslant}[2]{{\left.\raisebox{.2em}{$#1$}\middle/\raisebox{-.2em}{$#2$}\right.}}
\newcommand{\F}[1]{\mathfrak{#1}}
\newcommand{\C}[1]{\mathcal{#1}}
\newcommand{\R}[1]{\mathrm{#1}}

\newcommand{\g}{\mathfrak{g}^{\ast}}

\newcommand{\Qtot}{Q_{\mathcal{E},\mathcal{E}'}}

\newcommand{\Qtoto}{Q_{\mathcal{E},\mathcal{E}'}^{(0)}}

\newcommand{\cinf}{\mathcal{C}^{\infty}}

\newcommand{\dd}{\mathrm{d}}

\newcommand{\ad}{\mathrm{ad}}
\newcommand{\id}{\mathrm{id}}

\newcommand\DistTo{\xrightarrow{
		\,\smash{\raisebox{-0.65ex}{\ensuremath{\scriptstyle\sim}}}\,}}

\def\g{\mathfrak{g}}

\newcommand{\isthiswhatyouowant}{\smallskip \noindent}

\newcommand{\Funct}{\mathscr{O}}

\newcommand{\mN}{\mathbb{N}}
\newcommand{\mR}{\mathbb{R}}

\newcommand{\Vect}{\F{X}}

\newcommand{\diff}{\mathrm d}
\newcommand{\Linfty}{Lie\ $\infty$}
\newcommand{\LieInftyAlgebroid}{\Linfty-algebroid}

\newcommand{\cF}{\mathcal{F}}
\newcommand{\cO}{\mathcal{O}}
\newcommand{\cI}{\mathfrak{I}}

\usetikzlibrary{calc,bending}
\def\CricArrowRight#1{\tikz[baseline=(A.base)]
  \draw[-stealth,line width=.035em]
    (0,0) node[circle, inner sep=0cm](A){$#1$}
    let \p1=(A.center),\p2=(A.west), \n1={\x1-\x2} in
      (-90:\n1) arc(-90:190:\n1);}

\newcommand{\resolving}{of }

\newtheorem{definition}{Definition}[section]
\newtheorem{theoreme}[definition]{Theorem}
\newtheorem{proposition}[definition]{Proposition}

\newtheorem{lemme}[definition]{Lemma}

\newtheorem{corollaire}[definition]{Corollary}
\newtheorem*{convention}{Convention}
\newtheorem{remarque}[definition]{Remark}

\newtheorem{example}[definition]{Example}

\renewenvironment{abstract}
{
\noindent \rule{\linewidth}{.5pt}

\par{\bfseries \noindent \abstractname.}}
{
\vspace{0cm}

\noindent 
\rule{\linewidth}{.5pt}

\noindent {\textbf{Keywords}}: Singular foliations and singular leaves, Lie $\infty$-algebroids and Q-manifolds.
}

\title{The universal Lie \texorpdfstring{$\infty$}{infinity}-algebroid of a singular foliation}

\author{ Camille Laurent-Gengoux\thanks{Institut Elie Cartan de Lorraine, UMR 7502, Universit\'e de Lorraine,  France} \, \, \, \, 
	Sylvain Lavau\thanks{Instisut Math\'ematiques de Jussieu, Universit\'e Paris Diderot, Paris.} \, \, \, \, Thomas Strobl\thanks{ UMI  CNRS-2924, Instituto  de  Matem\'atica  Pura  e  Aplicada (IMPA), Rio de Janeiro, Brazil, 
		and Institut  Camille  Jordan,  Universit\'e de Lyon, France.}}

\AtEndDocument{\bigskip{\footnotesize%
  \textsc{C. Laurent-Gengoux,  Institut Elie Cartan de Lorraine (UMR 7502), Universit\'e de Lorraine,}
  \textit{ 3 rue Augustin Fresnel, 57000 Metz-Technop\^ole, France} 
  \par  
  \textit{E-mail address}, \texttt{camille.laurent-gengoux@univ-lorraine.fr} \par
  \addvspace{\medskipamount}
  \textsc{S. Lavau, Universit\'e Paris Diderot,} \textit{B\^atiment Sophie Germain, 8 Place Aur\'elie Nemours, 75013 Paris, France} 
  \par  
  \textit{E-mail address}, \texttt{sylvain.lavau@ens-lyon.fr} \par
\addvspace{\medskipamount}
  \textsc{T. Strobl, CNRS-2924, Instituto  de  Matem\'atica  Pura  e  Aplicada  (IMPA),}  \textit{ Estrada Dona Castorina 110, Rio de Janeiro, 22460-320, Brazil} 
  and  \textsc{Institut  Camille  Jordan,  Universit\'e Claude  Bernard  Lyon  1,} \textit{43 Boulevard du 11 Novembre 1918, 69622 Villeurbanne Cedex, France}.
  \par
  \textit{E-mail address}, \texttt{thomas@impa.br} and \texttt{strobl@math.univ-lyon1.fr } \par
}}

\date{}

\begin{document}
\maketitle

\date

\begin{abstract}
	\isthiswhatyouowant
	We consider singular foliations ${\cal F}$ as locally finitely generated $ {\mathscr O}$-submodules of  $ {\mathscr O}$-derivations closed under the Lie bracket, where  ${\mathscr O}$ is the ring of smooth,
holomorphic, or real analytic functions on some correspondingly chosen manifold. We first collect and/or prove several results about the existence of resolutions of such an ${\cal F}$ in terms of sections of vector bundles. In particular, on a compact smooth manifold $M$, such resolutions of length at most $\dim M + 1$ always exist, provided only that locally ${\cal F}$ admits real analytic generators. 

\isthiswhatyouowant
We show that every complex of vector bundles $(E_\bullet,\dd)$ over  $M$ providing a resolution of a given singular foliation ${\cal F}$  in the above sense admits the definition of brackets on its sections such that it extends these data into a {\LieInftyAlgebroid}.  This {\LieInftyAlgebroid}, including the chosen
underlying resolution, is unique up to homotopy and, moreover, every
other {\LieInftyAlgebroid} inducing the given ${\cal F}$ or any of its sub-foliations 
factors through it in an up-to-homotopy unique manner. We therefore call 
it the universal {\LieInftyAlgebroid} of the singular foliation ${\cal F}$. 

\isthiswhatyouowant
It encodes several aspects of the
geometry of the leaves of ${\cal F}$. In particular, it permits us to recover the holonomy groupoid of
	Androulidakis and Skandalis. Moreover, each leaf carries an isotropy
{\Linfty}-algebra structure that is unique up to isomorphism. It extends a minimal isotropy Lie algebra, that can be associated to each leaf, by higher brackets, which give rise to additional invariants of the foliation. As a byproduct, we construct an example of a foliation ${\cal F}$ generated by $r$ vector fields for which we show by these techniques that, even locally, it cannot result from 
a Lie algebroid of the minimal rank $r$. 
\end{abstract}
\newpage

\tableofcontents

\section{Introduction}

\isthiswhatyouowant
Regular foliations, i.e.~a partition of a manifold into embedded submanifolds of a given dimension, are familiar objects of interest in differential geometry, see e.g.~\cite{Hector}.
According to the Frobenius theorem, they are equivalent to involutive distributions. 

\isthiswhatyouowant
Singular foliations, on the other hand, are much less understood while, at the same time, they appear much more frequently. Typical Lie group actions have orbits of 
different dimensions. Similarly, the symplectic leaves of a Poisson manifold change dimension whenever the rank of the 
bivector jumps. Both of these two classes of singular foliations are an example of what one obtains on the base manifold of a Lie algebroid. 
It is therefore natural to ask if \emph{any} singular foliation arises in this way. 

\isthiswhatyouowant
To make this question more precise, we first need to clarify. One way of viewing a singular foliation would be a partition of the given manifold
into immersed submanifolds of possibly different dimensions. While in the case of regular foliations, the description in terms of generating vector fields 
is completely equivalent, here the latter characterization contains more information. Consider, for example, vector fields on a line vanishing at the origin up to order $k$.
While the corresponding partition of $\mR$ into leaves consists of $\mR_+$, $\mR_-$, and the origin $0$,
the generating module of functions is in addition invariantly characterized by the integer $k \in \mN$. We will therefore define a singular
foliation as an involutive submodule ${\cal{F}} \subset \Gamma(TM)$, where as usual, involutivity means $[\cal{F},\cal{F}] \subset \cal{F}$. 
\isthiswhatyouowant
Defined like this, it is, however, not even guaranteed that the vector fields $\cal{F}$ generate a subdivision of $M$ into leaves such that, at each point, $\cal{F}$ evaluated 
at this point would agree with the tangent of the leaf containing this point. Also, if we do not restrict $\cal{F}$ further, the answer to the question if it is generated by a 
Lie algebroid $A$ is definitely negative: The image of $A$ with respect to the anchor $\rho \colon A \to TM$ gives an involutive module ${\cal{F}} = \rho\big(\Gamma(A)\big)$, but 
evidently this is in addition locally finitely generated. Adding this as a  condition on $\cal{F}$, namely to be locally finitely generated, 
Hermann's theorem establishes that $M$ is indeed partitioned
by immersed submanifolds, called leaves \cite{Hermann1962}. 

\isthiswhatyouowant
Thus in this paper, we define a singular foliation of a manifold $M$ to be a locally finitely generated involutive $\mathscr{O}$-submodule of vector fields on $M$.
This perspective seems to also become more and more the prevailing one these days \cites{AndrouSkandal, AndrouZambis}. Here $\mathscr{O}$ can be chosen to be the 
ring of smooth functions $C^\infty(M)$, or, in the case that $M$ is a real analytic or a complex manifold, the ring of real analytic and holomorphic functions, respectively. 

\isthiswhatyouowant
So, now we are in the position to again pose the question about the existence of a finite rank Lie algebroid over $M$ that would induce a given singular
foliation on this manifold. This question can be split into a global and a local one. While it is easy to see that the answer to the global problem posed as such is 
negative---the minimal number of local generators of $\cal{F}$ does not need to be finite\footnote{Take $M=\mR^2$ and look at the module of vector fields vanishing to order $k$ at $(0,k)$ for all $k \in \mN$. The number of generators is unbounded in this case. \cite{AndrouZambis}}---the local problem is much more intricate and still open---however, after shifting our focus to other questions below, we will be able to  give partial answers to this question as well. But what the results exposed in the present article show is rather that there is a more interesting and more important question:
%
%
Is there a {\LieInftyAlgebroid} generating a given singular foliation? And, if so, can it be used to find invariants of the foliation? Both of these latter two questions 
will be answered essentially in the positive in our paper: For example, in the case of $\mathscr{O}$ being real analytic functions, we show that \emph{every} singular foliation $\cal F$ is locally 
generated by a {\LieInftyAlgebroid}.\footnote{For $\mathscr{O}=C^\infty(M)$ we need to assume the existence of a geometric resolution of ${\mathcal F}$ to arrive at the conclusion, 
something one can always infer locally in the real analytic or holomorphic case, cf.~below.}
In the case of $\mathscr{O}$ being smooth functions, the same conclusion holds over any relatively compact open subset of $M$, if, locally, the foliation can be generated by real analytic vector fields.
 Even more importantly, there is such a {\LieInftyAlgebroid} whose homotopy class is unique and universal:
the one constructed on a geometric resolution of the singular foliation. This is in sharp contrast to the Lie algebroid story: not only is a Lie algebroid $A$ over $M$
for a given singular foliation $\cal F$ far from unique, if it exists at all, also its homotopy class cannot be unique since homotopies of Lie algebroids do not change the
rank of the underlying vector bundle $A$. 

\isthiswhatyouowant
If we take any other {\LieInftyAlgebroid} whose induced foliation is $\cal F$, or even only a submodule
of $\cal F$, and whose underlying complex is now not necessarily a geometric resolution, there exists a morphism into every {\LieInftyAlgebroid} whose underlying  complex is a geometric resolution.  This morphism itself is unique up to homotopy (see Theorem \ref{theo:onlyOne}). So, considering the category of  {\LieInftyAlgebroid}s up to homotopy,
i.e.~the category where objects are  {\LieInftyAlgebroid}s over $M$ inducing a sub-singular foliation of $ {\mathcal F}$ 
and where arrows are homotopy classes of {\LieInftyAlgebroid} morphisms, {\LieInftyAlgebroid}s constructed on geometric resolutions are a terminal object. 
This justifies to call them \emph{universal {\LieInftyAlgebroid}s \resolving a singular foliation ${\mathcal F}$}.

\isthiswhatyouowant
The universal {\LieInftyAlgebroid} \resolving a singular foliation ${\mathcal F}$ 
turns out to be an efficient tool for the construction of invariants associated to a singular foliation, such as different types
of cohomology classes associated to a {\LieInftyAlgebroid} representing its universal class.
Let us explain the construction a bit further, starting first with the case that $\mathscr{O}$ is real analytic or holomorphic. Then by Syzygy theorems in the neighborhood
of any point $m \in M$ the $\mathscr{O}$-module $\cal F$ admits geometric resolutions of finite length by free modules, which we can reinterpret as sections of trivial vector bundles 
over the neighborhood:
\begin{center}
\begin{tikzcd}[column sep=0.7cm,row sep=0.4cm]  \label{sequenceE}
0\ar[r]&\Gamma(E_{-n-1})\ar[r]&\ldots\ar[r]&\Gamma(E_{-1})\ar[r,"\rho"]&\mathcal{F}\ar[r]&0
\end{tikzcd}
\end{center}
where $n$ is the dimension of $M$. The {\LieInftyAlgebroid} is then constructed over the corresponding complex of vector bundles $E_{-n-1} \to \ldots \to E_{-1}$ by showing the existence of $s$-brackets for $s = 2, \ldots , n+1$ so that together with the differential of the complex they satisfy the required higher Jacobi identities. Thus, in this case we even obtain a Lie $k$-algebroid with $k=n+1$. 
 Certainly, the above geometric resolution is not unique; in particular,
the ranks of the bundles are far from being fixed. These individual ranks can now be changed by homotopies in the category of {\LieInftyAlgebroid}s, only their total index
\begin{equation}
\mathrm{Ind}(E_\bullet) := \sum_{i \geq 1} (-1)^{i+1} \mathrm{rk}(E_{-i})
\end{equation}
remains invariant. This index, on the other hand, is nothing but the highest possible dimension of the leaves in the neighborhood of $m$. 
There are certainly much more subtle
invariants associated to the foliation that result from our construction. 
For instance, restricting the universal {\LieInftyAlgebroid}s \resolving a singular foliation ${\mathcal F}$ to a point and taking its cohomology,
we get a \Linfty-algebra, that we call the isotropy \Linfty-algebra of ${\mathcal F}$ at the point $m \in M$. This  \Linfty-algebra has by construction no $1$-ary bracket,
i.e.~no differential.
Therefore, its $3$-ary bracket 
is a class in the Chevalley-Eilenberg cohomology for the isotropy graded Lie algebra bracket given by the $2$-ary bracket.
We show, on the other hand, that there cannot be a generating Lie algebroid of minimal rank in the neighborhood of a point $m$ where this class does not vanish. 
An explicit example of such a foliation will be provided in 
Example \ref{ex:NMLRA} below: the vector fields on ${\mathbb C}^4\ni (x,y,z,t)$ preserving the function $x^3+y^3+z^3+t^3$ 
form a singular foliation of rank six, i.e.~they need at least six generating vector fields to be defined by means of generators and relations. Since its above 3-class is shown to be non-zero, 
it follows that this particular singular foliation cannot be defined as the image through the anchor map of a Lie algebroid of rank six.
Notice that, given a singular foliation of rank $r$, the problem of finding a Lie algebroid or rank $r$ that induces it 
has a priori no relation with higher structures. It is interesting to see that it can be answered through the use of those.
\bigskip

\isthiswhatyouowant
The structure of the paper is as follows.
Section \ref{sec:main} contains the main definitions and results about the
construction of the universal {\LieInftyAlgebroid} of a singular foliation $\cF$ as well as its algebraic consequences for $\cF$. 
Proofs, examples, and some definitions are postponed to the subsequent two sections, Section \ref{sec:construction} and Section \ref{sec:geometry}.
An ordered list of examples of singular foliations and some useful
lemmas are presented in Section \ref{DefExSSingFol}.
In Section \ref{existence} we address
the
  question of whether and when a singular foliation $ {\mathcal F}$
admits a geometric resolution as a module over functions.
  Here we do not just mean a projective resolution,
   but a resolution by sections of vector bundles. Around a point, it is
equivalent to require
   the existence of a resolution of ${\mathcal F}$ by free finitely
generated ${\mathcal O}$-modules.
   In general, the answer is no, and a counter-example is given, but
classical results, called syzygy theorems,
   imply that the answer is yes in the real analytic, algebraic, and
holomorphic cases in a neighborhood of a point.
   Moreover, in the real analytic case, the real-analytic geometric resolution can
be proved to be also a smooth geometric resolution by classical
   results of Malgrange and Tougeron. In turn, we show that these smooth resolutions can be glued to yield a geometric resolution on a relatively compact open subset. This part then is followed by
examples of such resolutions in Section \ref{exampleresolutions}. Section \ref{subsec:homotopies} recalls classical results about
{\LieInftyAlgebroid}s,  in particular the useful perspective of it as a
differential positively graded manifold, equivalently known under the
name of an NQ-manifold.
We put particular emphasis on
homotopies between {\LieInftyAlgebroid}s morphisms, where precision
about  boundary conditions is required; we believe that the point of
view presented about homotopies will be of
  interest also in other contexts. 
   Only then, in Sections and \ref{sectionpreuve1}  and \ref{sec:context}, we turn to the proof of
the main theorems, Theorem \ref{theo:existe} about equipping any such a
resolution with a
   {\LieInftyAlgebroid} structure and Theorem \ref{theo:onlyOne} about
its uniqueness up to homotopy and its universality property.
   We prove all these theorems by careful step-by-step constructions of
brackets, morphisms, and homotopies.
   We conclude this section by providing examples in Section
\ref{sec:examplesLieInfinity}.
  
  \isthiswhatyouowant
Section \ref{sec:geometry} is devoted to the geometrical meaning of the
previously found structures and induced algebraic invariants.
Since the universal {\LieInftyAlgebroid}  \resolving a singular
foliation is unique up to homotopy, most cohomologies constructed out of it
do not depend on the choices made in the construction and are thus
associated to the initial foliation.
In particular, as argued in  Section \ref{sec:universal}, the cohomology
of the degree 1 vector field $Q$ describing the universal
{\LieInftyAlgebroid} \resolving
a singular foliation is canonical, i.e.~it depends only on the singular
foliation.
In Section \ref{sec:isotropy} we derive even more interesting
cohomological spaces
by restricting the universal {\LieInftyAlgebroid} structure
to a point $m \in M$. This is analogous to the familiar situation for
Lie algebroids, where the Lie algebroid bracket induces a Lie algebra
bracket
on the kernel of the anchor map at a given point, called the isotropy
Lie algebra of the point or its leaf (since for different points on a
given leaf these Lie algebras are isomorphic).
Essentially the same construction applies here and allows us to induce
a \Linfty-algebra  on a graded vector space which coincides with
the fiber over $m$ of the
resolution of the foliation in degree $-2,-3, \dots$ and to the kernel
of the anchor map in degree $-1$.
If the geometric resolution is chosen to be minimal at $m$---the ranks of the
vector bundles
that define the geometric resolution are as small as they can be---we obtain a
\Linfty-algebra that we call the isotropy \Linfty-algebra
of ${\mathcal F}$ at the point $m \in M$. It has several interesting
features:
First, its differential or $1$-ary bracket vanishes, so that its $2$-ary
bracket defines an honest graded Lie algebra.
But it may still have $k$-ary brackets for $k\geq 3$.
Second, this structure is unique up to isomorphism, its $2$-ary bracket
being unique even on the nose, cf.~Proposition \ref{prop:isotropyStrictIso}
below.
In Section \ref{moineau} we then prove that, like for isotropy Lie
algebras of Lie algebroids, the isotropy \Linfty-algebras of
{\LieInftyAlgebroid}s
only change by isomorphisms along any leaf of ${\mathcal F}$. Section
\ref{sec:examplesLieInfinityIso} contains examples of these isotropy algebras.

\isthiswhatyouowant
In Section \ref{sec:3aryNotZero} we return to the issue of
{\LieInftyAlgebroid} versus Lie algebroid. Evidently, we can always add
a non-acting Lie algebra to every Lie algebroid, which increases its
isotropy Lie algebras at each point accordingly while not changing the
induced foliation.  This is in sharp contrast to the isotropy \Linfty-algebras of a
singular foliation and its graded Lie algebra introduced in this paper. Moreover, as we will prove,
a Lie algebroid inducing a given singular foliation, even if it exists,
may in some cases need more generators than the initial singular
foliation does. More precisely,
we will show that the $3$-ary bracket of the isotropy \Linfty-algebra of
${\mathcal F}$ at every point $m \in M$ is a Chevalley-Eilenberg cocycle
with respect to the $2$-ary bracket and that this cocycle is exact if
there exists
a Lie algebroid of rank $r$ defining the singular foliation
(with $r$ being the rank of the foliation).
Example \ref{ex:NMLRA} then presents a singular foliation for which this
Chevalley-Eilenberg class does not vanish.
Section \ref{sec:leibnizoid} concludes this discussion with a side
remark that every singular foliation  admitting a geometric resolution of finite
length is the image through
the anchor map of a Leibniz algebroid.

\isthiswhatyouowant
   In Section \ref{sec:sub-grou} we show that our structure induces the
holonomy algebroid and groupoid of Androulidakis and Skandalis
\cite{AndrouSkandal} by an appropriate truncation (and integration),
   cf.~, in particular, Proposition \ref{prop:holoLieAlg} below.

\isthiswhatyouowant
The existence of {\LieInftyAlgebroid}s seems therefore actually more 
interesting than an answer to the initially posed question about a Lie algebroid generating the same foliation. Among others, 
the universal {\LieInftyAlgebroid} provides a whole bunch of cohomological invariants associated to a singular foliation that are directly suggested 
by the data of this construction, which would supposedly not be
so easy to construct by other means, and in particular not by methods adapted to ordinary Lie algebroids.


\subsection*{Relation to other work:} 
\noindent
 Tom Lada and Jim Stasheff present in \cite{LS} a
 construction of a \Linfty-algebra on a resolution
 of a Lie algebra: our construction follows the same pattern and is
 inspired by theirs.
We were told that results more or less equivalent to
Theorem \ref{theo:existe} were discussed  by Ralph Mayer and Chenchang
Zhu as well as by Ted Voronov
and his collaborators.
Moreover, Johannes Huebschmann suggests such a result also in his
introduction to \cite{huebschmann}, without giving further details though.
We claim, however, that we are the first ones to have clearly stated and
published a proof for Theorem \ref{theo:existe} and,
more importantly, to have found, formulated, and proven unicity, cf.~Theorem \ref{theo:onlyOne}.
Also, we are not aware of any predecessor for Section
\ref{sec:geometry}, where we derive some geometric implications of the
construction. In November 2018, Ya\"el Fr\'egier and Rigel A. Juarez-Ojeda presented in \cite{FregierRigel} a construction based on model theory, which is different in nature from ours but allows them to prove a result similar to Theorems \ref{theo:existe} and \ref{theo:onlyOne}.

\isthiswhatyouowant
Several results of the present work were already presented in the Ph-D thesis of S.L. \cite{thesis}, defended under the supervision
of Henning Samtleben and T.S. in November~2016. First versions of the present work were made public through the electronic systems HAL and ArXiv in December 2017 and June 2018, respectively.

\subsection*{Acknowledgment}

\noindent

\isthiswhatyouowant
We have benefited of several crucial comments by Jim
Stasheff and feel honored by the interest he took in this project.
In addition to his numerous mathematical comments he corrected many
typos and mistakes of English.
The present work is also the outcome of several years of discussions and
lectures. We mention in the text the very points where we received direct help from
Iakovos Androulidakis, Alexei Bolsinov, Alfonso Garmendia, Fran\c{c}ois Petit (several times), Pavel Safronov,
Jean-Louis Tu, Rupert Yu, and Marco Zambon.
Finally, we acknowledge important discussions with and comments from Anton Alekseev, 
Damien Calaque, Pierre Cartier, Claire Debord, Ya\"el Fr\'egier, Benjamin Hennion, Rigel A. Juarez-Ojeda, Yvette
Kosmann-Schwarzbach, Alexei Kotov, Pavol \v{S}evera, Bruno Vallette, Ted Voronov and Ping Xu.

\nopagebreak
\isthiswhatyouowant
This research was partially supported by CMUP (UID/MAT/00144/2013) funded
by FCT (Portugal) with national funds. 
C.L.-G. is grateful to the Universit\'e de Lorraine for according a CRCT of six months.
\nopagebreak
 T.S. is grateful to the CNRS for according a delegation
 \nopagebreak
and an attachment to beautiful IMPA for one semester and equally to IMPA and its staff
\nopagebreak
including in particular Henrique Bursztyn for their hospitality during his stay. 

\newpage

\section{Main results}
\label{sec:main}
\isthiswhatyouowant
We provide the definition and examples of singular foliations in the subsequent section, see in particular Section \ref{DefExSSingFol}. 
Definitions and basic facts about {\LieInftyAlgebroid}s are given in Section \ref{infinityalgebroids}. We assume for the moment that the reader is familiar with those and we
state the main results of the article. Proofs and further examples are provided later on as well.

\isthiswhatyouowant
We intend to state results that are true in the smooth,
algebraic, real analytic, and holomorphic settings all together, sometimes with adaptations.
We let $ {\mathcal O}$ be the sheaf of (smooth, polynomial, real analytic or holomorphic---depending on the context) functions on $M$ and, for every vector bundle $E\to M$, $\Gamma(E)$  the sheaf of sections of $E$.

\begin{definition}\label{def:resolution}
Let ${\mathcal F}  \subset \Vect ( M)$ be a singular foliation on a manifold $M$. A \emph{geometric resolution $(E,{\mathrm d},\rho)$ of the singular foliation $\C{F}$} is a triple consisting of:
\begin{enumerate}
\item a collection of vector bundles $E={\bigoplus}_{i\geq1}E_{-i}$ over  $ M$,
\item a collection ${\mathrm d} = ({\mathrm d}^{(i)})_{i \geq 2}$ of vector bundle morphisms
${\mathrm d}^{(i)}\colon E_{-i} \to E_{-i+1}$ over the identity of $M$,
\item a vector bundle morphism $\rho\colon E_{-1} \to T M$ over the identity of $ M$ called the \emph{anchor of the geometric resolution},
\end{enumerate}
 such that the following sequence of ${\mathscr O}$-modules is an exact sequence of sheaves:
\begin{center}
\begin{tikzcd}[column sep=0.7cm,row sep=0.4cm] \label{eq:resolutions}
\ldots\ar[r,"\dd^{(3)}"]&\Gamma(E_{-2})\ar[r,"\dd^{(2)}"]&\Gamma(E_{-1})\ar[r,"\rho"]&\mathcal{F}\ar[r]&0
\end{tikzcd}
\end{center}
A geometric resolution is said to be \emph{of length $n$} if $E_{-i}=0$ for $i \geq n+1$.

\isthiswhatyouowant
We shall speak of a \emph{resolution by trivial bundles} when all the vector bundles $(E_{-i})_{i \geq 1}$ are trivial vector bundles. 

\isthiswhatyouowant
We shall say that a geometric resolution is \emph{minimal at a point $m \in M$} if, for all $i  \geq 2$, the linear maps $\dd^{(i)}_m \colon \left. E_{-i} \right|_m \to \left. E_{-i+1} \right|_m$ 
vanish. 
\end{definition}

\isthiswhatyouowant Note that the complex of a geometric resolution  $ (E,\dd,\rho)$, 
\begin{center}
\begin{tikzcd}[column sep=0.7cm,row sep=0.4cm] \label{eq:resolutionsprime}
\ldots\ar[r,"\dd^{(3)}"]&E_{-2}\ar[r,"\dd^{(2)}"]&E_{-1}\ar[r,"\rho"]&TM
\end{tikzcd}
\end{center}
when restricted to a point $m \in M$ is, in general, not exact. Its cohomology $H^\bullet({\mathcal F}, m)$ does not depend on the chosen resolution, see Lemma \ref{lem:canonical} below. A geometric resolution is minimal at $m \in M$, if, by definition, $ \left. E_{-i} \right|_m = H^{-i}({\mathcal F}, m)$ for all $i> 1$. In such a case, the geometric resolution is a minimal model around this point.

\isthiswhatyouowant
  Since sections of vector bundles over $M$ are projective $\Funct$-modules, geometric resolutions of  singular foliations 
  are projective  resolutions of $\mathcal{F}$ in the category of $\Funct$-modules.
  It is a classical result that such  resolutions always exist. But the projective modules of a projective resolution may not correspond to vector bundles---they may not be locally finitely generated. 
  By the Serre-Swan theorem \cite{Serre}, however:
  
  \begin{lemme}
  \label{lem:serreSwan}
 For smooth compact manifolds, geometric resolutions of a singular foliation $ \mathcal{F}$ are in one-to-one correspondence with resolutions of $\mathcal{F}$ by 
 locally finitely generated  projective ${\mathscr O}$-modules.
 
 \isthiswhatyouowant
 In the smooth, holomorphic, algebraic, or real analytic cases, 
 geometric resolutions by trivial vector bundles are in one-to-one correspondence with geometric resolutions by free finitely generated ${\mathscr O}$-modules.
  \end{lemme}
  
  \isthiswhatyouowant
  There are several contexts in which such geometric resolutions exist, at least locally, and are of finite length. We collect such statements in the following two propositions, which we are going to prove in Section \ref{existence}. 

\begin{proposition}
\label{bonjourprop} The following items hold:
\begin{enumerate}
\item In the neighborhood of every point, a holomorphic singular foliation on a complex  manifold $M$ of (complex) dimension $n$ admits a geometric resolution by trivial vector bundles 
whose length is less or equal to $n+1$, i.e.~$E_{-i}=0$ for all $i \geq n+2$. The same statement holds true for a real analytic singular foliation $\mathcal{F}$ on a real analytic manifold of dimension $n$.
\item A real analytic geometric resolution of a real analytic singular foliation $\mathcal{F}$ is also a smooth geometric resolution of $\mathcal{F}$ when regarded as a smooth singular foliation.
\item Every algebraic singular foliation on a Zarisky open subset of ${\mathbb C}^n$ admits a geometric resolution by trivial vector bundles whose length is less or equal to $n+1$.
\item There exist  smooth singular foliations on ${\mathbb R}$ that do not admit geometric resolutions (cf., e.g., Example \ref{tu} below).
\item If a geometric resolution of finite length exists, then for every point $m\in M$ there is a geometric resolution of finite length in a neighborhood of this point which is minimal at $m$.
\end{enumerate}
\end{proposition}
\isthiswhatyouowant
In Definition \ref{Def:locrealanal} below, we consider \emph{locally real analytic singular foliations}, which are smooth foliations with the property to admit local charts such that the generators are real analytic, while these charts are patched together by smooth transition functions only. In this setting one can show:
\begin{theoreme} \label{newprop} A locally real analytic singular foliation admits a geometric resolution of length at most $\dim M+1$ over any relatively compact open subset of $M$.
\end{theoreme}

\isthiswhatyouowant
 In particular, every locally real analytic singular foliation on a compact manifold $M$ has a  geometric resolution of finite length. 

\isthiswhatyouowant We are thankful to Marco Zambon for leading us to the following result, which is useful to know in the context of singular foliations with geometric resolutions, since it arrives at a globally valid fact about the leaves:

\begin{proposition} \label{prop:Marco}
 If a singular foliation ${\mathcal F} $ on a connected manifold $M$ admits geometric resolutions of finite length in a neighborhood of all points in $M$, then all its regular leaves
 have the same dimension $r$. Moreover, for every geometric resolution of finite length  $(E,{\rm d},\rho)$ of $ {\mathcal F}$ over an open subset of $M$:
    $$r = \sum_{i \geq 1 }  (-1)^{i-1} {\rm rk} (E_{-i}).$$
Here ${\rm rk}$ stands for the rank of a vector bundle.
\end{proposition}

\isthiswhatyouowant We again refer to Section \ref{existence} for a proof of this proposition, the first part of which is obvious in the real analytic or holomorphic case.

\isthiswhatyouowant We now introduce the main object and the two main theorems of the present article. 
Here we assume that the reader is familiar with {\LieInftyAlgebroid}s, which we denote for short by  $(E,Q)$ in this paper, and with the notion of linear part of those:  
all these notions are explained in detail in Section \ref{infinityalgebroids}.

\begin{definition}
Let $\mathcal{F}$ be a singular foliation on a manifold $M$. We call a {\LieInftyAlgebroid} $(E,Q)$ over $M$  a \emph{universal
{\LieInftyAlgebroid} \resolving $\mathcal{F}$}, if  the linear part of $(E,Q)$ is a geometric resolution of  $ \mathcal{F}$.\footnote{In particular, 
$ \mathcal{F}$ is the foliation associated to $(E,Q)$, i.e.~$\rho\big(\Gamma(E_{-1})\big) = \mathcal{F}$, with
$\rho$ the anchor map of $(E,Q)$.} 
If $E_{-k}=0$ for all $k \geq n+1$, we speak of a \emph{universal Lie $n$-algebroid \resolving $\mathcal{F}$}. 
\end{definition}

\isthiswhatyouowant Here is the first main result:

\begin{theoreme}
\label{theo:existe}
Let $\mathcal{F}$ be a singular foliation on a smooth manifold $M$ which permits a geometric resolution  $(E,{\mathrm d},\rho)$. Then there is a 
universal {\LieInftyAlgebroid} \resolving $\mathcal{F}$ the linear part of which is the given geometric resolution. 

\isthiswhatyouowant For a real analytic or holomorphic singular foliation there is a universal Lie $k$-algebroid in the neighborhood of every point $m \in M$ with $k \leq n+1$, where 
$n$ is the real and complex dimension of $M$, respectively.
For a locally real analytic singular foliation\footnote{see Definition \ref{Def:locrealanal}}, there is a universal Lie $k$-algebroid over every relatively compact open subset with $k \leq n+1$, where 
$n$ is the real dimension of $M$.
\end{theoreme}

\isthiswhatyouowant We note as a side remark that the existence of a geometric resolution is not sufficient for a global result in the holomorphic or real analytic case, 
since in the corresponding construction one uses a partition of unity, which is not available in those cases. 
The use of the word ''universal' is justified retrospectively by the following result:

\begin{theoreme}\label{theo:onlyOne}
Let $(E,Q)$ be  a universal {\LieInftyAlgebroid}  \resolving a singular foliation $ \mathcal{F}$ on a smooth manifold. 
For every {\LieInftyAlgebroid} $ (E',Q')$ defining a sub-singular foliation of $ \mathcal{F}$ (i.e.~such that  $\rho'\big(\Gamma(E_{-1}')\big) \subset \mathcal{F}$),
there is a {\LieInftyAlgebroid} morphism from $ (E',Q')$ to $ (E,Q)$ over the identity of $M$ and any two such {\LieInftyAlgebroid} morphisms are homotopic.

\isthiswhatyouowant The same result holds in the holomorphic and the real-analytic setting in a neighborhood of an arbitrary point.
\end{theoreme}

\isthiswhatyouowant
This implies in particular that in the category where objects are {\LieInftyAlgebroid}s whose induced singular foliation is included in $ {\mathcal F}$ 
and where arrows are homotopy classes of morphisms, every universal {\LieInftyAlgebroid} over $ {\mathcal F}$ is a terminal object.
Terminal objects always satisfy what is called a "universality property". An immediate consequence of the theorem is the following uniqueness result: 
\begin{corollaire}
	\label{coro:unique}
Two  universal {\LieInftyAlgebroid}s  \resolving the singular foliation $ \mathcal{F}$ are homotopy equivalent and two such homotopy equivalences are homotopic.
\end{corollaire}
\isthiswhatyouowant
It is a general fact that two terminal objects are related by a unique invertible arrow (which is the case in the category just defined).
Of course, in most well-known cases, e.g., the universal enveloping algebra, this invertible unique arrow is bijective, while here it
is only a homotopy class of invertible-up-to-homotopy morphisms.

\isthiswhatyouowant
Another consequence of Theorem \ref{theo:onlyOne} is the following one: suppose there is  an (ordinary) Lie algebroid $A$ which defines a singular foliation $\mathcal{F} $ and that this foliations admits a geometric resolution. Then there exists 
a {\LieInftyAlgebroid} morphism from $A$ to any universal {\LieInftyAlgebroid} $ (E,Q)$ \resolving $\mathcal{F}$ and any two such morphisms are homotopic. 
This illustrates once more why the  universal {\LieInftyAlgebroid} $ (E,Q)$ \resolving $\mathcal{F}$ is 
more important than a Lie algebroid that possibly can be used to define the same foliation $\mathcal{F}$; even if it exists, it is in particular far from unique.

\isthiswhatyouowant
Let us say a few words about the proofs of the previous results. 
A crucial result is Proposition \ref{prop:fondamental3}, which states that vertical vector fields on a geometric resolution $E$, seen as a graded manifold, 
have cohomology concentrated in a low number of degrees only.
The proofs are mainly based on step-by-step constructions using this Lemma.
For clarity, we have dedicated different subsections to the proofs of different results: Theorem \ref{theo:existe} is proven in Section \ref{sectionpreuve1},
Theorem \ref{theo:onlyOne} is proven in Section~\ref{sec:context}. 

\isthiswhatyouowant
As mentioned in the Introduction, several invariants and geometric properties of the singular foliation can be derived
from these two theorems: this is the subject of  Section~\ref{sec:geometry}. 
In a generalization of the known isotropy Lie algebras of a given point $m$ on the base $M$ of a Lie algebroid, every Lie $\infty$-algebroid gives rise to a canonical isotropy Lie $\infty$-algebra at $m$. Applying this functor, defined in more detail in Section \ref{sec:functor}, to a universal Lie $\infty$-algebroid of a singular foliation,  
we prove the following result.

\begin{theoreme} \label{theo:lasttheorem} Let  ${\mathcal F}$ be a singular foliation on $M$ admitting a geometric resolution in the neighborhood of  $m \in M$. 
The above construction equips $H^\bullet({\mathcal F},m)$ with a canonical 
Lie $\infty$-algebra structure. Its 1-bracket vanishes and its 2-bracket, restricted to degree minus one, reproduces the isotropy Lie algebra of \cites{AndrouSkandal, AndrouZambis}. For two points in the same leaf, the Lie $\infty$-algebras are isomorphic. 
\end{theoreme}

\isthiswhatyouowant
We refer to this Lie $\infty$-algebra structure as the \emph{isotropy Lie $\infty$-algebra of $ {\mathcal F}$} at $m$.
In Proposition	\ref{propdef:NMRLA} below we show that the restriction  of the $3$-ary bracket  to $\g= H^{-1} ({\mathcal F},m)$,
$$ \{\cdot, \cdot, \cdot \}_3 \colon \Lambda^3  H^{-1} ({\mathcal F},m)  \longrightarrow H^{-2} ({\mathcal F},m),$$
 is a $3$-cocycle for the Chevalley-Eilenberg complex of the isotropy Lie algebra $\g$ valued in  the $\g$-module $H^{-2} ({\mathcal F},m)$. 
Its class does not depend on the choices made within this construction and thus is an invariant associated directly to the foliation $ {\mathcal F}$ and the (singular) point $m\in M$. This is of relevance also because according to Proposition \ref{prop:NMRLA} below  there cannot be a Lie algebroid of minimal rank $r$  defining $ {\mathcal F}$ in a neighborhood of $m$ if this Chevalley-Eilenberg cohomology 3-class is non-vanishing ($r$ is the rank of $ {\mathcal F} $ as an ${\mathcal O}$-module).  
We also provide an example where this class is in fact non-vanishing: for the origin of the singular foliation $ {\mathcal F} $ of vector fields on ${\mathbb C}^4$ tangent to the  level sets of the function $ (z_1)^3+(z_2)^3+(z_3)^3+(z_4)^3$, see Example \ref{ex:NMLRA}.

\isthiswhatyouowant
 Proposition \ref{prop:recover_AS} below, finally, states that the (appropriately defined) fundamental groupoid of the universal Lie $\infty$-algebroid of a singular foliation is the universal cover of the holonomy groupoid described by Androulidakis and Skandalis in \cite{AndrouSkandal}

\begin{remarque}
\normalfont
Although presented here for singular foliations only, most results of this paper, in particular 
Theorems \ref{theo:existe} and \ref{theo:onlyOne}, can be adapted to every locally finitely generated sheaf of Lie-Rinehart algebras 
over the ring of functions on a manifold $M$.
\end{remarque}

\newpage

\section{The universal Lie \texorpdfstring{$\infty$}{infinity}-algebroid of a singular foliation: existence and uniqueness }
\label{sec:construction}

\subsection{Singular foliations: definitions and examples}
\label{DefExSSingFol}

\isthiswhatyouowant
Let $M$ be a manifold that may be smooth, real analytic, or complex. 
It may also be a Zarisky open subset  $U \subset {\mathbb C}^n$.
Generalizing to affine or projective varieties would be an interesting topic by itself.

\isthiswhatyouowant
Let  ${ U}$ be an open subset of $M$. Denote by $ {\mathscr O}( U)$ the algebra of polynomials, smooth, 
real analytic, or holomorphic functions over $ U$, depending on the respective context, and by ${\mathfrak X}( U) $ the ${\mathscr O}( U) $-module
of vector fields.
The assignment ${\mathfrak X}\colon { U} \mapsto  {\mathfrak X}( U)$ is a sheaf of Lie-Rinehart algebras, i.e.~a sheaf of Lie algebras 
and a sheaf of $ {\mathscr O}$-modules, and both are compatible~\cite{huebschmann}.

\isthiswhatyouowant
We say that a sheaf $\Gamma\colon U\mapsto \Gamma(U)$ is \emph{locally finitely generated}, if for every $m\in M$ there exists an open neighborhood $U_m$ of $m$
and a finite number  of sections $X_1,\ldots,X_p\in\Gamma(U_m)$ such that for every  open subset $V  \subset U_m$ the vector fields
$X_1|_V,\ldots,X_p|_V$ span $\Gamma(V)$.
The minimal number of local generators of a finitely generated sheaf at a given point $m \in M$ is called its \emph{rank at $m$}.  

\isthiswhatyouowant
We define singular foliations in the smooth, complex, real analytic, and algebraic context as follows:

\begin{definition}
\label{def:sing_fol}
A \emph{singular foliation} is a subsheaf ${\mathcal F}\colon {U} \mapsto  {\mathcal F}(U)$ of
the sheaf of vector fields ${\mathfrak X}$, which is locally finitely generated as an ${\mathscr O}$-submodule and closed with respect to the Lie bracket of vector fields.
\end{definition}

\isthiswhatyouowant Note that, in this way, regular foliations are considered as particular singular ones. There is a type of smooth singular foliation which is of particular interest: 

\begin{definition} \label{Def:locrealanal} A locally real analytic singular foliation is a smooth singular foliation over a smooth manifold $M$ which admits, around each point, generators which are real analytic in some local chart.\footnote{We do not require these local charts to be compatible and to turn  $M$ into a real analytic manifold.}
\end{definition}

\isthiswhatyouowant
Several authors \cites{Debord,AndrouSkandal,AndrouZamb} prefer to consider compactly supported vector fields. 
As pointed out by Alfonso Garmendia in \cite{GarmendiaThesis}, or in Remark 2.1.3 in \cite{Roy}, this makes no difference:
For a smooth manifold $M$,  subsheaves of the sheaf of vector fields which are locally finitely generated 
are in one-to-one correspondence with sub-modules of the module of compactly supported vector fields on $M$
whose restriction to an open subset is locally generated.

\isthiswhatyouowant
A \emph{singular sub-foliation} ${\mathcal F}'$ of a singular foliation ${\mathcal F}$ is a singular foliation such that ${\mathcal F}'( U) \subset {\mathcal F}( U)$ for 
all open subsets ${ U} \subset M$. 
A singular foliation on a manifold $M$ will be said to be \emph{finitely generated} if there exist $k$ vector fields $X_1, \dots, X_k \in {\mathfrak X}(M)$
such that for every open subset ${U} \subset M$ the restriction of $X_1, \dots, X_k$  to $U$ generates 
${\mathcal F}(U)$  over $ {\mathscr O}(U)$.
 
 \isthiswhatyouowant
An \emph{anchored bundle} is a (smooth, real analytic, or holomorphic) vector bundle $A$ over $M$ together with a vector bundle morphism $\rho\colon A \to TM$ over the identity of $M$.
We say that it covers a singular foliation $ {\mathcal F}$, if
every point $m\in M$ admits a neighborhood $ U$ such that $ {\mathcal F}( U) =  \rho\big(\Gamma_{ U}(A)\big) $.\footnote{In terms of sheaves, this condition means 
that the sheaf ${\mathcal F}$ is obtained by sheafifying the presheaf $ \rho\big(\Gamma(A)\big)$.}

\isthiswhatyouowant
We call \emph{leaves} of a singular foliation ${\mathcal F}$ the connected submanifolds $N$ of $M$ whose tangent space is, at every point $m \in N$,
obtained by evaluating at $m$ all the local sections of the sheaf ${\mathcal F}$, and which is maximal
among such submanifolds. The following result, now classical, is due to R.~Hermann:

\begin{proposition}
\label{prop:Hermann}
\cite{Hermann1962}
A singular foliation ${\mathcal F}$ on a manifold $M$ induces a partition of $M$ into leaves. 
\end{proposition}

\isthiswhatyouowant
 Let $r_m$ be the dimension of the subspace of $T_m M$ obtained by evaluating
all the vector fields of a singular foliation~$ {\mathcal F}$  at $m\in M$.
We say that a point $m \in M$ is \emph{regular} if $r_m$ is constant in a neighborhood of $ m$
and \emph{singular} otherwise. 
Since the function $x \to r_x$ is lower-semi-continuous, regular points form an open dense subset of $M$.
On a connected complex or real analytic manifold, regular points are those for which the function $m \mapsto r_m$
reaches its maximum. This may not be true on non-connected manifolds or on smooth connected manifolds:
For instance, let us choose a function $\chi \colon  {\mathbb R} \to {\mathbb R}$ which vanishes on $ {\mathbb R}_-$
and which is strictly positive on  $ {\mathbb R}_+$, then for the singular foliation on ${\mathbb R}$ generated by $\chi(x) \frac{d}{d x} $,
all points $x \in \mR$ are regular except for $\{0\}$, but $r_x =0$ for $x<0$ and $r_x=1$ for $x>1$.

\isthiswhatyouowant
A non-trivial statement is that a leaf contains a singular point if and only if all its points are singular,
so that it makes sense to define \emph{singular leaves} as being those made of singular points
and likewise \emph{regular leaves} as being those made of regular points.\footnote{Of course, alternative definitions of singular or regular leaves could be given.
Consider the singular foliation on ${\mathbb C}^3$ of all vector fields $X$ such that $X[\varphi]=0$,
with $\varphi(x,y,z)=x^2-y^2-z^2$. In algebraic geometry, it would be natural to say that the subset $\{\varphi(x,y,z)=0\}$ is a singular leaf. 
With our definition, however, it is the union of the  regular leaf $\{ x^2-y^2-z^2=0 \} \cap \{ (x,y,z) \neq (0,0,0)\} $ with the singular leaf $\{(0,0,0)\} $.}

\nocite{Debord2}
\nocite{AndrouSkandal3}

\begin{remarque}
\normalfont
Unlike the case of regular foliations,  singular foliations are not characterized by their leaves, and two different singular foliations may have the same leaves but differ as sheaves 
of vector fields. For instance, as noticed in \cite{AndrouZamb}, for $M$ a real or complex vector space,  consider ${\mathcal F}_k$ to be the module of all smooth,
real analytic, and holomorphic vector fields, respectively,  vanishing to some fixed order $k  \geq 1$ at the origin. This is clearly a singular foliation for all $k$, and all such 
singular foliations have exactly the  same two leaves: the origin and the complement of the origin. They are not, however,  identical as sub-modules of the module of vector fields. 
\end{remarque}

\isthiswhatyouowant
Here is a first class of singular foliations:

\begin{example}\label{Ex:AlgebroidsAreExamples}
\normalfont
For $A $ a (smooth or holomorphic \cite{LSX2}) Lie algebroid over ${ M}$ with anchor $\rho \colon A \to T M$, the ${\mathscr O}$-module $\rho\big(\Gamma(A)\big)
$ is a singular foliation. It is a finitely generated foliation, because there always exists a  vector bundle $B$ such 
that the direct sum $A \oplus B$ is trivial. 
\end{example}

\isthiswhatyouowant
This class of examples includes regular foliations, orbits of a connected Lie group action, orbits of a Lie algebra or a Lie algebroid action, 
symplectic leaves of a Poisson manifold, and foliations induced by Dirac structures. 

\isthiswhatyouowant
Singular foliations as close as possible to regular ones deserve their own name:
\begin{definition}\cite{Debord}
	\label{def:debord}
	 A singular foliation $ {\mathcal F}$ is said to be a \emph{Debord foliation} if it is the image of a Lie algebroid $(A,[\cdot, \cdot],\rho) $ through an anchor map which is injective on a dense open subset.
\end{definition}
\isthiswhatyouowant 
In the smooth case, a singular foliation $ {\mathcal F}$ is Debord if and only if it is projective
as an $ {\mathcal O}$-module. 

\isthiswhatyouowant
Example \ref{Ex:AlgebroidsAreExamples} can be enlarged by a notion more general than the one of a  Lie algebroid:

\begin{definition} \cites{huebschmann}
An \emph{almost-Lie algebroid} over ${ M}$ is a vector bundle $A \to  M$, equipped with a vector bundle morphism $\rho \colon A \to T M$ called the \emph{anchor map}, and a skew-symmetric bracket $[\,.\, ,.\, ]_A $ on $\Gamma (A) $, 
satisfying the \emph{Leibniz identity},
\begin{align}\label{algebroid}
\forall\ x,y\in\Gamma(A),\ f\in\cinf({ M})\hspace{1cm}&[x,fy]_{A}=f[x,y]_{A}+\rho(x)[f] \, y,
\end{align}
together with the \emph{algebra homomorphism condition}:
\begin{equation}\label{algebroid2}
\forall\  x,y\in\Gamma(A)\hspace{1cm}\rho\big([x,y]_{A}\big)=\big[\rho(x),\rho(y)\big].
\end{equation}
\end{definition}

\isthiswhatyouowant
We do not require the bracket $[\cdot,\cdot ]_A $ to be a Lie bracket: It may not satisfy the Jacobi identity.
However, the Jacobi identity being satisfied for vector fields, Condition (\ref{algebroid2}) imposes that the Jacobiator takes values in 
the kernel of the anchor map at all points. 
The following result appeared in Proposition 2.1.4 of \cite{thesis} and in Proposition 3.17 in \cite{meinrenken}.
We include a proof.

\begin{proposition} 
\label{prop:Almost}
Let $M$ be a smooth, real analytic, or complex manifold and $(A,\rho)$ an anchored vector bundle.
\begin{enumerate}
\item For every almost-Lie algebroid structure on $A \to  M$, the image of the anchor map $\rho : \Gamma(A) \to {\mathfrak X}({ M})$ is a singular foliation.
 \item Every finitely generated foliation on $M$ is the image under the anchor map of an almost-Lie algebroid, defined on a trivial bundle.
 \item In the smooth case, every anchored vector bundle $(A,\rho)$ over $M$  that covers a singular foliation ${\mathcal F}$ can be equipped with an almost-Lie algebroid structure
 with anchor $\rho$.
 \item  In the smooth case, a singular foliation is the image under the anchor map of an almost-Lie algebroid if and only if it is finitely generated.
 \end{enumerate}
\end{proposition}
\begin{proof}
The first item follows directly from \eqref{algebroid} and \eqref{algebroid2}. 
Let us prove the second  item. Let $X_1,\dots,X_r $ be generators of a singular foliation ${\mathcal F} $. Since ${\mathcal F}$ is closed under the Lie bracket of vector fields, there exist functions $c^k_{ij} \in {\mathscr O}(M)$ satisfying:
 \begin{equation} [X_i,X_j] = \sum_{k=1}^r \ c^k_{ij}\, X_k, \end{equation}
for all indices $i,j \in \{1, \dots,r\}$. Upon replacing $ c^k_{ij}$ by
$  \frac{1}{2}(c^k_{ij}-c^k_{ji})$ if necessary, we can assume that the functions $c^k_{ij} \in {\mathscr O}(M)$ satisfy the skew-symmetry 
relations $ c^k_{ij}= -c^k_{ji}$ for all possible indices. Define $A$ to be the trivial bundle $A = {\mathbb R}^r \times { M} \to { M}$. Denote its canonical global sections by $e_1, \dots,e_r$ and define:
\begin{enumerate}
\item an anchor map by $\rho(e_i) = X_i$ for all $i=1, \dots,r$,
\item a skew-symmetric bracket by $[e_i,e_j]_{A} = \sum_{k=1}^r c^k_{ij} e_k$ for all $i,j=1, \dots,r$.
\end{enumerate}
One then extends these definitions to all sections by ${\mathscr O}$-linearity and the Leibniz property, respectively. By construction, this defines an almost-Lie algebroid structure on $A$ 
such that~$\rho(\Gamma(A))={\mathcal F}$.

\isthiswhatyouowant
Let us now prove the third item.
Unlike Lie algebroid brackets, 
almost-Lie algebroid brackets can be glued  using partitions of unity. 
More precisely, let $(\varphi_i)_{i \in I}$ be a partition of unity subordinate to an open cover $({ U}_i)_{i \in I} $ by open sets
trivializing the vector bundle $A$. By the proof of item 2, we can define an almost-Lie algebroid structure with anchor $\rho$ on the restriction of $A$ to $U_i$,
that is, a bracket $[\cdot,\cdot ]_{ U_{i}} $ that satisfies Equations (\ref{algebroid}) and (\ref{algebroid2}) 
for all sections in $\Gamma_{ U_{i}}(A)$. 
The bracket
\begin{equation}
[\cdot,\cdot ]_{A}=\sum_{i \in I}\ \varphi_i\, [\cdot ,\cdot]_{ U_{i}}
\end{equation}
still satisfies Equations \eqref{algebroid} and \eqref{algebroid2} and hence defines an almost-Lie algebroid structure on $A$ with anchor~$\rho$.
The last item follows from the third one.
\end{proof}

\isthiswhatyouowant
 It has been conjectured \cite{AndrouZambis} that not every smooth singular foliation is of the type described in Example~\ref{Ex:AlgebroidsAreExamples}, not even only in a neighborhood of a given point.
As far as we know, the question remains open to this day.\footnote{Not all singular foliations are, globally, the image through 
 the anchor map of a Lie algebroid.  Here is a counter-example, given in \cite{AndrouZambis}: the singular foliation of all 
 vector fields on ${\mathbb R}^2$ vanishing to order $k$ at the point of coordinates $(0,k)$ is locally finitely generated by Example \ref{ex:orderk}, but its rank is not bounded.
 Hence it can not be the image through the anchor map of a Lie algebroid.}
 There are quite a few singular foliations for which the underlying Lie algebroid structure, if it exists at all, is at least not easy to find, for instance those described in Examples \ref{ex:symmetries}, \ref{ex:orderk}, \ref{ex:bivectorfield}, and \ref{ex:Leibnizoids} below.
 
\begin{example}\label{ex:symmetries}
\normalfont 
Let $\mathbb{K}=\mathbb{R}$ or $\mathbb{C}$:
 \begin{enumerate}
 \item Let $P:=(P_1,\dots,P_k)$ be a $k$-tuple of polynomial functions in $d$ variables over $ {\mathbb K}$. The symmetries of $ P$,  i.e.~all polynomial vector fields $X \in {\mathfrak X}({\mathbb K}^d)$ that satisfy $X[P_i] =0$ for all $i \in \{1, \dots,k\}$ form a singular foliation. 
 The special case of $k=1$  appears in toy models for the  Batalin-Vilkovisky formalism, where one may consider the symmetries of a polynomial function $S$ representing the classical action  \cite{Felder}.
 \item \emph{Symmetries of some affine variety $W \subset {\mathbb K}^d$}, i.e.~all polynomial vector fields $X$ such that $X[I_W] \in I_W $, where
 $I_W$ is the ideal of polynomial functions vanishing on $W$.
 \item Vector fields on ${\mathbb C}^d$ vanishing at all points of an affine variety  $W \subset {\mathbb C}^d$.
\end{enumerate} 
  All the previous spaces of polynomial vector fields are closed under the Lie bracket and form a sub-module of
  the module of algebraic vector fields over the ring of polynomial functions on ${\mathbb K}^d$. Since the latter is finitely generated
  and since the ring of polynomial functions is Noetherian, each of these spaces is a finitely generated module over the polynomial functions. The ${\mathscr O}({\mathbb K}^d)$-module
  generated by these polynomial vector fields 
  is therefore also a singular foliation; here ${\mathscr O}$ stands again for smooth, real analytic, or holomorphic functions. 
   \end{example}

\isthiswhatyouowant We now provide further examples of  singular foliations.

 \begin{example}
  \label{ex:sousvariete}
 \normalfont
Vector fields on a manifold $M$ which are tangent to a submanifold $L$. 
Of course, $L$ is the only singular leaf of this singular foliation, while the connected components of  $M\backslash L$
 are the regular ones.
 \end{example}

\begin{example}\label{ex:orderk}
\normalfont
For every $k,n \in \mN$,
vector fields vanishing to order $k$ at the origin of ${\mathbb R}^n$ form a singular foliation.
For $k=1$, it is the singular foliation associated to the action of the group ${\rm GL}(n)$. In particular, it is the image through the anchor map of a Lie algebroid.
For $ k \geq 2$ and $ n \geq 2$, however, it is not known if it can be realized as the image through the anchor map of a Lie algebroid.
\end{example}

\begin{example}\label{ex:bivectorfield}
\normalfont
A bivector field $\pi \in \Gamma(\wedge^2 T M)$ on a manifold $M$ is said to be \emph{foliated}  \cite{Turki}, if $\pi^\#(\Omega^1(M)) $
is closed under the Lie bracket. In that case, $\pi^\#(\Omega^1(M)) $ is a singular foliation. When $\pi$ is Poisson bivector or  twisted Poisson  \cites{Park,Klimcik,Severaweinstein}), it
is known that $T^* M$ comes equipped with a Lie algebroid structure \cite{costeweinstein} with anchor $\pi^\#:T^*  M \to T  M$, but for `generic' 
foliated bivector fields no such formula  exists, cf.~\cite{Turki}.
\end{example}

\begin{example}\label{ex:Leibnizoids}
\normalfont
For a Leibniz algebroid, cf.~Section \ref{sec:leibnizoid} below or \cite{Kotov}
 for a definition, the image of the anchor map is  a singular foliation as well.  
Courant algebroids \cite{LiuXu} and vector bundle twisted Courant algebroids \cite{Melchior} are particular examples of those. 
Another example of a Leibniz algebroid, now defined on $\wedge^2 TM$, arises from any function $S$ on $M$ in the following way:  the anchor is defined by means of 
$P \mapsto P_S := P^\# ({\diff } S)$ and the bracket between two bivector fields $P$ and $Q$ by means of 
 $$  (P,Q) \mapsto {\mathcal L}_{P_S} Q . $$
Note 
 that for $M$ a vector space and $S$ a polynomial function, the associated singular foliation is a sub-foliation of 
 the foliation of symmetries of $S$ described in Example \ref{ex:symmetries} (or also \ref{ex:Koszul} below). 
\end{example}

\isthiswhatyouowant 
Now we give examples of a sub-sheaf of the sheaf of vector fields which is closed under the Lie bracket, but which is \emph{not} a singular foliation.
\begin{example}
\normalfont
On $M = {\mathbb R}$, smooth vector fields vanishing on ${\mathbb R}_-$ are closed under the Lie bracket but are not locally finitely generated  \cite{Texas},
hence they do not form a  singular foliation in our sense---while they still generate leaves in the obvious way. 

\isthiswhatyouowant  Note as an aside that if instead we consider the foliation generated by a single vector field
$\chi(x) \frac{ \partial}{\partial x} $ for some fixed chosen smooth function $\chi$ vanishing for not strictly positive values of
 $x\in \mR$ and being non-zero otherwise, this module is generated by only one vector field and thus defines a singular foliation, the leaves of which coincide precisely with those from above. 
 \end{example}
 
\begin{example}
\normalfont
Consider $M={\mathbb R}^2$ with variables $(x,y)$ and the $\mathcal{C}^{\infty}(M)$-module generated by
the vector field $ \frac{\partial }{\partial x}$ and vector fields of the form $ \varphi \frac{\partial }{\partial y}$
where, similarly to the function $\chi$ above, $ \varphi$ a smooth function vanishing on the half-plane $x \leq 0$ and being non-zero for $x>0$. This module is closed
under the Lie bracket, but it is not locally finitely generated. This counter-example is interesting also due to the non-existence of a good notion of leaves: evidently the flow of these vector fields allows to connect any two points of $M={\mathbb R}^2$, whereas the evaluation of the module gives all of $T_{(x,y)}M$ for $x>0$, but only a one-dimensional sub-bundle for $x\leq 0$.
\end{example}


\subsection{Existence of geometric resolutions of a singular foliation}
\label{existence}

\isthiswhatyouowant
This section is devoted to the proof of several results of Section \ref{sec:main} related to the existence and the properties of geometric resolutions of singular foliations. We start with some standard material needed henceforth.

\begin{definition}
A \emph{complex of vector bundles $(E,\dd,\rho)$ over a singular foliation $ {\mathcal F}$} is a collection $E$ of vector bundles $(E_{-i})_{i \geq 1}$ over  $ M$,
 a collection $\dd$ of vector bundle morphisms
${\mathrm d}^{(i)} \colon E_{-i} \to E_{-i+1}$,  and a vector bundle morphism $\rho \colon E_{-1} \to TM$  such that ${\mathrm d}^{(i-1)} \circ {\mathrm d}^{(i)}=0 $
for all $i \geq 3$, $ \rho \circ \dd^{(2)}=0$ and $\rho(\Gamma(E_{-1})) \subset {\mathcal F} $. 
\end{definition}
\isthiswhatyouowant
A geometric resolution  $(E,\dd,\rho)$ of a singular foliation $ {\mathcal F}$, cf.\ Definition \ref{def:resolution}, is an example of such a complex of vector bundles over  $ {\mathcal F}$. In fact, every complex of vector bundles over $ {\mathcal F}$ is a geometric resolution of this foliation, if and only if it is exact on the level of sections and satisfies $\rho(\Gamma(E_{-1})) = {\mathcal F} $. 

\begin{definition}
\isthiswhatyouowant
\begin{itemize}
\item A \emph{morphism} $\phi$ between two complexes of vector bundles $ (E,\dd,\rho)$ and  $(E',\dd',\rho')$ over  $ {\mathcal F}$ is a collection of vector bundle morphisms $\phi_i  \colon  E_{-i} \to E_{-i}'$ over the identity of $M$ making the following diagram commutative
\begin{center}
	\begin{tikzcd}[column sep=0.9cm,row sep=0.6cm]
		\dots  \ar[r,"\dd^{(4)}"] &E_{-3}\ar[d,"\phi_{3}"] \ar[r,"\dd^{(3)}"]&E_{-2}\ar[d,"\phi_{2}"] \ar[r,"\dd^{(2)}"] &E_{-1}\ar[d,"\phi_{1}"] \ar[r,"\rho"] &TM\ar[d,"\mathrm{id}"]\\
		\dots  \ar[r,"\dd'^{(4)}"]  &E_{-3}' \ar[r,"\dd'^{(3)}"]&E_{-2}' \ar[r,"\dd'^{(2)}"]  &E_{-1}' \ar[r,"\rho' "] &TM .
	\end{tikzcd}
\end{center}

\item Two  morphisms $\phi,\psi\colon (E,\dd,\rho)\to (E',\dd',\rho')$ are said to be \emph{homotopic}, if there exists a collection $h_i  \colon  E_{-i}\to E_{-i-1}'$ of vector bundle morphisms
such that $ \phi_i = \psi_i + \dd'^{(i+1)} \circ h_i + h_{i-1} \circ \dd^{(i)} $ for all $i \geq 2$ and $ \phi_1 = \psi_1 + \dd'^{(2)} \circ h_1 $.
\begin{center}
	\begin{tikzcd}[column sep=0.9cm,row sep=0.6cm]
		\dots  \ar[r,"\dd^{(4)}"] &E_{-3}  \ar[ddl,dashed,"h_3"] \ar[dd,shift left =0.5ex,"\phi_{3}"] \ar[dd,shift right =0.5ex,"\psi_{3}" left] \ar[r, "\dd^{(3)}"]&E_{-2}\ar[ddl,dashed,"h_2"]  \ar[dd,shift left =0.5ex,"\phi_{2}"] \ar[dd,shift right =0.5ex,"\psi_{2}" left] \ar[r,"\dd^{(2)}"] &E_{-1}\ar[ddl,dashed,"h_1"]  \ar[dd,shift left =0.5ex,"\phi_{1}"] \ar[dd,shift right =0.5ex,"\psi_{1}" left] \ar[r,"\rho"] &TM\ar[dd,"\mathrm{id}"]\\
		& & & &\\
		\dots  \ar[r,"\dd'^{(4)}"]  &E_{-3}' \ar[r,"\dd'^{(3)}"]&E_{-2}' \ar[r,"\dd'^{(2)}"]  &E_{-1}' \ar[r,"\rho' "] &TM .
	\end{tikzcd}
\end{center}
\item Two complexes of vector bundles $(E,\dd,\rho)$ and $(E',\dd',\rho')$  over $ {\mathcal F}$ are said to be \emph{homotopy equivalent}, if there exist chain maps $\phi \colon (E,\dd,\rho) \to (E',\dd',\rho') $  and $\psi \colon (E',\dd',\rho') \to (E,\dd,\rho)$ such that both $\phi \circ \psi$ and $\psi \circ \phi$
are homotopic to the identity.
\end{itemize}
\end{definition}
\isthiswhatyouowant In other words, homotopy equivalence holds if the maps $\phi$ and $\psi$ satisfy  
\begin{equation} \psi_i \circ \phi_i = \left(\mathrm{id} + h \circ \dd + \dd \circ h \right)|_{E_{-i}}
\end{equation}  for all $i \ge 2$ together with 
\begin{equation}  
\psi_1 \circ \phi_1 = \left(\mathrm{id}  + \dd \circ h \right)|_{E_{-1}}
\end{equation} 
and likewise conditions for $\phi \circ \psi$. 

\begin{lemme}
\label{lem:alt_sum}
 For any two homotopy equivalent complexes of vector bundles of finite length, the alternate sum of their ranks are equal.
\end{lemme}
\begin{proof}
By restricting the sequence to a point $m \in M$, one obtains a 
 finite length complex of vector spaces of finite dimension. Here it is known that the alternate sum of the dimensions is preserved under homotopy equivalence. 
This proves the lemma.
\end{proof}

\isthiswhatyouowant
Here is a second result of importance:
\begin{lemme}
\label{lem:existsAChain}
 Let $(E,\dd,\rho)$ be a geometric resolution of a singular foliation ${\mathcal F}$. 
 For every complex of vector bundles $(E',\dd',\rho') $ over $ {\mathcal F}$, there exists a morphism
 of complexes of vector bundles  over $ {\mathcal F}$
from $(E',\dd',\rho')$ to $(E,\dd,\rho) $ and any two such chain morphisms are homotopy equivalent.
\end{lemme}
\begin{proof}
This is a standard result of algebraic topology, see Porism 2.2.7 (and the discussion before) in \cite{Weibel}.\end{proof}

\isthiswhatyouowant
Lemma \ref{lem:existsAChain} admits the following immediate consequence:

\begin{lemme}
\label{lem:existsAChain2}
 Any two geometric resolutions $ (E,\dd,\rho)$ and $ (E',\dd',\rho')$ of a singular foliation $ {\mathcal F}$ are homotopy equivalent.
\end{lemme}

\subsubsection{Proof of  Proposition \ref{bonjourprop}}
We express our gratitude to Fran\c{c}ois Petit, whose knowledge of the matter was of crucial help.
\begin{proof}[Proof (of Proposition \ref{bonjourprop})]
The \emph{first and} the \emph{third item} are simply Hilbert's syzygy theorem, which is valid for finitely generated ${\mathscr O}$-modules,
with ${\mathscr O}$ being the ring of holomorphic functions in a neighborhood of a point in ${\mathbb C}^n$
or the ring of polynomial functions on ${\mathbb C}^n$, as proven in Theorem 4 page 137 in \cite{gunningrossi} for the holomorphic case
and \cite{Eisenbud} for the algebraic case. 
Recall that these theorems state that every finitely generated ${\mathcal O}$-module, with $ {\mathcal O}$ the algebra
of holomorphic or polynomial functions on an open subset $V \in {\mathbb C}^n$, admits a resolution of length  less or equal to $n+1$ by finitely generated free 
$ {\mathcal O}$-modules. It implies that for any other resolution by free modules, the kernel of $ \dd^{(n+1)}$ is a free module.

\isthiswhatyouowant
Let us deduce the real analytic case from the holomorphic one.
Every real analytic manifold $M$ admits a complexification $M^{\mathbb C} $ such that the original manifold is the fixed point set of an anti-holomorphic involution~$ \sigma \colon 
M^{\mathbb C} \to M^{\mathbb C}$.
A real-analytic singular foliation ${\mathcal F}$ on a real analytic manifold $M$ induces a holomorphic singular foliation $ {\mathcal F}^{\mathbb C}$
on the complexification $M^{\mathbb C} $.
such that $\rho$ and  $(\dd^{(i)})_{ i \geq 2}$ are real, that is to say such that they are invariant under the natural involution of the algebra ${\mathcal O}$ 
given by $ \tau (F) = \overline{F \circ \sigma} $.
This resolution may not be of finite length.
But the holomorphic syzygy theorem implies that the kernel of $ \dd^{(n+1)}$ is a free module, with $n$ the dimension of $M$.
As a consequence, this resolution can be truncated in degree $n+1$ to yield a resolution of finite length
whose anchor and differential are real.
Real analytic functions on $V$ being fixed points of the 
natural involution $ \tau$, this resolution can be restricted to fixed points of $ \tau$
to induce a geometric resolution of the real analytic foliation $ {\mathcal F}$.
This completes the proof of the first item in the real analytic case.

\isthiswhatyouowant
Now, we turn to the \emph{second item}. According to Theorem 4 in \cite{Tougeron}, germs of smooth functions at a point $m$ are
a flat module over germs of real analytic functions at $m$. 
By definition of flatness, it means that given a complex $E_{-k-1} \to E_{-k} \to E_{-k+1}$ of vector bundles on the base manifold such that germs of real analytic sections 
have no cohomology at degree $-k$, the sheaf of germs of smooth sections has no cohomology at degree $-k$. Let us choose 
$e \in \Gamma_U(E_{-k})$ a local smooth section of $E_{-k}$, defined on an open subset $U$, which is in the kernel of $\dd^{(k)} \colon E_{-k} \to E_{-k+1}$
at every point of $U$. According to the previous discussion, for every point $m \in U$, and for every neighborhood $U_m\subset U$ of $m$, there exists
a smooth section $f_m \in \Gamma_{U_{m}}(E_{-k-1})$ such that $ \dd^{(k+1)} (f_m) = e$. 
From the family $(U_m)_{m \in U}$, we can extract a locally finite open cover $(U_{m_{i}})_{i \in I}$ indexed by $I$  and choose  a partition of unity $
(\chi_i)_{i \in I}$ subordinate to it. Since $\dd^{(k+1)}$ is ${\mathscr O}$-linear, 
$f := \sum_{i \in I} \chi_i f_{i} $ is a section of $E_{-k-1}$  over $U$, which, by construction, satisfies $\dd^{(k+1)} (f)=e$.
This proves the second item.

\isthiswhatyouowant
The \emph{fourth item} is proved in Example \ref{tu}.

\isthiswhatyouowant
For \emph{item five}, one can proceed as follows. Let $(E,\dd,\rho)$ be a geometric resolution of ${\mathcal F}$
and $ m \in M$. Let $e_1, \dots,e_k\in E_{-1}|_m$ be a basis of ${\rm \dd}^{(2)}(E_{-2}|_m)$.
Denote by $\tilde{e}_1, \dots,\tilde{e}_k$ local sections of $E_{-2}$ whose images by $\dd^{(2)}$, when evaluated at $m$, 
coincide with $ e_1, \dots,e_k$, respectively. In a neighborhood $U_1$ of $m$, the sections
$\tilde{e}_1, \dots,\tilde{e}_k$, as well as their images $\dd^{(2)}\tilde{e}_1, \dots,\dd^{(2)}\tilde{e}_k$,
are independent at every point, and therefore define sub-vector bundles $F_{-2} \subset E_{-2}$ and $F_{-1} \subset E_{-1}$, respectively.
It is easy to check that $(E',\dd',\rho')$ is again a geometric resolution of $ {\mathcal F}$,
where $ E_{-i}':=E_{-i} $ for $i \neq 1,2$ and   $E_{-i}':={E_{-i}}\slash{F_{-i}} $ for $i=1,2$
and where $\dd'$ and $\rho'$ are the uniquely induced maps on these quotient spaces.
For this new geometric resolution, the map $(\dd')^{(2)}\colon E_{-2}' \to E_{-1}'$ is zero at the point $m$ by construction.
The operation can then be repeated for the index $i=2$ to find a new geometric resolution such that $\dd^{(3)}$
is zero at the point $m$ and can be continued by recursion. Each step may require to shrink the neighborhood of $m$ on which the geometric resolution is defined, but since the geometric resolution is of finite length, 
only finitely many such operations are required, and the procedure gives a geometric resolution defined in a neighborhood of $m$. It is minimal at $m$ by construction. 
\end{proof}

\subsubsection{Proof of Theorem \ref{newprop} }

\isthiswhatyouowant
The proof of Theorem \ref{newprop} requires some preparation.
Let $(E,\dd)$ be a complex of vector bundles over the manifold $M$.  Consider the bicomplex $(\Gamma(  E_{-i}^* \otimes E_{-j} ),  \dd^*  \otimes  {\rm{id}},  {\rm{id}} \otimes \dd) $:
\begin{equation}
\label{eq:bicomplex}
 \xymatrix{ & \vdots & \vdots &  \vdots \\ \cdots
	  \ar[r]^{ }  &	\Gamma( E_{-2}^* \otimes E_{-3}) \ar[r]^{  {\rm{id}} \otimes \dd } 
	  \ar[u]^{  \dd^*  \otimes  {\rm{id}}}  &\Gamma( E_{-2}^* \otimes E_{-2}) \ar[u]^{  \dd^*  \otimes  {\rm{id}}}  \ar[r]^{  {\rm{id}} \otimes \dd }   & \Gamma( E_{-2}^* \otimes E_{-1} )   \ar[u]^{  \dd^*  \otimes  {\rm{id}} } \\  \cdots
  \ar[r]^{ }  &	\Gamma( E_{-1}^* \otimes E_{-3}) \ar[r]^{  {\rm{id}} \otimes \dd } 
   \ar[u]^{  \dd^*  \otimes  {\rm{id}}}  &\Gamma( E_{-1}^* \otimes E_{-2}) \ar[u]^{  \dd^*  \otimes  {\rm{id}}}  \ar[r]^{  {\rm{id}} \otimes \dd }   & \Gamma( E_{-1}^* \otimes E_{-1} )   \ar[u]^{  \dd^*  \otimes  {\rm{id}} }  } \end{equation}
A  graded vector bundle morphism $ \Xi \colon E_\bullet \to E_{\bullet+k}$ over the identity of $M$ can be seen as an element of degree $k$ in this bicomplex that we denote by  $ \overline{\Xi}$. 
\begin{lemme}
	\label{homAsBicompl}
	Let $(E,\dd)$ be a complex of vector bundles over the manifold $M$ and $D=\dd^*  \otimes  {\rm{id}}- {\rm{id}} \otimes \dd $ the total differential of the bi-complex (\ref{eq:bicomplex}). 
	Then:
	\begin{enumerate}
		\item A degree $0$ graded vector bundle morphism $ \Xi: E_\bullet \to E_{\bullet}$ is a chain map, if and only if $D(\overline{\Xi})=0$.
		\item Two chain maps $ \Xi, \Xi'\colon E_\bullet \to E_{\bullet}$ are homotopic with homotopy $h\colon  E_\bullet \to E_{\bullet-1} $, if and only if $\:\overline{\Xi} -\overline{\Xi'} = D (\overline{h}) $.
		\item Let $k \in \mathbb{N}$ and  $(E,\dd, \rho)$ be a geometric resolution of finite length. 
		Two chain maps $ \Xi, \Xi'\colon E_\bullet \to E_{\bullet}$ over ${\mathcal F}$ that coincide upon restriction to $ E_{-i}$ for $ i \geq k$ are homotopic with 
		respect to a homotopy whose restriction to $ E_{-i}$ is zero for all $ i \geq k$ if $k \ge 2$ and for all $i \ge 2$ if $k=1$. 
		
	\end{enumerate}
\end{lemme}
\begin{proof}
The first two items are straightforward to prove and we recommend it to a reader who is not yet familiar with it. For the \emph{third} item,
Lemma \ref{lem:existsAChain} implies that $\Xi,\Xi'$ are homotopy equivalent.
Since they agree in degree $i >k$, their difference
   $ \overline{\Xi}- \overline{\Xi'}$ belongs to $\bigoplus_{i=1}^k \Gamma( E_{-i}^* \otimes E_{-i} )$ and is exact by item two,  $ \overline{\Xi}- \overline{\Xi'} =D(\overline{h} )$. Because all  vertical lines 
 $ (\Gamma(  (E_{-i}^*)_{i\ge 1} \otimes E_{-j} ,  \dd^*  \otimes  {\rm{id}}) $ for a fixed $j$ in \eqref{eq:bicomplex}
 are exact except at $i=1$, we may choose this $\overline{h}$ to be an element in $\bigoplus_{i=1}^{k-1} \Gamma( E_{-i+1}^* \otimes E_{-i} ) $ if $k \geq 2$ and in $E_{-1}^* \otimes E_{-2}$, if $k=1,2$.
 \end{proof}

\isthiswhatyouowant
The following lemma about geometric resolutions is interesting by itself.
\begin{lemme}
	\label{lem:Assumetrivial}
	Assume that a singular foliation $ {\mathcal F} $ admits a geometric resolution $(E,\dd,\rho)$ of finite length $d$,
then it can be replaced by another one of the same length $d$, $(E',\dd',\rho')$, such that for all $ i \geq 2$ the vector bundle $ E_{-i}'$ is trivial. 	
\end{lemme}
\begin{remarque}
	\normalfont
	We remark in parenthesis that we cannot also make, e.g., $E_{-1}$ trivial: 
	For example, let $M$ be the two-dimensional Moebius strip, viewed upon as a non-trivial line bundle over $S^1$, and let ${\mathcal F}\subset TM$ be the regular vertical foliation. Now there is an isomorphism of line bundles $\rho \colon \Lambda^2TM \to {\mathcal F}$, which we can view as a length one resolution $0 \to E_{-1} \to TM $. Evidently, $E_{-1}$ is non-trivial since $M$ is not orientable. By the above procedure, we can make this resolution longer, but only shifting the non-trivial factor to the left.
\end{remarque}
\begin{proof}[Proof (of Lemma \ref{lem:Assumetrivial})] 
	Let $(E,\dd,\rho)$ be a geometric resolution of $ {\mathcal F}$. 
	 Replacing $ \dd^{(i)}\colon E_{-i}\to E_{-i+1}$ by $\dd^{(i)} \oplus {\rm{id}}_{V}\colon E_{-i}\oplus V \to E_{-i+1} \oplus V $ for some vector bundle $V$, we do not spoil the property of being a geometric resolution of $ {\mathcal F}$.
	 If we choose a vector bundle $V$ which turns $E_{-i}\oplus V$ into a trivial one, we have obtained a new geometric resolution for which the component of degree $-i$ is trivial, and the only modified components are those of degrees $-i$ and $-i+1$.
	 The geometric resolution satisfying the requirements of Lemma \ref{lem:Assumetrivial}
	 can be obtained by applying this procedure recursively: We start by turning the component of degree $-d$ into a trivial vector bundle, by modifying omponents of degree $ -d$ and $-d+1$ as above. Then we turn the  component of degree $ -d+1$ of the obtained geometric resolution into a trivial vector bundle, and so on. The last component that can be made trivial  by this procedure is the component of degree minus two. 
	\end{proof}

\isthiswhatyouowant    The proof of Theorem \ref{newprop} relies heavily on the next proposition:

\begin{proposition} \label{prop1}
For two geometric resolutions $(E,\mathrm{d},\rho)$ and  $(F,\mathrm{d}',\rho')$ of $ {\mathcal F}$ of finite length over $U$ and $V$, respectively,
which satisfy items $\alpha$ and $\beta$ below for some $k \ge 2$, we can modify them into two such resolutions satisfying these conditions for $k$ lowered by one.
\begin{enumerate}
	\item[$\alpha$.]  There are  two degree $-1$ maps $h\colon E_{\bullet}\to E_{\bullet-1}$ and $h'\colon F_{\bullet}\to F_{\bullet-1}$ which vanish upon restriction to $E_{-i}$ and $F_{-i}$ 
	for all $ i \geq k$ if $k \geq 2 $ and for all $ i \geq 2 $ if $k=1$, and two chain maps $\phi\colon E\to F$ and $\psi\colon F\to E$ and such 
	that the following diagram constitutes a homotopy equivalence over~$U \cap V$ 
	\begin{equation}
	\label{eq:recursion}
	{\tiny{
		\begin{tikzcd}[column sep=1.3cm,row sep=0.7cm]
		\cdots \ar[r,shift left =0.5ex,"\mathrm{d}"]	& E_{-k-1} \ar[r,shift left =0.5ex,"\mathrm{d}"] \ar[dd,shift right =0.5ex,"\phi_ {k+1}" left] & E_{-k} \ar[dd,shift right =0.5ex,"\phi_{k}" left] \ar[r,shift left =0.5ex,"\mathrm{d}"] &\ar[l, dashed, shift left =0.5ex,"h"]E_{-k+1} \ar[dd,shift right =0.5ex,"\phi_{k-1}" left] \ar[r,shift left =0.5ex,"\mathrm{d}"] &\ar[l, dashed, shift left =0.5ex,"h"]\ldots\ar[r,shift left =0.5ex,"\mathrm{d}"]&  \ar[l, dashed, shift left =0.5ex,"h"]E_{-1} \ar[dd,shift right =0.5ex,"\phi_1" left]\ar[dr,"\rho"]&\\
		& &&&&&TM|_{U \cap V}\\ 	\cdots \ar[r,shift left =0.5ex,"\mathrm{d}"]
			& \ar[r,shift left =0.5ex,"\mathrm{d}"] F_{-k-1}  \ar[uu,shift right =0.5ex,"\psi_{k+1}" right] & F_{-k} \ar[uu,shift right =0.5ex,"\psi_k" right]\ar[r,shift left =0.5ex,"\mathrm{d}'"] &\ar[l, dashed, shift left =0.5ex,"h'"] F_{-k+1} \ar[uu,shift right =0.5ex,"\psi_{k-1}" right]\ar[r,shift left =0.5ex,"\mathrm{d}'"] &\ar[l, dashed, shift left =0.5ex,"h'"] \ldots\ar[r,shift left =0.5ex,"\mathrm{d}'"]&  \ar[l, dashed, shift left =0.5ex,"h'"] \ar[uu,shift right =0.5ex,"\psi_1" right]F_{-1}\ar[ur,"\rho'" below right]&
		\end{tikzcd}
	}}
	\end{equation}
	In particular $ \phi_i$ and $ \psi_i $ are (strictly) inverse to one another for $ i \geq k+1$.
	\item[$\beta$.]    For $i\ge 2$, each one of the vector bundles $E_{-i}$  and $F_{-i}$ is trivial.
	\end{enumerate}
	\end{proposition}
		\isthiswhatyouowant
	The proof of Proposition \ref{prop1} requires two lemmas.	
	\begin{lemme}
		\label{lem:PhikPsik}
		 Let $\phi,\psi,h$ be a homotopy equivalence as in (\ref{eq:recursion}).
		The following vector bundle automorphisms of $ E_{-k}\oplus F_{-k} $ over $U \cap V$, 		\begin{equation*}
		\Phi_k=\begin{pmatrix}
		h\circ\mathrm{d} & -\psi_k \\
		\phi_k & -\mathrm{id}
		\end{pmatrix}
		\hspace{1cm}\text{and}\hspace{1cm}
		\Psi_k=\begin{pmatrix}
		-\mathrm{id} & \psi_k \\
		-\phi_k &  h'\circ\mathrm{d}'
		\end{pmatrix} ,
		\end{equation*}
		written in an obvious matrix notation,
are inverse to one another. 
	\end{lemme}
	\begin{proof}
		A direct calculation yields:
		\begin{equation*}
		{\tiny{
				\Psi_k\circ\Phi_k=\begin{pmatrix}
				\mathrm{id}&0 \\
				\big(h'\circ \phi_{k-1}-\phi_k \circ h\big)\circ\mathrm{d} & \mathrm{id}
				\end{pmatrix}}}
		\hspace{1cm}\text{and}\hspace{1cm}{\tiny{
				\Phi_k\circ\Psi_k=\begin{pmatrix}
				\mathrm{id}&\big(h \circ \psi_{k-1}-\psi_k \circ h'\big)\circ\mathrm{d}' \\
				0 & \mathrm{id}
				\end{pmatrix} \, .
			}}
			\end{equation*}

			\isthiswhatyouowant
			By item $\alpha$ in the recursion relation, we have $  \phi_k \circ \psi_k =\mathrm{id} + h' \circ {\rm d}' $ and 
			$  \psi_k \circ \phi_k = \mathrm{id} + h \circ {\rm d} $. 
			Applying $\phi_k$ on the right to the first of these relations and to the left of the second one, we obtain
			$$\phi_k \circ \psi_k \circ \phi_k = \phi_k + h' \circ \phi_{k-1} \circ {\rm d} = \phi_k + \phi_k \circ h \circ {\rm d} .$$
			This implies that the lower left block of $\Psi_k\circ\Phi_k$ is zero. The same procedure applied to the upper right block of $\Phi_k\circ\Psi_k$ concludes the proof. 
		\end{proof}

			\isthiswhatyouowant
			The proof of the next lemma follows from a direct computation, left to the reader.
			\begin{lemme}
				\label{lem:PhikPsikMoinsUn}
				  Let $\phi,\psi,h$ be a homotopy equivalence as in (\ref{eq:recursion}) and  $ \Phi_k,\Psi_k$  as in Lemma \ref{lem:PhikPsik}.
				Let us consider the following vector bundle automorphisms $\Phi_{k-1}$ and $\Psi_{k-1}$ of $E_{-k+1}\oplus F_{-k+1}$  over $U \cap V$:
				\begin{equation*}
				\Phi_{k-1}=\begin{pmatrix}
				h & -\psi_k \\
				\phi_{k-1} & -\mathrm{d}'
				\end{pmatrix}
				\hspace{1cm}\text{and}\hspace{1cm}
				\Psi_{k-1}=\begin{pmatrix}
				-\mathrm{d} & \psi_{k-1} \\
				-\phi_k &  h'
				\end{pmatrix}\,  .
				\end{equation*}
				They make the following diagrams commutative,
				\begin{equation}
				\label{eq:modExtension1}
				\begin{tikzcd}[column sep=1.3cm,row sep=0.7cm]\cdots
				E_{-k-1}\ar[d,shift right =0.5ex,"\phi_{k+1}" left] \ar[r,"\overline{D}"]&\ar[d,shift right =0.5ex,"\Phi_{k}" left] E_{-k}\oplus F_{-k} \ar[r,"D"] &E_{-k+1}\oplus F_{-k}\ar[d,shift right =0.5ex,"\Phi_{k-1}" left] \ar[r,"\widetilde{D}"] &E_{-k+2}\ar[d,shift right =0.5ex,"\phi_{k-2}" left] \cdots\\
				\cdots F_{-k-1}\ar[r,"\overline{D}'"]& E_{-k} \oplus  F_{-k} \ar[r,"D'"] &E_{-k} \oplus F_{-k+1} \ar[r,"\widetilde{D}'"] &F_{-k+2}\cdots
				\end{tikzcd}
				\end{equation}
				and
				\begin{equation}
				\label{eq:modExtension2}
				\begin{tikzcd}[column sep=1.3cm,row sep=0.7cm]\cdots
				E_{-k-1}\ar[r,"\overline{D}"]& E_{-k}\oplus F_{-k} \ar[r,"D"] &E_{-k+1}\oplus F_{-k} \ar[r,"\widetilde{D}"] &E_{-k+2} \cdots\\ \ar[u,shift right =0.5ex,"\psi_{k+1}" left]
				\cdots F_{-k-1}\ar[r,"\overline{D}'"]&  \ar[u,shift right =0.5ex,"\Psi_{k}" left] E_{-k} \oplus  F_{-k} \ar[r,"D'"] &E_{-k} \oplus F_{-k+1}  \ar[u,shift right =0.5ex,"\Psi_{k-1}" left] \ar[r,"\widetilde{D}'"] &F_{-k+2}\cdots  \ar[u,shift right =0.5ex,"\psi_{k-2}" left]
				\end{tikzcd}
				\end{equation}
				if $D$ and $\overline{D}$ are defined as follows:
				\begin{equation}
				\label{eq:DDprime}
				\left\{
				\begin{array}{llll}
				\overline{D}=\begin{pmatrix}
				\mathrm{d}\\
				0
				\end{pmatrix},&
				\hspace{1cm}
				D=\begin{pmatrix}
				\mathrm{d} & 0 \\
				0 & \mathrm{id}
				\end{pmatrix},&
				\hspace{1cm}\text{ and}\hspace{1cm}&
				\widetilde{D}=\begin{pmatrix}
				\mathrm{d}&0
				\end{pmatrix} 
				,\\
				\overline{D}'=\begin{pmatrix}
				0\\
				\mathrm{d}'
				\end{pmatrix},&
				\hspace{1cm}
				D'=\begin{pmatrix}
				\mathrm{id} & 0 \\
				0 & \mathrm{d}'
				\end{pmatrix},
				&\hspace{1cm}\text{and}\hspace{1cm}&
				\widetilde{D}'=\begin{pmatrix}
				0 & \mathrm{d}'
				\end{pmatrix}.
				\end{array} \right.
				\end{equation}
				
			\end{lemme}

\begin{proof}[Proof (of Proposition \ref{prop1})]
By item $\beta$, we know that there exist isomorphisms $\tau_k \colon E_{-k} \DistTo U \times \mathbb{W}$ and $\tau_k' \colon F_{-k} \DistTo V \times \mathbb{W}'$ for some vector spaces $\mathbb{W}$ and $\mathbb{W}'$, respectively. Now we replace the resolutions in the induction assumption by two new ones, which coincide with the previous ones except for in degrees $-k$ and $-k+1$, where:
\begin{equation}
\label{eq:modExtension}
\begin{tikzcd}[column sep=1.3cm,row sep=0.7cm]\cdots
E_{-k-1}\ar[r,"\overline{D}"]& E_{-k}\oplus U \times \mathbb{W}' \ar[r,"D"] &E_{-k+1}\oplus U \times \mathbb{W}' \ar[r,"\widetilde{D}"] &E_{-k+2} \cdots\\
\cdots F_{-k-1}\ar[r,"\overline{D}'"]& V \times \mathbb{W} \oplus  F_{-k} \ar[r,"D'"] &V \times \mathbb{W} \oplus F_{-k+1} \ar[r,"\widetilde{D}'"] &F_{-k+2}\cdots
\end{tikzcd}
\end{equation}
with $ \overline{D},{D},\widetilde{D},\overline{D}',{D'},\widetilde{D}'$ as in \eqref{eq:DDprime}. 

	\isthiswhatyouowant            
Over the restriction to $ U \cap V$, we define a degree $0$ map  $\tilde{\Phi} = (\Phi_i)_{i \geq 1} $  of the upper geometric resolution in (\ref{eq:modExtension}) to the lower one:
$$ \left\{  \begin{array}{rcl} \tilde{\Phi}_{i}&=& \phi_i \hbox{ for $i \neq k,k-1$} , \\ \tilde{\Phi}_{k-1} & = & 
  \begin{pmatrix}
 \tau_k & 0 \\
  0 &  \mathrm{id}
  \end{pmatrix}  \circ \Phi_{k-1}\circ \begin{pmatrix}
\mathrm{id} & 0 \\
0 & \tau_k'
\end{pmatrix}^{-1} ,\\ \tilde{\Phi}_{k}& =&  \begin{pmatrix}
\tau_k & 0 \\
0 &  \mathrm{id}
\end{pmatrix}  \circ \Phi_k \circ \begin{pmatrix}
\mathrm{id} & 0 \\
0 & \tau_k'
\end{pmatrix}^{-1},\\  \end{array} \right. $$
where $\Phi_k$ is the map defined in Lemma \ref{lem:PhikPsikMoinsUn}.
Similarly, using the map $\Psi_k$,  we define a degree $0$  map $\tilde{\Psi} = (\Psi_i)_{i \geq 1} $  from the lower geometric resolution in (\ref{eq:modExtension}) to the upper one.
By Lemma \ref{lem:PhikPsikMoinsUn}, both  $\tilde{\Phi}$ and  $\tilde{\Psi}$ are chain maps.
By the inductive assumption,  $\tilde{\Phi}_i$ is the inverse of $\tilde{\Psi}_i$
for all $ i \geq k+1$ and  Lemma \ref{lem:PhikPsik} implies that $\tilde{\Phi}_k$ is the inverse of $\tilde{\Psi}_k$. 
As a consequence, $\tilde{\Psi} \circ \tilde{\Phi}$ and the identity map  coincide upon restriction to $ E_{-i}|_{ U \cap V}$ for $ i \geq k$ and
likewise so for  $\tilde{\Phi} \circ \tilde{\Psi}$) upon restriction to $ F_{-i}|_{ U \cap V}$. The third item of Lemma \ref{homAsBicompl}, finally, yields homotopies, 
where the exceptional case $k=1$ is treated separately. 
These, together with $\tilde{\Psi} $ and $\tilde{\Phi}$, show that the modified geometric resolutions \eqref{eq:modExtension} indeed satisfy the conditions of
Proposition \ref{prop1} for $k$ lowered by one.
%
%
%
%
\end{proof}

\isthiswhatyouowant

\begin{lemme} \label{prop2}
	Two geometric resolutions  of $ {\mathcal F}$ of finite length over $U$ and $V$, respectively, which satisfy items $\alpha$ and $\beta$  in Proposition \ref{prop1} for $k =1$ have  restrictions to $U \cap V $ which are isomorphic.
\end{lemme}
	\begin{proof}
Let us spell out items $\alpha$ and $\beta$  in Proposition \ref{prop1} for $k = 1$: There exist
 two degree minus one maps $h\colon E_{\bullet}\to E_{\bullet-1}$ and $h'\colon F_{\bullet}\to F_{\bullet-1}$ which vanish upon restriction to $E_{-i}$ and $F_{-i}$, respectively, 
 for $i \geq 2$. In addition, there are two chain maps $\phi\colon E\to F$ and $\psi\colon F\to E$
 such that, for all $ i \geq 2$, the linear maps $ \phi_i$ and $ \psi_i$ are inverse to one another.  Together they make  the following diagram  a homotopy equivalence over~$U \cap V$:
\begin{center}
	\begin{tikzcd}[column sep=1.3cm,row sep=0.7cm]
		\ldots\ar[r,"\mathrm{d}"]&E_{-3}  \ar[r,"\dd"]\ar[dd,shift right =0.5ex,"\phi_3" left]&E_{-2} \ar[r,"\dd"] \ar[dd,shift right =0.5ex,"\phi_2" left]& E_{-1}\ar[dd,shift right =0.5ex,"\phi_1" left]\ar[dr,"\rho"] \ar[l, dashed, shift left =0.5ex,"h"] &\\
		& &&&TM\vert_{U \cap V}\\
		\ldots\ar[r,"\mathrm{d}'"]& \ar[uu,shift right =0.5ex,"\psi_3" right]F_{-3}\ar[r,"\dd'"] &  \ar[uu,shift right =0.5ex,"\psi_2" right] F_{-2} \ar[r,"\dd'"]& \ar[uu,shift right =0.5ex,"\psi_1" right] F_{-1} \ar[ur,"\rho'" below right] \ar[l, dashed, shift left =0.5ex,"h' "]& \\
	\end{tikzcd}.
\end{center}
Let us show that  $(E,\dd,\rho)$ and $(F,\dd',\rho') $ are isomorphic upon restriction to $U \cap V $.
		A direct computation shows that
		 $ \psi_1 - \dd \circ \psi_2 \circ h' $ is a right inverse of $ \phi_1$:
		 $$ \phi_1 \circ ( \psi_1 - \dd  \circ \psi_2 \circ h') = \mathrm{id} + \dd' \circ h' - \phi_1 \circ \dd  \circ \psi_2 \circ h'  = \mathrm{id} + \dd'  \circ h' - 
		 \dd'  \circ \phi_2 \circ \psi_2 \circ h'  = \mathrm{id}. $$ 
		 Since, for any two resolutions, the alternate sum of the ranks are equal by Lemma \ref{lem:alt_sum}, 
		 the ranks of $E_{-1}$ and $F_{-1}$ have to be equal, and thus the right-invertible morphism $\phi_1$ is invertible. 		\end{proof}

\begin{proof}[Proof (of Theorem \ref{newprop})]
	According to items one and two of Proposition \ref{bonjourprop}, locally real analytic singular foliations admit geometric resolutions  in some neighborhood of each point.
	We therefore have to show that geometric resolutions can be "glued".
	More precisely, we want to show that given  $U,V\subset M$ open sets over  ${\mathcal F}$  which admit geometric resolutions $(E,\mathrm{d},\rho)$ and  $(F,\mathrm{d}',\rho')$  of finite length $n+1$, respectively, one can construct a smooth geometric resolution $(G,\mathrm{d}'',\rho'')$ of $ {\mathcal F}$ on $ U \cup V$, which is equally of length $n+1$. Then the statement about relatively compact subsets of $M$ follows immediately. 
	

	\isthiswhatyouowant
	 We use Lemma \ref{lem:Assumetrivial} to replace both  $(E,\mathrm{d},\rho)$ and  $(F,\mathrm{d}',\rho')$  with two new ones for which each of the vector bundles $E_{-i}$ and $F_{-i}$ is trivial for $i \geq 2$. 
	 By Lemma \ref{lem:existsAChain2}, the two requirements $\alpha$  and $\beta$ of Proposition \ref{prop1} are satisfied for $k = n+1$.


	\isthiswhatyouowant
	Applying recursively Proposition \ref{prop1} from $n+1$ down to $2$, we construct two geometric resolutions over $U$ and $V$ which satisfy items $\alpha$ and $\beta$ for $k=1$. Lemma \ref{prop2}, finally, shows that the obtained geometric resolutions over $U$ and $V$ have restrictions to $ U \cap V$ which are isomorphic. 
	By taking their disjoint union  quotiented by the isomorphism, we obtain a geometric resolution of $ {\mathcal F}$ on $U \cup V$. 
	This concludes the proof of the theorem.
\end{proof}

	\begin{remarque} \normalfont Theorem \ref{newprop} does not make use of the Lie bracket of vector fields of a singular smooth foliation ${\mathcal F}$, but only uses the local $\cO$-module structure of ${\mathcal F}$. 
Thus, in fact, what we have proven is:  If a sheaf of locally finitely generated $ C^\infty(M)$-modules admits a finitely generated projective resolution of finite length at most $k$ around each point $m \in M$, it also admits a finitely generated projective resolution of length at most $k$ over any relatively compact open subset of $M$. 
	\end{remarque}

\subsubsection{Proof of Proposition \ref{prop:Marco}}

\begin{proof}[Proof (of Proposition  \ref{prop:Marco})]
Let $m \in M$ be a point and let $(E,{\rm d},\rho)$ be a geometric resolution  of $ {\mathcal F}$ of finite length, defined on a neighborhood $U$ of $m$.
Since the geometric resolution $(E,{\rm d},\rho)$ is of finite length, we can  assume without any loss of generality that the complex
\begin{center} 
\begin{tikzcd}[column sep=0.7cm,row sep=0.4cm]\label{eq:resolF}
\ldots\ar[r,"\dd^{(3)}"]&\Gamma_V(E_{-2})\ar[r,"\dd^{(2)}"]&\Gamma_V(E_{-1})\ar[r,"\rho"]&\mathcal{F}_V \ar[r]&0
\end{tikzcd}
\end{center}
is exact for all open subsets  $V \subset U$ \cite{MikeField}.

\isthiswhatyouowant
Let $m'$ be a regular point of $ {\mathcal F}$ contained in $U$ and $ V\subset U$ a neighborhood of $m'$ on which the foliation is generated by $r$ vector fields $X_1,\dots,X_r$.
Under this assumption, the restriction to $V$ of the foliation ${\mathcal F}$ admits a geometric resolution  $(E',{\rm d}',\rho')$ of length one:
It is given by the vector bundles $E_{-i}'=0$ for $i \neq 2$, $E_{-1}':={\mathbb R}^r \times V$ and the anchor map 
$\rho'\colon (\lambda_1, \dots, \lambda_r)  \mapsto \sum_{i=1}^r \lambda_i X_i $
(with  the understanding that sections of $E_{-1}'$ are seen as $r$-tuples of functions $(\lambda_1, \dots, \lambda_r)$ on $V$.)

\isthiswhatyouowant
Hence, the restriction to $V$ of the singular foliation $ {\mathcal F}$
admits two different geometric resolutions: the restriction to $V $ of $(E,\dd,\rho)$
and the geometric resolution  $(E',{\rm d}',\rho')$.
By Lemma \ref{lem:existsAChain2}, these geometric resolutions are homotopy equivalent.
In particular, the alternate sum of the ranks is equal to $r$, Lemma \ref{lem:alt_sum} implies that $r = \sum_{i \geq 1 }  (-1)^{i-1} {\rm rk} (E_{-i})$.
It particular, all the regular leaves contained in $U$ have the same dimension. 
The manifold $M$ being connected, the argument above being valid for some neighborhood $U$ of an arbitrary point in $M$, the dimensions of all the
regular leaves of ${\mathcal F} $ have to be equal to $r$. This completes the~proof.
\end{proof}

\subsection{Examples of geometric resolutions of singular foliations}
 \label{exampleresolutions}

 \isthiswhatyouowant
We give several examples of geometric resolutions of singular foliations:

\begin{example}\normalfont
\label{ex:regularFoliation}
For a regular foliation ${\mathcal F}$,  a geometric resolution is given by $E_{-1}=T{\mathcal F}$ and $E_{-i}=0$ otherwise.
The anchor map is the inclusion $T{\mathcal F} \hookrightarrow TM$.
\end{example}

\begin{example}\normalfont
	\label{ex:debord}
	Debord foliations (Definition \ref{def:debord})
are precisely singular foliations that admit a geometric resolution of length one.
\end{example}

The following four examples are set up in the algebraic context, but can be considered also within the holomorphic,  real analytic setting, and smooth setting, in part with the obvious adaptations.
Recall in this context that real analytic geometric resolutions are also smooth geometric resolutions, see Proposition \ref{bonjourprop}.

\begin{example}\label{sl2}\normalfont
The Lie algebra $\mathfrak{sl}_{2}$ has three canonical generators $h,e,f$ that satisfy $ [h,e]=2e, [h,f]=-2f $ and $[e,f]=h$.
It acts on ${\mathbb R}^2$  through the vector fields:
 \begin{equation} \label{eq:sl2} 
  \underline{h} = x \frac{\partial}{\partial x} - y  \frac{\partial}{\partial y}, \quad \underline{e} = x \frac{\partial}{\partial y},\quad
  \underline{f}= y \frac{\partial}{\partial x} .  
 \end{equation}
 Here, we denote by $x,y$ the coordinates of $ {\mathbb R}^2$.
  The ${\mathscr O}({\mathbb R}^2)$-module generated by $\underline{h},\underline{e},\underline{f}$ is a singular foliation,
  that can  be seen  as smooth or real-analytic. 
  The vector fields given in
 (\ref{eq:sl2}) are not independent over ${\mathscr O}({\mathbb R}^2)$, but every relation between them is a multiple of the relation:
   \begin{equation} \label{eq:relationsl2}    xy \underline{h}  + y^2 \underline{e} - x^2 \underline{f} =0. 
   \end{equation}   
Let us describe a geometric resolution.
We define $E_{-1}$ to be the trivial vector bundle of rank $3$ generated by $3$ sections that we denote by
$ \tilde{e},\tilde{f},\tilde{h}$. Then we define an anchor by
 \begin{equation} \rho(\tilde{e})=\underline{e},\quad \rho(\tilde{f})=\underline{f},\quad \rho(\tilde{h})=\underline{h}.\end{equation}
 We define $E_{-2}$ to be the trivial vector bundle of rank  $1$, generated by a section that we call~$\textbf{1}$.
 We define a vector bundle morphism from $E_{-2}$ to $E_{-1}$ by:
  \begin{equation} {\mathrm d}^{(2)} (\textbf{1}) = xy \tilde{h}  + y^2 \tilde{e} - x^2 \tilde{f}. \end{equation}
We then impose  $E_{-i}=0$ and $\dd^{(i)}=0$ for $i \geq 3$.
  The triple $(E,{\mathrm d},\rho)$ is a geometric resolution of the singular foliation given by the action of $\mathfrak{sl}_{2}$ on ${\mathbb R}^2$.
\end{example}

\begin{example}\normalfont
\label{conjecture}
We owe this example to discussions with Rupert Yu and Alexei Bolsinov.
 The adjoint action of a complex  Lie algebra   $\mathfrak{g}$ on itself defines a holomorphic singular foliation ${\mathcal F}_{ad}$ on the manifold $M := \mathfrak{g}$.

 \isthiswhatyouowant Let $ P_1, \dots,P_l$ be  generators of $S({\mathfrak g})^{\mathfrak g}$, i.e.~the algebra of polynomial
functions on ${\mathfrak g}$ invariant under adjoint action. According to Chevalley's theorem \cite{Chevalley}, these generators can be chosen to be independent as polynomials,
and $l$ coincides with the rank of ${\mathfrak g}$. More precisely, it follows from Theorem 6.5 and Corollary 6.6 in \cite{Veldkamp} that the differentials of $P_1,\dots,P_l$ are independent covectors at every point in a regular orbit.  Since singular points are of codimension $3$ by Theorem 4.12 in \cite{Veldkamp}, the
differentials of these functions are independent at every point outside a sub-variety of codimension $3$ in ${\mathfrak g}$. Let us call regular points the points in this Zarisky open subset of $\mathfrak{g}$.

 \isthiswhatyouowant The foliation given by the adjoint action of ${\mathfrak g}$ on itself admits a geometric resolution that can be described as follows.
Let $ E_{-1} $ be the trivial bundle over  $M= \mathfrak{g} $ with typical fiber $ \mathfrak{g}$,
 and $E_{-2}$ to be the trivial bundle over $M$ with typical fiber $ {\mathbb R}^l$.
 Let $\rho : E_{-1} \to TM $ be, at a point $m \in M= \mathfrak{g}$, the vector bundle morphism obtained by mapping $a \in (E_{-1})_m \simeq \mathfrak{g} $ 
 to $[a,m] \in T_m M \simeq \mathfrak{g} $.
 Let $ \dd^{(2)}$ be the vector bundle morphism mapping, for all $m \in M $, an $l$-tuple
 $ (\lambda_1, \dots,\lambda_{l}) \in (E_{-2})_m $ to $ \sum_{i=1}^{l} \lambda_i \,  {\rm grad}_m (P_i)  \in (E_{-1})_m \simeq \mathfrak{g} $,
where ${\rm grad}$ stands for the gradient computed with the help of the Killing form.
 
 \isthiswhatyouowant By construction, we have $\rho\big(\Gamma(E_{-1})\big)= {\mathcal F}_{ad}$. It is clear that $ \rho \circ \dd^{(2)} =0$ and that the image of $\dd^{(2)} $ coincides
 with the kernel of $\rho$ at all regular points. Alexei Bolsinov gave us the following argument, that shows that the previous complex is a geometric resolution of ${\mathcal F}_{ad}$.
 Let $ a$ be a holomorphic section of $ E_{-1} $ which is in the kernel of $\rho$ at all point. Then, at all regular points $ m \in {\mathfrak g}$,
 we know that there exists unique $\lambda_1(m), \dots,\lambda_{l}(m) $ such that:
  $$  a(m) \, =  \, \sum_{i=1}^l \lambda_{i}(m) \, {\rm grad}_m (P_i) .$$
The functions $ m \mapsto \lambda_{i}(m)$ are meromorphic on $ {\mathfrak g}$ and are holomorphic at regular points. Since singular points are of codimension $3$ by the discussion above, these
functions extend to holomorphic functions on ${\mathfrak g}$, by, e.g.~Theorem III.6.12 in \cite{Fichte}. This means that the section $a$ lies in the image of ${\dd}^{(2)}$. Hence the complex is exact.
  \end{example}

 \begin{example}
 \label{ex:KoszulLieuSing}
  \normalfont
  Let $\varphi$ be a  polynomial function on $V :={\mathbb C}^n$.
The contraction by  ${\dd} \varphi$ defines a complex of trivial vector bundles over $V$:
   \begin{center}
\begin{tikzcd}[column sep=0.9cm,row sep=0.6cm]
\ldots\ar[r,"\iota_{\dd\varphi}"]& \wedge^3 TV  \ar[r,"\iota_{\dd\varphi}"]& \wedge^2 TV \ar[r,"\iota_{\dd\varphi}"]&   TV \ar[r,"\iota_{\dd\varphi}"]&\underline{{\mathbb C}},
\end{tikzcd}
\end{center} 
 where $\underline{W}$ stands for the trivial bundle $V \times W$ with fiber $W$. Let $ {\mathfrak X}^i :=\Gamma(\Lambda^i TV) $ stand for the sheaf of $i$-multivector fields on $V$. Taking sections of the previous complex of vector bundles, 
  we obtain the a complex of ${\mathcal O}$-modules  called a \emph{Koszul complex} associated to $\varphi$: 
  \begin{center}
\begin{tikzcd}[column sep=0.9cm,row sep=0.6cm]
\ldots\ar[r,"\iota_{\dd\varphi}"]& \mathfrak{X}^3 \ar[r,"\iota_{\dd\varphi}"]& \mathfrak{X}^2\ar[r,"\iota_{\dd\varphi}"]&  \mathfrak{X}^1\ar[r,"\iota_{\dd\varphi}"]&\mathcal{O},
\end{tikzcd}
\end{center} 
The image of the map $   {\mathfrak X}^1 (V) \to {\mathscr O}(V)$
is the ideal generated by the functions $  \frac{\partial \varphi}{\partial x_1},  \dots, \frac{\partial \varphi}{\partial x_n} $.
According to a classical theorem of Koszul \cite{Eisenbud}, the previous complex is exact, if 
the sequence $\frac{\partial \varphi}{\partial x_1},  \dots, \frac{\partial \varphi}{\partial x_n} $
is a regular sequence. This happens in particular when $\varphi$ is weight-homogeneous and admits an isolated singularity at the origin.

\isthiswhatyouowant
In that case, the following complex of vector bundles over $M$
  \begin{center}
\begin{tikzcd}[column sep=0.9cm,row sep=0.6cm]
\ldots\ar[r,"\iota_{\dd\varphi} \otimes {\rm id}"]& \wedge^3 TV  \otimes \underline{V} \ar[r,"\iota_{\dd\varphi} \otimes {\rm id}"]& \wedge^2 TV \otimes  \underline{V} 
\ar[r,"\iota_{\dd\varphi} \otimes {\rm id}"]&   TV
 \otimes  \underline{V}\ar[r,"\iota_{\dd\varphi} \otimes {\rm id}"]&  \underline{{\mathbb C}} \otimes  \underline{V}\simeq   \underline{V},
\end{tikzcd}
\end{center} 
is again trivial at the level of sections, since it is just $n$ copies of the above complex.
It is therefore a geometric resolution of the singular foliation
generated by the vector fields
$$  \left\{ \frac{\partial \varphi}{\partial x_i}  \frac{\partial }{\partial x_j}\,, \ \hbox{  with $i,j =1, \dots,n$ } \right\}
.$$
This singular foliation is what we will call \emph{singular foliation of vector fields vanishing on the singular locus of $ \varphi$}.
When the ideal generated by $\frac{\partial \varphi}{\partial x_1}, \dots, \frac{\partial \varphi}{\partial x_n}$ is nilradical,
it is exactly the singular foliation of vector fields vanishing on the subset 
$\frac{\partial \varphi}{\partial x_1} = \dots = \frac{\partial \varphi}{\partial x_n}=0$.
\end{example}

 \begin{example}
 \normalfont
 \label{ex:vfVanishingAtZero}
Let ${\mathcal F}$ be the singular foliation of all vector fields vanishing at the origin $0$ of a vector space $V = {\mathbb C}^n$ (or  $V = {\mathbb R}^n$). Applying Example \ref{ex:KoszulLieuSing} 
to the function $ \varphi = \frac{1}{2} \sum_{i=1}^n x_i^2 $, we obtain a geometric resolution of $ {\mathcal F}$.
Let us describe it in a precise manner, where, in the first factor and to simplify the notation, we identify $V$ and $V^*$ using the canonical inner product:
\begin{enumerate}
 \item For all $i \in {\mathbb Z}$, $E_{-i}$ is the trivial bundle 
over $V$ with fiber $ \wedge^{i} V^* \otimes V  $.
 \item At a given point $e \in V$, the anchor map $V^* \otimes V \to T_e V \simeq V$ is given by
 $\rho(\alpha \otimes v ) = \iota_e \alpha  \, v$.
 \item At a given point $e \in V$, the differential 
 $ {\rm d}^{(i+1)}   \colon  E_{-i-1}  \to E_{-i} $ is given by:
 \begin{equation}
 {\rm d}^{(i+1)} (\alpha \otimes u ) := ({\mathfrak i}_e \alpha)  \otimes u, \end{equation} for all $\alpha \in \wedge^{i+1} V^*, u \in V.$
\end{enumerate}
\end{example}

 \begin{example}
 	\label{ex:vanishingOrder2}
 	 \normalfont
 	Let ${\mathcal F}$ be the singular foliation of all vector fields vanishing to order $k$ at the origin $0$ of a vector space $V $ of dimension $n$ over ${\mathbb K}={\mathbb R}$ or  ${\mathbb K}={\mathbb C}$. 
By Hilbert's syzygy theorem, there exist a projective resolution of length $n+1$
 of the ideal $ {\mathcal I}_0^k$ of functions on $V$ vanishing to order $k$ at $0 \in V$. This projective resolution has to be of the form:
$$
\xymatrix{
	\ldots\ar[r]^{\delta}  &\Gamma(I_{-2} ) \ar[r]^{\delta }& \Gamma(I_{-1}) 
	\ar[r]^{\rho_0 }  &{\mathcal I}_0^k 
} $$
for some family of vector bundle $(I_{-i})_{i \geq 1} $ over $V$ (which are in fact trivial vector bundles).
   A resolution of $ {\mathcal F}$ is then given by
$$
 		\xymatrix{
 			\ldots\ar[rr]^{\delta \otimes {\rm id}}&  &\Gamma(I_{-2}  \otimes_{\mathbb K} V) \ar[r]^{\delta \otimes {\rm id}}& \Gamma(I_{-1} \otimes_{\mathbb K} V) 
 			\ar[rr]^{\rho_0  \otimes {\rm id}} &  &{\mathcal I}_0^k \otimes_{\mathbb K} V = {\mathcal F}
 		} $$
 		(with the understanding that, e.g., an element $ f \otimes v \in {\mathcal I}_0^k\otimes_{\mathbb K} V$ is the vector field $f \underline{v}$ on $V$ where $\underline{v}$ is the constant vector field associated to $v \in V$).
 		
 \isthiswhatyouowant		For instance, for vector fields vanishing at $0$ on $ V={\mathbb K}^2 $, a resolution of functions vanishing to order $2$ at $0$ is given by:
 		 $$ 0 \to {\mathcal O}^2  \stackrel{\delta}{ \to} {\mathcal O}^3 \stackrel{\rho_0}{ \to} {\mathcal I}_ 0^2$$
 		 with $ \delta(F,G) = ( yF , -xF- yG , xG)$ and $ \rho_0 (F,G,H) = x^2 F + xy G + y^2 H $. Here $(x,y)$ are the canonical coordinates on $V$. 
 		 In this example,  $I_{-1}$ is easily identified with the trivial bundle with fiber $ S^2(V^*) $. Under this identification, the anchor maps $ \alpha \otimes v $ (with $ \alpha \in S^2( V^*)$, $v \in V$) to $ \alpha^* \, \underline{v}$,
 		 with the understanding that $\alpha^*$ stands for the homogeneous polynomial function of degree $2$ on $V$ associated to $ \alpha \in S^2(V^*)$. 
 \end{example}

\begin{example}
 \label{ex:Koszul}
   \normalfont
   Let $\varphi$ be a function  on $V:={\mathbb C}^n$ such that $ \left( \frac{\partial \varphi}{\partial x_1},  \dots, \frac{\partial \varphi}{\partial x_n} \right)$ is a regular sequence.
Consider the singular foliation ${\mathcal F}_{\varphi}$  made of all vector fields $X$ on $V$ such that $X[\varphi]=0$. 
Since the Koszul complex defined in Example \ref{ex:KoszulLieuSing} has no cohomology in degree $-1$,
the singular foliation  ${\mathcal F}_{\varphi}$ is generated by the vector fields:
\begin{equation}   \left\{ \frac{\partial \varphi}{\partial x_i}  \frac{\partial }{\partial x_j} - \frac{\partial \varphi}{\partial x_j}  \frac{\partial }{\partial x_i}\,,
\ \hbox{  with $ 1 \leq i < j \leq n$ } \right\} . \label{forthecounting} \end{equation}
Since the Koszul complex defined in Example \ref{ex:KoszulLieuSing} has no cohomology in degree $-i$ for $i \geq 2$,
a geometric resolution of that foliation is given by  $E_{-i} := \wedge^{i+1} V $  and $\dd := {\iota}_{{\rm d} \varphi}$.
In that case, $  \Gamma(E_{-i}) = {\mathfrak X}^{i+1} $
with ${\mathfrak X}^i$ being the projective $ {\mathcal O}(M)$-module of $i$-vector fields on $M \equiv {\mathbb C}^n$.
 \end{example}

\begin{example}
	\label{ex:affine2}
	\normalfont
	In the case of items 2 and 3 of Example \ref{ex:symmetries},
	the  Hilbert Syzygy theorem implies that a geometric resolution by trivial vector bundles of length $n+1$
	 exists \cite{Eisenbud}.
	 
\isthiswhatyouowant	 Similarly,  let ${\mathcal O}$ be the algebra of functions on an affine variety $W$. Since ${\mathcal O}$ is finitely generated, derivations ${\mathcal X}(W)$ of ${\mathcal O}$ are finitely generated as well and closed under the Lie bracket. 
	 Since ${\mathcal O}$ is Noetherian, ${\mathcal X}(W)$ admits a geometric resolution by trivial vector bundles. The Syzygy theorem implies that it can be chosen to be of finite length again.
	\end{example}

\begin{example}\label{tu}\normalfont The following example, that we owe to Jean-Louis Tu, provides a smooth singular foliation that does not admit smooth geometric resolutions.
Let $\chi$ be a smooth real-valued function on $M:={\mathbb R}$ vanishing identically on ${\mathbb R}_-$ and strictly positive on ${\mathbb R}^*_+$.
Consider the singular foliation ${\mathcal F}$ generated by the vector field ${ v}$ on $ {\mathbb R}$ defined by:
\begin{equation} {v} := \chi(t)\frac{\dd}{\dd t}. \label{xhi}\end{equation}
All points in ${\mathbb R}_-^*$ and all points in ${\mathbb R}_+^*$ are regular points.
There are, therefore, an uncountable family of regular leaves of dimensions $0$ and there is one regular leaf of dimension $1$.
Proposition \ref{prop:Marco} implies that no geometric resolution of finite length exists.

\isthiswhatyouowant But we can prove more: There is no geometric resolution of infinite length in a neighborhood of the point $t=0$.
Assume there is one. By an obvious adaptation of the last item in Proposition \ref{bonjourprop}, we can replace it
on a neighborhood of $0$ by a geometric resolution $(E,\dd,\rho)$ such that the vector bundle morphism $\dd^{(2)}$ is zero at the point $t=0$.
Without any loss of generality, we can assume there exists a nowhere vanishing section $e$ such that $\rho (e) = v $.
Since an open interval of ${\mathbb R}$ is a contractible manifold, the vector bundle $E_{-1} $ must be trivial.
Denote by $n_{1}$ the rank of $E_{-1}$. The module $\Gamma (E_{-1})$ is generated by $n_1$ generators $e_1, \dots,e_{n_1}$. 
Without any loss of generality, we can assume that $e_1=e$.
 We then have for every $1\leq k\leq n_{1}$: $\rho(e_k) = g_k { v} = g_k \rho(e_1)$ for some function $g_{k}\in\cinf(U)$. 
 It implies that $\rho(e_k - g_k e_1)=0$. Since $\mathrm{Im}\big(\dd^{(2)}\big)=\mathrm{Ker}(\rho)$, there exist  $f_2,\dots,f_{n_1} \in \Gamma(E_{-2})$ such that $\dd^{(2)}(f_k)=e_k-g_k e_1$ for all $k=2, \dots,n_1$. 
 This contradicts the assumption  that $\dd^{(2)}$ vanishes at $t=0$, unless $n_1=1$.
 Now, if $n_1=1$, then the kernel of $\rho \colon \Gamma_U(E_{-1}) \to {\mathfrak X}({U})$ is made of all real-valued functions vanishing on ${\mathbb R}_+ \cap U$.
 This is not a finitely generated module.
 As a conclusion, ${\mathcal F}$ does not admit smooth geometric resolutions. 
 
\isthiswhatyouowant Notice, however, that this foliation \emph{does} result from the trivial rank one Lie algebroid whose anchor maps the constant section 1 to the vector field \eqref{xhi}.
\end{example}

\subsection{Lie \texorpdfstring{$\infty$}{infinity}-algebroids, their morphisms, and homotopies of those}\label{infinityalgebroids}
\label{subsec:homotopies}

\isthiswhatyouowant
In this section we recall and/or provide the needed facts about {\LieInftyAlgebroid}s, morphisms between {\LieInftyAlgebroid}s, and a good notion of homotopies between such morphisms.
It will be defined in the smooth, real analytic, and holomorphic settings all together.
We will also define the functor to the isotropy Lie $\infty$-algebra associated to each point $m \in M$.

\subsubsection{Lie $\infty$-algebras, {\LieInftyAlgebroid}s,  and $NQ$-manifolds} 

\isthiswhatyouowant
Let us explain how we deal with the notion of {\LieInftyAlgebroid}s and its dual notion of $NQ$ manifolds.
In this article, we think and prove results with the $NQ$-manifold point of view, because this is the point of view that makes morphisms 
easier to deal with.  But we state theorems using the {\LieInftyAlgebroid}s point of view, since it is a notion which seems easier to grasp for 
mathematicians who work on singular foliations, but are not used to the language of graded geometry.

\isthiswhatyouowant
\begin{definition}
	\label{def:LiInfinity}
	A \emph{\Linfty-algebra} is a graded vector space $ { E} = \bigoplus_{ i\geq 1} { E}_{-i}  $ 
together with a family of graded-symmetric $n$-multilinear maps $ \big(\{ \ldots \}_n\big)_{n \geq 1} $ of  degree $+1$, called $n$-ary brackets, which  
satisfy a set of compatibility conditions that are called \emph{higher Jacobi identities}. For all $n\geq2$ and for every $n$-tuple of homogeneous elements $x_{1},\ldots,x_{n}\in { E}$, they are defined by:
\begin{equation}\label{superjacobi}
\sum_{i=1}^{n}\sum_{\sigma\in {\rm Un}(i,n-i)}\epsilon(\sigma)\big\{\{x_{\sigma(1)},\ldots,x_{\sigma(i)}\}_{i},x_{\sigma(i+1)},\ldots,x_{\sigma(n)}\big\}_{n-i+1}=0,
\end{equation}
where $\epsilon(\sigma)$ is the \emph{Koszul sign} induced by the permutation of the elements $x_1,\ldots, x_n$:
\begin{equation}
 x_{\sigma(1)}\odot\ldots \odot x_{\sigma(n)} =\epsilon(\sigma) x_1\odot\ldots\odot x_n,
\end{equation}
where $\odot$ is the symmetric product on $\Gamma\big(S^n(E)\big)$.  ${\rm Un}(i,n-i)$ denotes the set of un-shuffles.
A \Linfty-algebra structure is said to be a \emph{Lie $n$-algebra} when ${ E}_{-i}=0$ for all $i \geq n+1$.
\end{definition}

 \isthiswhatyouowant
We now provide a possible definition of {\LieInftyAlgebroid}s \cite{LadaStasheff}:

\begin{definition}
	\label{def:Linftyoids}
 Let  $M$ be a smooth/real analytic/complex manifold whose sheaf of functions we denote by ${\mathscr O}$. 
 Let $E$ be a sequence $E=(E_{-i})_{1\leq i <\infty}$ of vector bundles over $M$.
 A \emph{{\LieInftyAlgebroid} (resp. Lie $n$-algebroid) structure on $E$} consists of a \Linfty-algebra (resp. Lie $n$-algebra) structure on the sheaf of sections of $E$
together with a vector bundle morphism $\rho\colon E_{-1} \to TM$, called the \emph{anchor},
such that the brackets $\{\ldots\}_n$ are $\mathscr{O} $-linear in each of their $n$ arguments except if $n=2$ and at least one of its two entries has degree minus one. Then, the $2$-ary bracket satisfies the following \emph{Leibniz identity}:
 \begin{equation}\label{robinson}
 \{x,fy\}_2=f\{x,y\}_2+ \rho(x)[f] \, y.
 \end{equation}
 for all $x \in \Gamma(E_{-1})$, $y \in \Gamma(E)$ and $f\in\mathcal{C}^\infty$.
\end{definition}
 \begin{remarque}\label{remarque2}
\normalfont
Its follows from these axioms that $\rho$ is a morphism of brackets (for all $x,y \in \Gamma(E_{-1})$ one has $\rho(\{x,y\}_2 )= [\rho(x),\rho(y)]$), and that  $\rho \circ \{\cdot\}_1\big|_{E_{-2}}=0$.
The above definition relies on the symmetric convention of the Lie $\infty$-algebras found e.g.~in \cite{Kajiura}. The original definition of Lie $\infty$-algebras involves graded skew-symmetric brackets \cite{Marklada};
however, they are in one-to-one correspondence with the above ones, cf \cite{Mehta, Leonid}. 
\end{remarque}

\isthiswhatyouowant
We observe that  for Lie $n$-algebroids  all $k$-ary brackets with $k \geq n+2$ are trivial for degree reasons.
Moreover, for every {\LieInftyAlgebroid}, the following sequence:
 \begin{equation} \label{sequencelinear0}
\begin{tikzcd}[column sep=0.9cm,row sep=0.6cm]
\ldots\ar[r,"\dd^{(4)}"]& E_{-3} \ar[r,"\dd^{(3)}"]& E_{-2}\ar[r,"\dd^{(2)}"]&  E_{-1}\ar[r,"\rho"]& TM,
\end{tikzcd}
\end{equation}
is a complex of vector bundles that we call its \emph{linear part}. 

\isthiswhatyouowant
Also, for every {\LieInftyAlgebroid} $E$ over $M$, the $2$-ary bracket restricts to a skew-symmetric bilinear bracket on $\Gamma(E_{-1})$. Together with the anchor map, 
it defines an almost-Lie algebroid structure on $E_{-1}$. Therefore, by using the first item of Proposition \ref{prop:Almost}, we obtain:
\begin{proposition} \label{prop:fromNQtoSingularFoliations2}
For every {\LieInftyAlgebroid} $E$ over $M$ with anchor $\rho$, the sheafification of the pre-sheaf $\rho\big(\Gamma(E_{-1})\big)$ is a singular foliation.
\end{proposition}
\noindent We call this singular foliation the \emph{singular foliation of the {\LieInftyAlgebroid} structure on $E$}. 

\isthiswhatyouowant
The definition of {\LieInftyAlgebroid}s above, although elementary, is quite cumbersome and often hard to use---especially when dealing  with morphisms later on. 
$Q$-manifolds with purely non-negative degrees, called $NQ$-manifolds, are much more efficient objects and in one-to-one correspondence
with the  {\LieInftyAlgebroid}s. Let us define them.

\isthiswhatyouowant
  We call a sequence $ E:=(E_{-i})_{i \geq 1}$ of finite rank vector bundles over $ {M}$ 
  indexed by negative numbers an \emph{$N$-manifold $E \to M$}.\footnote{Strictly speaking, this is only an example of an $N$-manifold, but after the choice of what is called \emph{splittings}, every $N$-manifold takes this form. We will henceforth always talk about these \emph{split} graded manifolds without further mention.}
   An element $x\in \Gamma(E_{-i})$ is said to be \emph{of degree $-i$}, written as $|x|=-i$.   
  The sheaf of graded commutative $ {\mathscr O}$-algebras of smooth, real analytic, or holomorphic sections of the graded-symmetric algebra $ S(E^*) $ will be denoted by $\mathcal{E}$ and called the \emph{functions on the $N$-manifold $E \to M$}. 
  Here, it is understood that $E^* = \bigoplus_{i \geq 1}E_{-i}^* $ and sections of $E_{-i}^*$ are considered to be of degree $+i$.
  
  \isthiswhatyouowant
  By construction, $\mathcal{E}$ is a sheaf of graded commutative ${\mathscr O}$-algebras.  
  For every positive $k$ and $n$, sections of
  $$ \bigoplus_{i_1+\dots+i_k=n} E^{*}_{-i_{1}}\odot\ldots\odot E^{*}_{-i_{k}} , $$
where $\odot$ denotes the graded-symmetric tensor product,   will be said to be of 
 \emph{degree} $n$ and of \emph{arity} $k$ and  will be denoted by $\mathcal{E}_n^{(k)}$. 
  
  \isthiswhatyouowant
 Graded derivations of $\mathcal{E} $  will be called \emph{vector fields} on the $N$-manifold $E \to M$.
A vector field $Q$ is said to be of \emph{arity} $k$ if, for all function $F \in \mathcal{E}$ of arity $l$, the arity of $Q[F]$ is $l+k$.
Every vector field $Q$ can be decomposed as an infinite sum:
 $$ Q = \sum_{k \geq -1} Q^{(k)}$$
with $Q^{(k)}$ being a vector field of arity $k$. 
The \emph{degree} of a vector field is defined in in a similar manner. 
A vector field $Q$ of odd degree commuting with itself, i.e.~satisfying $Q^2:=\frac{1}{2}[Q,Q]=0$, is said to be \emph{homological}.

\begin{definition}
An $NQ$-\emph{manifold} is a pair $ ( E,Q)$ where $  E \to { M}$ is an $N$-manifold over some base ${ M}$ and where $Q$ is a homological vector field of degree $+1$. 
\end{definition}

\isthiswhatyouowant
By construction, for every $NQ$-manifold $(E,Q)$ with sheaf of functions $\mathcal{E}$, we have an isomorphism of sheaves $\mathcal{E}_0 \simeq {\mathscr O}$, while
$ \mathcal{E}_1   \simeq \Gamma(E_{-1}^*)$, so that the derivation $Q $ maps  ${\mathscr O}$ to $\Gamma(E_{-1}^*) $. By the derivation property,
there exists a unique morphism of graded vector bundles $\rho \colon E_{-1} \to T M$ such that 
 \begin{equation}\label{eq:anchor} \big\langle Q \, f, x  \big\rangle = \rho(x) \, f\hspace{0.7cm} \text{for all}\hspace{0.4cm}  f \in {\mathscr O},\, x \in \Gamma(E_{-1}) .\end{equation}
 Here  $\langle \,.\, , . \,\rangle$ stands for the duality pairing between sections of a vector bundle and sections of its dual and the application of a vector field to a function is understood when the former one precedes the latter one.
We call the vector bundle morphism $ \rho$ the \emph{anchor map of the $NQ$-manifold $(E,Q) $}. 

\isthiswhatyouowant For a degree one vector field $Q$, the arity decomposition takes the form
\begin{equation} \label{Qarity}
Q = \sum_{k \geq 0} Q^{(k)}
\end{equation}
The next result is classical by now \cite{Voronov} 
and describes the duality between {\LieInftyAlgebroid}s  and $NQ$-manifolds.

\begin{theoreme} \label{thm:fromNQtoLinfinity}
Let $E=(E_{-i})_{i\geq 1}$ be a sequence of vector bundles over a manifold $M$. There is a one-to-one correspondence between (split) $NQ$-manifolds and {\LieInftyAlgebroid} structures on $E$. The anchor $\rho$ of both is identified by means of Equation \eqref{eq:anchor} above. 
In addition, under this correspondence:
\begin{enumerate}
\item The differential $\dd$ of the linear part of the {\LieInftyAlgebroid} structure is obtained by dualizing the arity zero component $Q^{(0)}$ of $Q$,
i.e.~for all $\alpha \in \Gamma(E^*) $ and $ x \in \Gamma(E)$:
\begin{equation} \label{Q0}
\big\langle Q^{(0)}\, \alpha,x\big\rangle=(-1)^{|\alpha|}\big\langle \alpha,\dd(x)\big\rangle.
\end{equation} 
\item The $2$-ary bracket $\{ .\, ,. \}_2$ and the arity one component $Q^{(1)}$    are related by:
  \begin{align}\label{def:bracket} \big\langle Q^{(1)}\, \alpha,x\odot y\big\rangle=\rho(x)\, \langle\alpha,y\rangle-\rho(y)\, \langle\alpha,x\rangle-\langle \alpha, \{x,y\}_{2}\rangle,\end{align}
for all homogeneous elements $x,y \in \Gamma(E)$ and $\alpha \in \Gamma(E^*)$, with  the understanding that $\rho$ vanishes on $ E_{-i}$ for $i \neq 1$.
\item For every $n\geq3$, the $n$-ary brackets $\{\ldots\}_n \colon \Gamma\big(S^n (E)\big) \to \Gamma (E)$ and the component $Q^{(n-1)} \colon  \Gamma(E^*) \to \Gamma (S^n E^*)$
of arity $n-1$ of $Q$ are dual one to the other. 
\end{enumerate}
\end{theoreme}

\isthiswhatyouowant This theorem justifies the following convention:

\begin{convention}
We shall denote {\LieInftyAlgebroid}s as pairs $(E,Q)$, with $Q$ the homological vector field of the corresponding $NQ$-manifold $E$.
As before,  ${\mathcal E}=\Gamma\big(S(E^*)\big)$ is its sheaf of functions. Its linear part  shall be denoted by $(E,\dd,\rho)$, with $\rho$ being the anchor map.
\end{convention}

	\subsubsection{Morphisms of {\LieInftyAlgebroid}s}

\isthiswhatyouowant

\begin{definition}
\label{def:morphism}
A {\LieInftyAlgebroid} morphism
from a {\LieInftyAlgebroid} $ ( E',Q')$ to a {\LieInftyAlgebroid} $ (E,Q)$, 
with sheaves of functions $\mathcal{E'}$ and $ {\mathcal E}$, respectively, is a graded commutative algebra morphism $\Phi \colon {\mathcal E}\to \mathcal{E}'$  which intertwines $Q$ and $Q'$:
\begin{equation}\label{Qmorphism}
\Phi\circ Q=Q'\circ\Phi.
\end{equation}
\end{definition}

\isthiswhatyouowant
Every {\LieInftyAlgebroid} morphism $\Phi$ induces a smooth map $\phi \colon M' \to M$ that we call the \emph{base morphism}.
We say that a {\LieInftyAlgebroid} morphism $\Phi$ is \emph{over the identity of $M$}, if  the base morphism $\phi$ is the identity map. 
It also induces a graded vector bundle morphism $\phi_0\colon E'_\bullet \to E_\bullet $ over $\phi$ that we call the \emph{linear part of $\Phi$}.
%

\begin{remarque}
	\normalfont
	\label{rmk:linearparts}
	Equation (\ref{Qmorphism}), 
	restricted to terms of arity $0$, implies that 
	the linear part $\phi_{0}$ of a  {\LieInftyAlgebroid} morphism $\Phi$ from $(E',Q')$ to $(E,Q)$
	is a chain map between their respective linear parts:
	\begin{center}
		\begin{tikzcd}[column sep=1.3cm,row sep=1.4cm]
			\ldots\ar[r,"\dd'"] & E'_{-3} \ar[r,"\dd'"]\ar[d,"\phi_0"]& E'_{-2} \ar[d,"\phi_0"]\ar[r,"\dd'"]&  E'_{-1} \ar[d,"\phi_0"]\ar[r,"\rho'"]& TM' \ar[d,"\phi_\ast"].\\
			\ldots\ar[r,"\dd"]& E_{-3} \ar[r,"\dd"]& E_{-2}\ar[r,"\dd"]&  E_{-1}\ar[r,"\rho"]& TM
		\end{tikzcd}
	\end{center}
\end{remarque}

\isthiswhatyouowant
An ${\mathscr O}$-linear map from $ {\mathcal E}:=\Gamma\big(S(E^*)\big)$ to ${\mathcal E}':=\Gamma\big(S(E'^*)\big)$, which is not necessarily a {\LieInftyAlgebroid} morphism, is said 
to be of \emph{arity/degree $k$}
if it maps functions of arity/degree $l$ in $\mathcal{E}$  to functions of arity/degree $l+k$  in $\mathcal{E}'$. 
By construction, the component $ \Phi^{(k)}$ is such that the arity of $\Phi^{(k)}(F) $ is $k+l$ for every function $F$ of arity $l$.
In particular, $\Phi$ can be decomposed into components according to  arity, 
which allows us to consider $\Phi$ as a formal sum:
\begin{equation}
\Phi=\sum_{k \in {\mathbb Z}}\ \Phi^{(k)}\, .
\end{equation}
The component $ \Phi^{(k)}$ of arity $k\geq 0$ maps $\Gamma(E^*) $ to  $ \Gamma\big(S^{k+1}(E'^*) \big)$. By $\cO$-linearity, it gives rise to a section of $S^{k+1}(E'^*)  \otimes E$
that we denote by $\phi_k$ and call the \emph{$k$-th Taylor coefficient of $\Phi$}. For all $\alpha \in \Gamma(E^*)$ one has, by definition, 
\begin{equation} \Phi^{(k)}(\alpha) = \langle \phi_{k}, \alpha\rangle. \label{Taylor}\end{equation} It deserves to be noticed that $\Phi\colon  {\mathcal E} \to  {\mathcal E}'$ is a graded morphism of algebras if and only if $ \Phi^{(i)}=0$ for all $i < 0$ and 
for all $k,n \in {\mathbb N}$ and all  $\alpha_1, \dots, \alpha_k \in \Gamma(E^*)$ one has:
\begin{equation}
\label{boum}
\Phi^{(n)}(\alpha_1 \odot \dots \odot \alpha_k)\ =\sum_{i_1+\dots+i_k =n } \Phi^{(i_1)}(\alpha_{1})  \odot \cdots \odot \Phi^{(i_k)}(\alpha_{k}).
\end{equation}
$\Phi$ is a morphism of $N$-manifolds if, in addition, it has degree zero. Evidently, every algebra morphism is determined uniquely by its Taylor coefficients. 

\isthiswhatyouowant The following  notions will be of importance later on. They show  in part similarities with the AKSZ-formalism \cite{AKSZ,Roytenberg2,Kotov} as well as the formalism developed in \cite{Bojowald}.

\isthiswhatyouowant
Let $(E,Q)$ and $(E',Q')$ be two {\LieInftyAlgebroid}s over $M$ with sheaves of functions $ {\mathcal E}$ and $ {\mathcal E}'$, respectively.
We define a degree one operator on the space of linear maps $\mathrm{Lin}(\mathcal{E},\mathcal{E}')$ from $\mathcal{E}$ to $\mathcal{E}'$ by
\begin{align*}
\Qtot \colon \mathrm{Lin}(\mathcal{E},\mathcal{E}') & \to \mathrm{Lin}(\mathcal{E},\mathcal{E}'),\\
\Psi\hspace{0.5cm} &\mapsto  Q'\circ\Psi-(-1)^{|\Psi|}\Psi\circ Q 
\end{align*}
for every map of graded manifolds $\Psi\colon\mathcal{E}\to\mathcal{E}'$ of homogeneous degree $|\Psi|\in\mathbb{Z}$. 
This operator squares to zero because both vector fields are homological. 
It generalizes the BV-operator in the AKSZ-construction. One evidently has the following statement:
\begin{lemme} Let $(E,Q)$and $ (E',Q')$ be Lie $\infty$-algebroids.  A graded algebra morphism $ \Phi \colon {\mathcal E} \to {\mathcal E}'$ is a Lie $\infty$-algebroid morphism if and only if it is a degree zero  $\Qtot$-cocycle.
\end{lemme}

\begin{definition}
	For every graded algebra morphism $\Phi \colon\mathcal{E}\to\mathcal{E}'$,  a homogeneous map $W \colon\mathcal{E}\to\mathcal{E}'$ of degree $k$ which satisfies
	\begin{equation}
	\label{eq:PhiDiff}
	W(F\odot G) = W(F) \odot \Phi (G) + (-1)^{k|F|} \Phi (F) \odot W (G)
	\end{equation}
	for all homogeneous functions $F,G \in {\mathcal E}$ is called a \emph{$\Phi$-derivation of degree $k$}.
	 We denote the space of $\Phi$-derivations  by $ {\mathfrak X}({\mathcal E} \stackrel{\Phi}{\to}  {\mathcal E}')$ and its restriction to $\cO$-linear ones  by $ {\mathfrak X}_{\mathrm{vert}}({\mathcal E} \stackrel{\Phi}{\to}  {\mathcal E}')$.
\end{definition}
\isthiswhatyouowant To give an example:
\begin{lemme} For every graded algebra morphism  $\Phi\colon  {\mathcal E} \to {\mathcal E}'$ of degree $k$:
	\label{lem:DeltaPhi}
	\begin{enumerate}
	\item  $	{\Qtot (\Phi)}$ is a $ \Phi$-derivation of degree $k+1$.
	\item 	${\Qtot (\Phi)}$ is ${\mathcal O} $-linear if $\Phi$ is $\cO$-linear and $\rho' = \rho \circ \phi_0$. 
\end{enumerate}
\end{lemme}

\begin{proof}
 The first item follows from a straightforward computation:
	\begin{eqnarray}
\label{eq:quasider0}
	{\Qtot (\Phi)}(F  \odot G)& =& Q'\left( \Phi (F)  \odot \Phi(G) \right)  -\Phi\left( Q(F)  \odot G + (-1)^{|F| } F  \odot Q(G) \right)   \\
	& =& Q'\circ  \Phi (F)   \odot  \Phi (G)   - \Phi \circ Q(F)  \odot \Phi( G) \nonumber \\& & + (-1)^{|F|} \Phi(F)  \odot \left(  Q' \circ  \Phi(G) - \Phi \circ  Q (G)   \right) \nonumber \\ &=&
	 {\Qtot (\Phi)}(F)  \odot \Phi(G) - (-1)^{|F|} \Phi(F) \odot  {\Qtot (\Phi)} (G), \nonumber
	\end{eqnarray}
	for all degree-homogeneous $F,G \in {\mathcal E}$.  The second item follows from
		$$\begin{array}{rcll}
		{\Qtot (\Phi)}(f)& =&   Q'(f) -  \Phi \circ Q(f) & \hbox{(since $\Phi(f)=f$) } \\
	& =&   (\rho')^* (\dd f) -  \Phi^{(0)} \circ \rho^* (\dd f) & \hbox{($\Phi$ and $ \Phi^{(0)}$ coincide on $\Gamma(E^*_{-1})$) }\\
	&=&  (  \rho' - \rho \circ \phi_0 )^*(\dd f)& \hbox{(by duality)} \\
	&=& 0. & 	\end{array} $$
\end{proof} 
 \isthiswhatyouowant  Such as algebra morphisms, also a $\Phi$-derivation is determined uniquely by its Taylor coefficients $w_{i} \in \Gamma\left(S^{i+1}(E'^*) \otimes E\right)$, where now $i \geq -1$:
\begin{equation}
\label{boum2} W^{(n)}  (\alpha_1 \odot \cdots \odot \alpha_k) =  \sum_{j=1}^k \sum_{i_1+ \dots+  i_k = n } \epsilon_j \, \Phi^{(i_1)} (\alpha_1) \odot \cdots \odot \langle w_{i_j},\alpha_j\rangle \odot \cdots \odot   \Phi^{(i_k)} (\alpha_k)
\end{equation} 
with $ \epsilon_j =  (-1)^{|W|(|\alpha_1|+ \dots + |\alpha_{j-1}|)}$; again, for every $\alpha \in \Gamma(E^*)$, 
$\langle w_k,\alpha \rangle := W^{(k)}(\alpha)$. Conversely, given an arbitrary graded algebra morphism $\Phi \colon {\mathcal E}\to {\mathcal E}'$ and any section $w \in \Gamma\left(S^\bullet(E'^*) \otimes E\right)$, there is a unique $\Phi$-derivation $W=: w^\Phi$ whose arity $n$ component satisfies Equation \eqref{boum2} where $w_k$ is the restriction of $w$ to $\Gamma\left(S^{k+1}(E'^*) \otimes E\right)$. 


\isthiswhatyouowant
Now let $\Phi$ be a  {\LieInftyAlgebroid} morphism from $(E',Q')$ to $(E,Q)$. 
By a straightforward computation one shows that  for every $\Phi$-derivation $W$ of degree $k$, the linear map $\Qtot(W)$ is a ${\Phi}$-derivation 
of degree $k+1$.  This implies in particular:

\begin{lemme}
	\label{lem:complexPhideriv}
	For every {\LieInftyAlgebroid} morphism $\Phi$  from $(E',Q')$ to $(E,Q)$, the graded space $  {\mathfrak X}({\mathcal E} \stackrel{\Phi}{\to}  {\mathcal E}')$ of $\Phi$-derivations is a complex when equipped with the differential $\Qtot$. $ ({\mathfrak X}_{\mathrm{vert}}({\mathcal E} \stackrel{\Phi}{\to}  {\mathcal E}'),\Qtot)$ is a subcomplex.
\end{lemme}


	\subsubsection{Homotopies between morphisms of {\LieInftyAlgebroid}s}


\isthiswhatyouowant
Let us define what we mean by a piecewise-{\smooth} path valued in {\LieInftyAlgebroid} morphisms from $(E',Q')$ to $(E,Q)$. 
A \emph{piecewise-{\smooth}}
\emph{path valued in  $\Gamma(B)$}, with $B$ a vector bundle over $M$,
is a map  $\psi \colon M \times I \to B$ to the manifold $B$ such 
that for all fixed $t \in I\equiv [0,1]$, the map $ m \mapsto \psi(m,t)$ is a section of $B$ 
and there exists a subdivision $a=t_0 < \dots < t_k=b$ of $I=[a,b]$ such that the map $\psi \colon M \times ]t_{i},t_{i+1} [\to B$
is of class {\smooth}.

\begin{definition}
\label{def:pickyAboutPaths}
Let $(E,Q)$ and $(E',Q')$ be {\LieInftyAlgebroid}s over $M$.
A path $t\mapsto  \Phi_t$ valued in {\LieInftyAlgebroid} morphisms from $E'$ to $E$
is said to be  \emph{continuous piecewise-{\smooth}} 
when for all $k \in {\mathbb N}$, its Taylor coefficients $ t \mapsto \phi_k (t)$ of arity $k$ is a piecewise-{\smooth}  path  valued
in $\Gamma\big(S^{k+1}\big(E'^*\big) \otimes E \big)$, which is also continuous---even at the junction points..
\newline
\indent Given a piecewise-{\smooth} path $t \mapsto  \Phi_t$ valued in {\LieInftyAlgebroid} morphisms from $(E',Q')$ to $(E,Q)$,
we say that a path $t \mapsto H_t $, with $H_t$ a $\Phi_t$-derivation,  is \emph{piecewise smooth}
if its Taylor coefficients $ t \mapsto h_k (t) $ of arity $k$
is a piecewise-smooth path valued in $\Gamma\big(S^{k+1}\big(E'^*\big) \otimes E \big)$.
\end{definition}

\begin{remarque}
\normalfont
A subtle point in this definition is that the subdivision of $ I$ with respect to which $\phi_k(t)$ is piecewise-$ C^\infty$ may depend on $k$.
The derivative $\tfrac{d}{dt} \Phi_t $ is well-defined for all $ t \in I$  which are not in the countable set of points delimiting all these subdivisions. 
For Lie $n$-algebroids, since the components of arity $k$ of $\big(S(E'^*)\otimes E\big)_{0}$ vanish for $k$ large enough, this subdivision of $I$
can be chosen to be the same for all values of $k \geq 0$.
\end{remarque}

\isthiswhatyouowant
It is routine to check that $ \tfrac{\dd }{\dd t} \Phi_t$ is a $\Phi_t$-derivation of degree $0$ for each value of $t$ for which it is defined: it satisfies $\Qtot \left( \tfrac{\dd}{\dd t} \Phi_t\right)=0$,
i.e.~it is a cocycle for the complex of Lemma \ref{lem:complexPhideriv} for $\Phi$ replaced by $\Phi_t$. This justifies the following definition, whose rough idea is that homotopies are curves of {\LieInftyAlgebroid} morphisms whose derivatives are coboundaries
for the complex of $\Phi$-derivations:

\begin{definition}
\label{def:homotopy}
Let $\Phi$ and $\Psi$ be two {\LieInftyAlgebroid} morphisms from $(E',Q')$ to $(E,Q)$ covering the identity morphism. A \emph{homotopy between 
$\Phi$ and $\Psi$} is a pair $(\Phi_{t},H_{t})$ consisting of:
\begin{enumerate}
\item a continuous piecewise-$C^{\infty}$ path $t\mapsto\Phi_t$ valued in {\LieInftyAlgebroid} morphisms between $E'$ and $E$ such that:
\begin{equation*}
\Phi_0=\Phi \qquad {\hbox{ and }} \qquad  \Phi_1=\Psi,
\end{equation*}
\item a piecewise smooth path $t\mapsto H_t $ valued in $\Phi_t$-derivations of degree $-1$, such that the following equation:
\begin{equation}\label{Hortense}
\frac{\dd\Phi_t}{\dd t}=\Qtot \left( H_{t} \right) \equiv Q' \circ H_t + H_t \circ Q 
\end{equation}
holds for every $t\in\,[0,1]$ where it is defined.
\end{enumerate}
\end{definition}

\begin{remarque}
\normalfont
Let us be more precise about the meaning of Equation (\ref{Hortense}): It should be understood as meaning that 
the equality 
$$\frac{\dd\Phi^{(k)}_t}{\dd t}=\left( \Qtot \left( H_{t}\right)\right)^{(k)} = \sum_{i=0}^k \left( (Q')^{(i)} \circ H_{t}^{(k-i)}  + H_{t}^{(k-i)} \circ Q^{(i)}  \right)   $$
holds for every $k \in {\mathbb N}$ and every $t\in\,[0,1]$ where it is defined, i.e.~which are not junction points of the partition of $I$ with respect to which 
$\Phi^{(k)}_t$ is piecewise-{\smooth}.
\end{remarque}

\isthiswhatyouowant
The following fact is obvious:
\begin{proposition}\label{equivhomotopy}
Homotopy of {\LieInftyAlgebroid} morphisms is an equivalence relation, denoted by $\sim$, 
 which is compatible with composition.
\end{proposition}
\begin{proof}
Let us show that homotopy defines an equivalence relation $\sim$ between {\LieInftyAlgebroid} morphisms:
\begin{itemize}
\item \emph{reflexivity}: $\Phi\sim\Phi$, as can be seen by choosing $\Phi_{t}=\Phi$ and $H_{t}=0$ for every $t\in[0,1]$.
\item \emph{symmetry}: $\Phi\sim\Psi$ implies that $\Psi\sim\Phi$ by reversing the flow of time, i.e.~by considering the homotopy $(\Phi_{1-t},-H_{1-t})$.
\item \emph{transitivity}: if $\Phi\sim\Psi$ and $\Psi\sim\Xi$ then there exists a homotopy $(\Theta_{1\, t},H_{1\, t})$ between $\Phi$ and $\Psi$ and a homotopy $(\Theta_{2\, t},H_{2\, t})$ joining $\Psi$ and
$\Xi$. It is then sufficient to glue $\Theta_{1}$ and $\Theta_{2}$ and rescale the time variables, so that the new time variable takes values in the closed interval $[0,1]$. 
(Notice that the resulting function will be continuous at the junction, but not differentiable in general at that point.)
\end{itemize}

\isthiswhatyouowant
Now assume that $\Phi,\Psi\colon \mathcal{E}\to\mathcal{E}'$ are homotopic {\LieInftyAlgebroid} morphisms between $(E',Q')$ and $(E,Q)$, and that $\Phi',\Psi'\colon \mathcal{E}'\to\mathcal{E}''$ are homotopic {\LieInftyAlgebroid} morphisms between $(E'',Q'')$ and $(E',Q')$. Let us denote by $(\Phi_{t},H_{t})$ 
the homotopy between $\Phi$ and $\Psi$, and $(\Phi_{t}',H_{t}')$ the homotopy between $\Phi'$ and $\Psi'$. Then $\Phi'\circ\Phi$ and $\Psi'\circ\Psi$ are 
homotopic via $\big(\Phi_t'\circ\Phi_{t},H_t'\circ\Phi_{t}+\Phi_t'\circ H_{t}\big)$.
\end{proof}

\isthiswhatyouowant
We now give an important example, that shall be used in the sequel:
\begin{example}
\normalfont
\label{ex:homotopies}
Let $(E,Q)$ and $(E',Q')$ be {\LieInftyAlgebroid}s over $M$ and let  $w$ be a section of degree $-1$ of $\Gamma \big(S^{i+1} (E'^*) \otimes E\big)$ for some $i \geq 0$. 
For every {\LieInftyAlgebroid} morphism $\Phi \colon {\mathcal E} \to {\mathcal E}'$ from $(E',Q')$ to $(E,Q)$, we denote the ${\mathscr O}$-linear $\Phi$-derivation with Taylor coefficient $w$ by $w^\Phi$.

\isthiswhatyouowant
 The following differential equation has a unique solution for $\Phi_t$ for all $t \in {\mathbb R}$:
  \begin{equation}\label{eq:eqdiff_wanted} \frac{\dd \Phi_t }{\dd t} = \Qtot \left( w^{\Phi_t}\right)\qquad \hbox{ and } \qquad\Phi_0=\Phi. \end{equation}
   Therefore, the pair $(\Phi_t,w^{\Phi_t})$ is a homotopy between the {\LieInftyAlgebroid} morphism $\Phi$ and
   the {\LieInftyAlgebroid} morphism $\Phi_1$.
  
  \isthiswhatyouowant
  Let us check this point. The differential equation above decomposes into a sequence of coupled differential equations:   For all $    \alpha \in \Gamma(E^*) $   one has
    \begin{equation}
    \label{eq:recursion2}
    \begin{array}{rcl}
    \frac{\dd \Phi_t^{(k)} (\alpha)}{\dd t} &=& \sum_{j=0}^{k} \left( Q'^{(k-j)} \circ \left(w^{\Phi_t}\right)^{(j)} (\alpha)- \left( w^{\Phi_t} \right)^{(j)}  \circ Q^{(k-j)}  (\alpha)\right)  \\   &=&   Q'^{(k-i)} \circ  w \, (\alpha) - \sum_{j=0}^{k-1} \left( w^{\Phi_t}\right)^{(j)}  \circ Q^{(k-j)}  (\alpha)  \\
    & & -  \delta_{i}^k \, \, w \circ Q^{(0)}  (\alpha)  \end{array}
    \end{equation} 
where $\delta_{i}^k$ is the Kronecker symbol.

 \isthiswhatyouowant
  Equation \eqref{boum2} shows that $ (w^{\Phi_t})^{(j)}$ is a fixed  expression involving $w$ and $\Phi_t^{(k')}$ for $k' =0, \dots, j-i$.
     	This implies that the r.h.s. of the differential Equation \eqref{eq:recursion2} does not depend on $ \Phi_t^{(k)}$ but only on $ \Phi_t^{(k')}$ for $ k' \leq k-1$. (In fact, a closer look shows that it only depends on $ \Phi_t^{(k')}$ for $ k' \leq k-1-i$).  Therefore  $\Phi_t^{(k)}(\alpha)$ is obtained by integration of the r.h.s. of $ \eqref{eq:recursion2}$, which itself depends only on  $ \Phi_t^{(k')}$ for $ k' \leq k-1$. Equation \eqref{eq:eqdiff_wanted} can therefore be solved by recursion. 
 Since  for $i \geq 1$  we have that $ \tfrac{\dd \Phi_t^{(0)} }{\dd t}$ vanishes identically
 and for $i=0$ it satisfies
   $$  \frac{\dd \Phi_t^{(0)} (\alpha)}{\dd t}  = Q'^{(0)} \circ w (\alpha) + w \circ Q^{(0)}(\alpha) \quad \forall \alpha \in \Gamma (E^*), $$
   one finds that $\Phi_t^{(k)}(\alpha)$ is a polynomial in $t$ for all $k \geq 0$ and $\alpha \in \Gamma(E^* )  $. Equation \eqref{eq:eqdiff_wanted} therefore has solutions defined on all of ${\mathbb R} $.
\end{example}

\isthiswhatyouowant
The importance of Definition \ref{def:homotopy} relies on the following result, which states that two homotopic {\LieInftyAlgebroid} morphisms are related by a $\Qtot$-exact term:
\begin{proposition}
\label{prop:HomotopyMeansHomotopy}
Let $(E,Q)$ and $(E',Q')$ be {\LieInftyAlgebroid}s over $M$.
For every two homotopic  Lie $\infty $-morphisms $\Phi$ and $\Psi$ from $(E',Q')$ to $(E,Q)$, there exists an ${\mathscr O}$-linear map
$ H \colon {\mathcal E } \to {\mathcal E}'$ of degree $-1 $  such that:
 \begin{equation}\label{tolley} \Psi - \Phi = \Qtot \left( H \right) \equiv Q' \circ H + H \circ Q . \end{equation}
\end{proposition}
\begin{proof}
We shall use the following property: the variation of a continuous piecewise-$C^{\infty}$ function is equal to the integral of its derivative. 
From the relation $  \frac{\dd }{\dd t} \Phi_t = \Qtot \left( H_t \right) $ and from the fact that the path $t \mapsto \phi_k (t) $ is continuous piecewise-$C^{\infty}$ for all $k \in {\mathbb N}$, 
we therefore obtain:
\begin{eqnarray*}  \Psi - \Phi &=& \int_0^1 \frac{\dd }{\dd t} \Phi_t \, \dd t = \int_0^1 \Qtot \left( H_t \right) \, \dd t \\
&=& \int_0^1 \big(Q' \circ H_t + H_t \circ Q \big)\, \dd t \\ 
&=&  Q' \circ \left( \int_0^1 H_t \,\dd t \right)  + \left(\int_0^1 H_t\, \dd t \right)\circ Q   \end{eqnarray*}
Hence $H= \int_0^1 H_t\, \dd t $ satisfies Condition \eqref{tolley}. Also, $H$ is ${\mathscr O}$-linear because so is $H_{t}$ for all $t \in [0,1]$.
\end{proof}

\isthiswhatyouowant
It deserves to be noticed that the map $H$ introduced in the Proposition \ref{prop:HomotopyMeansHomotopy} is, in general, neither an algebra morphism nor a derivation of any sort.

\begin{remarque}
\normalfont
Taking the arity $0$ part of Equation \eqref{tolley}, one finds a homotopy of the two underlying chain maps: 
 \begin{center}
\begin{tikzcd}[column sep=1.3cm,row sep=0.7cm]
\ldots\ar[r,"\dd"] & E'_{-3} \ar[r,"\dd"]\ar[dd,shift left =0.5ex,"\phi_0"]\ar[dd,shift right =0.5ex,"\psi_0" left]& E'_{-2} \ar[ddl,dashed, "h\ " above left] \ar[dd,shift left =0.5ex,"\phi_0"]\ar[dd,shift right =0.5ex,"\psi_0" left]\ar[r,"\dd"]& E'_{-1} \ar[ddl,dashed,"h\ " above left] \ar[dd,shift left =0.5ex,"\phi_0"]\ar[dd,shift right =0.5ex,"\psi_0" left]\\
&\\
\ldots\ar[r,"\dd'"]& E_{-3} \ar[r,"\dd'"]& E_{-2}\ar[r,"\dd'"]&  E_{-1}.
\end{tikzcd}
\end{center}
Above, $\phi_0$ and $\psi_0$ are the linear parts of $\Phi$ and $\Psi$, respectively, and $h$ is the dual to the component of arity $0$ of $H$.
\end{remarque}

\isthiswhatyouowant
We now define what we mean by a homotopy equivalence of {\LieInftyAlgebroid}s:
\begin{definition}
	\label{def:homtequiv}
Let $(E,Q)$ and $(E',Q')$ be two {\LieInftyAlgebroid}s over $M$ and $\Phi\colon \mathcal{E}'\to\mathcal{E}$ a {\LieInftyAlgebroid} morphism between them.
We say that $\Phi$ is a \emph{homotopy equivalence} if there exists a {\LieInftyAlgebroid} morphism $\Psi\colon \mathcal{E}\to\mathcal{E}'$
such that
\begin{equation*}
\Phi\circ\Psi\sim\mathrm{id}_{\mathcal{E}}\hspace{1cm}\text{and}\hspace{1cm}\Psi\circ\Phi\sim\mathrm{id}_{\mathcal{E}'}.
\end{equation*}
In such a case,  the {\LieInftyAlgebroid}s $(E,Q)$ and $(E',Q')$ are said to be \emph{homotopy equivalent}.
\end{definition}

\subsubsection{Comparison with cylinder homotopies}

	Although it may seem quite different at first look, Definition \ref{def:homotopy} is in fact very similar 
	to a more classical and natural definition of homotopy given  by \cite{Bojowald,valette},  that we shall refer as the \emph{cylinder homotopy}. The only difference with our definition lies in an important relaxation of the regularity conditions.
	It consists in defining homotopies between two morphisms as being {\LieInftyAlgebroid} morphisms of differential graded algebras from ${\mathcal E}$ to the tensor product
	$ {\mathcal E}'  \otimes \Omega^\bullet \big([0,1]\big) $, where $\Omega^\bullet \big([0,1]\big) $ stands for forms on $[0,1]$, equipped with de Rham differential, 
	whose restrictions to $\{0\}$ and $\{1\}$ are the two given {\LieInftyAlgebroid} morphisms.
 Equivalently, a cylinder homotopy between two {\LieInftyAlgebroid} morphisms $\Phi$ and $\Psi$ from $(E,Q)$ to $(E',Q')$ is a {\LieInftyAlgebroid} morphism  $ (TI, \dd_{dR}) \times (E,Q) \mapsto (E',Q') $ whose restrictions to $\{0\} \times (E,Q) $
		and $ \{1\} \times (E,Q) \ $ are $ \Phi$  and $\Psi$, respectively.
	
	\isthiswhatyouowant
		\begin{proposition}
			\label{homotopy:usualNotion}
			Cylinder homotopies are homotopies $(\Phi_t,H_t)$ as in Definition \ref{def:homotopy} that depend smoothly on the parameter $t$.
		\end{proposition}
		
	\isthiswhatyouowant
	There is, however, a technical issue in the proof of Theorem \ref{theo:onlyOne} that imposes the need to  use of continuous piecewise {\smooth}-paths.
		
	\begin{proof}
	\isthiswhatyouowant
	Let us explain the correspondence between both definitions.
	Let $(E,Q)$, $(E',Q')$, and $(\Phi_t,H_t)$ be as in Definition \ref{def:homotopy}. 
	
	\isthiswhatyouowant
	Let us equip the tensor product
	$ {\mathcal E}'  \otimes \Omega^\bullet\big([0,1]\big)$
	with the differential $D$ given for all $F \in {\mathcal E}_i'$  and $\omega \in  \Omega^\bullet\big([0,1]\big)$ by
	$$ D \colon F \otimes \omega \mapsto Q'(F) \otimes \omega - (-1)^{|F|} F \otimes {\rm d}_{\text{dR}} \omega .$$
	The graded commutative algebra ${\mathcal E}'  \otimes \Omega^\bullet\big([0,1]\big)$ can be identified with the algebra made of sums 
	$   F_t + G_t \, \epsilon$
	with $F_t,G_t$ families of elements in ${\mathcal E}'$ depending smoothly on the parameter $t \in [0,1]$  and $\epsilon$ some free parameter of degree $+1$ that squares to $0$ (that we invite the reader to think of it as being "dt").
	The product in ${\mathcal E}'  \otimes \Omega^\bullet\big([0,1]\big)$  is then given by $ (F_t + G_t \, \epsilon ) (\tilde{F}_t + \tilde{H}_t \, \epsilon) = F_t \tilde{F}_t + \big( F_t \tilde{G}_t + \tilde{F}_t G_t \big) \epsilon  $.
	Also, the operator $D$ is given for all $F_t + G_t \epsilon$ of degree $i$, by:
	\begin{equation}
	\label{eq:homotopiesExplicit}
	D \big(  F_t + G_t \epsilon \big)   = \big(Q' ( F_t )\big)  +  \left(-  (-1)^{i} \frac{\dd F_t}{\dd t} + \big(Q' (G_t) \big) \right) \epsilon .
	\end{equation}
	Consider the map of degree $0$ given by:
	$$  \begin{array}{rrcl} \widetilde{\Phi}:= &{\mathcal E} & \to & {\mathcal E}'  \otimes \Omega^\bullet\big([0,1]\big)  \\
	&  F & \mapsto  & t \mapsto \Phi_t(F)  + (-1)^{|F|} H_t(F) \epsilon .
	\end{array} $$
	This map is a graded algebra morphism. This follows from the fact that $\Phi_t$ is an algebra morphism and $H_t$ is a $\Phi_t$-derivation for all $t$,
	as can be seen by a direct computation, valid for all $F \in {\mathcal E}_i,G \in {\mathcal E}_j$: 
	\begin{eqnarray*}
		\widetilde{\Phi} (FG) &=&  \Phi_t(FG)  + (-1)^{i+j} H_t(FG) \,  \epsilon \ \\
		& = &  \Phi_t(F) \Phi_t(G)  + \big( (-1)^{i+j} H_t(F) \Phi_t(G) + (-1)^j \Phi_t(F) H_t(G)\big) \,  \epsilon  \\
		& = & \big(\Phi_t(F)  + (-1)^i H_t(F) \,  \epsilon \big)  \, \cdot \,  \big(\Phi_t(G)  +  (-1)^j H_t(G) \,  \epsilon\big)\\
		& = & \widetilde{\Phi}(F) \, \cdot \,  \widetilde{\Phi}(G).
	\end{eqnarray*}
	Equation (\ref{Hortense}) holds,
	as can be seen by the following computation: 
	$$\left\{
	\begin{array}{l} \widetilde{\Phi} \circ Q(F) =  \Phi_t \circ Q (F) + (-1)^i H_t \circ Q ( F) \, \epsilon \\
	D \circ  \widetilde{\Phi}  (F)   = Q \circ \Phi_t  (F)   -  (-1)^{|i|}  \, \frac{\dd \Phi_t(F)}{\dd t} \epsilon + (-1)^i \, Q \circ H_t( F) \, \epsilon .\end{array}\right.$$ 
	As a consequence, $ \widetilde{\Phi} $ is a chain map.  
	Since $\widetilde{\Phi}$ is a graded algebra morphism and a chain map, this implies that the data of Definition \ref{def:homotopy} induces, when it is smooth, a 
	cylinder homotopy. The converse goes by going backward in the previous computations and  proves the equivalence of 
	the cylinder homotopy with homotopies as in Definition \ref{def:homotopy} given by smooth data. 
\end{proof}

\isthiswhatyouowant
	For a more enhanced discussion  about this more restricted notion of homotopy of {\LieInftyAlgebroid} morphisms, we refer to \cite{Poncin} or  \cite{valette}.

	\begin{proposition}
		The equivalence relation given by cylinder homotopies and the equivalence relation given by homotopies as in Definition \ref{def:homotopy} coincide on morphisms of Lie $\infty$-algebroids constructed on complexes of finite length.
	\end{proposition}
	\begin{proof} 
		Let $(\Phi_t,H_t)$ be a homotopy as in Definition \ref{def:homotopy} relating two morphisms $\Phi_0,\Phi_1$, defined on complexes of finite length. Then only finitely many of the Taylor coefficients of $\Phi_t$ and $H_t$ are different from zero. The set $S$ of points in $I=[0,1]$ at which $\Phi_t$ or $H_t$ are non-smooth is therefore finite. Let $f \colon I \to I$ be a  smooth function such that $f(0)=0$ and $ f(1)=1$ and such that $ f^{(k)}(t)=0 $ for all $k \geq 1$ and $t \in S$. The pair   $(\Phi_{f(t)},f'(t) H_{f(t)})$ is then a homotopy given by smooth data relating $\Phi_0$ and $ \Phi_1$.
		The result then follows from Proposition \ref{homotopy:usualNotion}.
	\end{proof}

	\subsubsection{The homotopy and the isotropy functors} 
	\label{sec:functor}
	
Let $M$ be a manifold.
For every $k\in \mathbb{N}\cup \infty$, Lie $k$-algebroids over $M$ together with Lie $\infty$-algebroid morphisms  of $M$ form a category that we denote by {\bf Lie-$k$-algoid}. In this article, we rather consider the \emph{quotient category} {\bf hLie-$k$-algoid$_M$} where objects are Lie $k$-algebroids over $M$ and arrows are homotopy classes of morphisms. 

\begin{remarque}
	\normalfont
	In this language, Theorem \ref{theo:onlyOne} translates into: Universal Lie $\infty$-algebroids of $ {\mathcal F}$ are terminal objects in the sub-category of {\bf hLie-$\infty$-algoid$_M$} of Lie $\infty$-algebroids for which the image of the anchor map is contained in ${\mathcal F}$.
\end{remarque}


\isthiswhatyouowant
When the manifold is  a point, we recover the usual categories of Lie $\infty$-algebras defined over negatively graded vector spaces. We denote by {\bf Lie-$k$-alg} and {\bf hLie-$k$-alg} the counterpart of the above-defined categories in this case.

\isthiswhatyouowant
Every point $m \in M$ induces a functor 
${\mathfrak I}_m\colon \hbox{{\bf Lie-$k$-algoid$_M$}} \mapsto \hbox{{\bf Lie-$k$-alg}}$,  called the isotropy functor, that we now describe.

\begin{enumerate}
	\item Let $(F,Q_F)$ be a Lie $k$-algebroid over $M$ with anchor $\rho$.
	According to the axioms of Lie $k$-algebroids,  the $k$-ary bracket restricts to the graded vector space
	\begin{equation} \label{eq:inducedBraket}  K^\bullet(F,m) := {\mathrm{Ker}} (  \rho|_m)\oplus\bigoplus_{i \geq 2}   F_{-i}|_m  \end{equation}
	for every $k \in {\mathbb N}$, see Definition \ref{def:Linftyoids}, Equation \eqref{robinson}, and Remark \ref{remarque2}. This implies that \eqref{eq:inducedBraket}  is  equipped with a Lie $k$-algebra structure. This defines ${\mathfrak I}_m$ on objects.
	\item Let  $\Phi \colon \Gamma(S(F^*)) \mapsto \Gamma(S(E^*))$  be an arbitrary Lie $k$-algebroid morphism from $ (E,Q_E) $ to $ (F,Q_F)$. 
Since it is ${\mathcal O}$-linear, it induces a graded Lie algebra morphism $\Phi|_m \colon S(F|_m^*) \mapsto S(E|_m^*)$. The linear part of $\Phi$  being a chain map  by Remark \ref{rmk:linearparts}, it restricts to
a graded algebra morphism $ {\mathfrak I}_m(\Phi) \colon S(K(F,m)^*) \mapsto S(K(E,m)^*)$, which is easily checked to be a Lie $\infty$-algebra morphism. This defines
 $ {\mathfrak I}_m$ on arrows.
\end{enumerate}
 
\isthiswhatyouowant
Homotopies between morphisms being given by ${\mathcal O}$-linear data $(\Phi_t,H_t)$, the following lemma is easily verified:

	\begin{lemme}
		\label{lem:homtFunctr}
		Let  $\Phi,\Psi \colon (E,Q) \mapsto (E',Q')$ be two  homotopic Lie-$\infty$ algebroid morphisms over $M$. For every point $m \in M$, 
		${\mathfrak I}_m(\Phi), {\mathfrak I}_m(\Psi) \colon {\mathfrak I}_m(E,Q) \mapsto {\mathfrak I}_m(E',Q')$ are homotopic Lie-$\infty$ algebra morphisms.
	\end{lemme}
		

\isthiswhatyouowant
This Lemma implies that the isotropy functor passes to the quotient  
to yield a functor
\begin{equation}
\label{eq:homotopyFunctor}
h{\mathfrak I}_m\colon \hbox{{\bf hLie-$k$-algoid$_M$}} \mapsto \hbox{{\bf hLie-$k$-alg}},
\end{equation}
that we call the \emph{isotropy functor}.

\isthiswhatyouowant
Let us finish this discussion with a few words on invertible arrows in these categories. 
Homotopy equivalences, see Definition \ref{def:homtequiv}, are invertible arrows in the category {\bf hLie-$k$-algoid$_M$}.
An invertible arrow in the category {\bf hLie-$k$-alg} relating two objects $(V,Q_V)$ and $(W,Q_W)$ is a differential graded algebra isomorphism of the corresponding symmetric algebras, $ S(V^*) \simeq S(W^*)$.

\subsection{Proof of Theorem \ref{theo:existe} about existence} 

\label{sectionpreuve1}
\isthiswhatyouowant

\isthiswhatyouowant
 Throughout this subsection we assume that we are working in the smooth setting but the arguments below also work in the real analytic or holomorphic case in a neighborhood of a point.

\isthiswhatyouowant	To make the result global in these cases, we would need to start with a singular foliation which admits a global geometric resolution $ (E,\dd,\rho)$, but, at least the way how the proof is performed, we would also have to assume that all vector bundles in this geometric resolution admit connections,  and, in particular, that an almost Lie algebroid structure can be constructed globally on $E_{-1}$. Then a step-by-step verification of the arguments presented here would lead to the same conclusion.


\subsubsection{Arity-deformation of complexes of vector bundles}
\label{sec:aritydefo}
\isthiswhatyouowant
A complex of vector bundles 
 \begin{equation} \label{sequencelinear1}
\begin{tikzcd}[column sep=0.9cm,row sep=0.6cm]
\ldots\ar[r,"\dd"]& E_{-3} \ar[r,"\dd"]& E_{-2}\ar[r,"\dd"]&  E_{-1}
\end{tikzcd}
\end{equation}
is in one-to-one correspondence with an NQ-manifold where $Q$ is of arity zero, $Q=Q^{(0)}$, see in particular Equation \eqref{Q0}.   In this subsection, we discuss deformations of such a complex inside the category of NQ-manifolds following the general expansion  \eqref{Qarity}: Theorem \ref{theo:existe} will then follow from this discussion. The homological condition $[Q,Q]=0$ is equivalent to the following set of equations:
\begin{align}
[Q^{(0)},Q^{(0)}]&=0\label{jacques00}\\
[Q^{(0)},Q^{(1)}]&=0\label{jacques0}\\
\forall\ n\geq2\hspace{1.3cm}[Q^{(0)},Q^{(n)}]&=-\frac{1}{2}\underset{i+j=n}{\sum_{1\leq i, j\leq n-1}} [Q^{(i)},Q^{(j)}].\label{jacques}
\end{align}
The first one is valid by assumption, the second of these equations determines the tangent vector of the deformation. It has to be a $Q^{(0)}$-closed, arity one, and degree one vector field. Such a vector field $Q^{(1)}$ corresponds to the following data: An anchor map $\rho \colon E_{-1} \to TM$ and a family of binary brackets on the sections of $E$ such that the Leibniz identity \eqref{robinson} holds true and $\dd$ is a derivation of the bracket. 

\isthiswhatyouowant
Now we assume that $E$, $Q^{(0)}$, and such a $Q^{(1)}$ are given. Then there is a standard theory of obstruction classes entailed by Equations \eqref{jacques}: The deformation problem is governed by the differential graded Lie algebra (DGLA) $$(\mathfrak{X}(E), [ \cdot , \cdot ], [Q^{(0)}, \cdot ]). $$  While the right hand side of Equation  \eqref{jacques} is always $Q^{(0)}$-closed by the Jacobi identity of the graded Lie bracket of vector fields, according to this equation it has to be exact so as to extend the procedure one step further. Thus the obstruction for finding $Q^{(n)}$ lives in $H^2(\mathfrak{X}(E),\ad_{Q^{(0)}})$, the cohomology  of $\ad_{Q^{(0)}}\equiv [Q^{(0)}, \cdot ]$ at degree two. 

\isthiswhatyouowant
What we said so far is valid for every DGLA. In the particular context of NQ-manifolds, it turns out to be useful to relate this problem to the following differential graded Lie subalgebra 
$$(\mathfrak{X}_{\mathrm{vert}}(E), [ \cdot , \cdot ], [Q^{(0)}, \cdot ]) $$ 
of vertical vector fields. By definition, \emph{vertical vector fields} are $\mathcal{O}$-linear derivations of $\mathcal{E}$. Now it is an important observation that within the expansion \eqref{Qarity} only the term within $Q^{(1)}$ which specifies the anchor  is non-vertical; all the other parts of $Q$ are vertical vector fields. Moreover, the Lie bracket of $Q^{(1)}$ with $Q^{(n)}$ is vertical whenever $n \geq 2$. Thus \emph{all} the obstruction classes for $n$ strictly bigger than two actually live inside 
$ H^2(\mathfrak{X}_{\mathrm{vert}}(E),\ad_{Q^{(0)}})$. In fact, noting that the Lie bracket preserves arity and the arity of the right hand side of Equation \eqref{jacques} is  $n$, for every $n> 2$ the cohomological obstruction lives in 
\begin{equation}\label{H2vert}
H^2(\mathfrak{X}^{(n)}_{\mathrm{vert}}(E),\ad_{Q^{(0)}}) \, .
\end{equation}
Here $\mathfrak{X}^{(n)}_{\mathrm{vert}}(E)$ denotes vertical vector fields of arity $n$.

\isthiswhatyouowant
The case $n=2$ needs special care. The condition $[Q^{(0)},Q^{(2)}]= \tfrac{1}{2}[Q^{(1)},Q^{(1)}]$ can be split into two parts: Since the right hand side of this equation is vertical, one needs 
\begin{equation} \label{Q1Q1}
[Q^{(1)},Q^{(1)}] \in \mathfrak{X}_{\mathrm{vert}}(E) \, . \end{equation}
If this is satisfied, then its class inside \eqref{H2vert} is defined and needs to vanish.

\subsubsection{Cohomology of vertical vector fields for geometric resolutions}

\isthiswhatyouowant
We denote the space of vertical vector fields on $E$ of arity $n$ and degree $k$ by $\mathfrak{X}_{\mathrm{vert}}^{(n)}(E)_k$. There is a natural isomorphism
\begin{equation}\label{eq:Phi0AsComplex}
  \mathfrak{X}_{\mathrm{vert}}^{(n)}(E)_\bullet \simeq  \Gamma( S^{n+1}(E^*) \otimes E )_\bullet.
\end{equation}
Consider the following map, which we call the \emph{root map}:
$$ \mathrm{rt} \colon {\mathfrak{X}}_{\mathrm{vert}}^{(n)}(E)_\bullet \to \Gamma( S^{n+1}(E^*))_{\bullet}\otimes_{\mathcal O} {\mathcal F}[-1] $$
which is obtained  by applying $ \id \otimes \rho $ to the component  of a vertical vector field in $ S^{(n+1)} (E^*)_\bullet \otimes E_{-1} $. 
The shift in degree is needed if we want $\mathrm{rt}$ to be of degree zero.
Later, it will be useful to characterize this map also in the following way: Let $f$ denote a function on $M$, $\dd_{\mathrm{dR}}f \in \Gamma(T^*M)$ its differential, and $\rho^{\ast} \colon \Gamma(T^*M) \to \Gamma(E_{-1}^*)$ the dual of the anchor map. Then for every vertical vector field $W$ on $E$, 
one has:
\begin{equation}\label{rootdef}
W(\rho^*\dd_{\mathrm{dR}}f) = \mathrm{rt}(W)[f] \, ,
\end{equation}
since for every $x \in \Gamma(E_{-1})$, $\langle x, \rho^*\dd_{\mathrm{dR}}f \rangle = \rho(x)[f]$.   We will call the image $\mathrm{rt}(W)$ of a vertical vector field $W$ with respect to this map the \emph{root} of this vector field.
\begin{proposition}
	\label{prop:fondamental3}
	If $ (E,\dd, \rho)$ is a geometric resolution of $ {\mathcal F}$, then 
	$$\mathrm{rt} \colon ({\mathfrak{X}}_{\mathrm{vert}}^{(n)}(E)_\bullet,  \ad_{Q^{(0)}} ) \to ( \Gamma( S^{n+1}(E^*))_\bullet \otimes_{\mathcal O} {\mathcal F}[-1] ,  Q^{(0)} \otimes \id) $$
	is a quasi-isomorphism.
\end{proposition}

\isthiswhatyouowant
The proof of this proposition requires some preparation. The following lemma is straightforward to show: We leave the details to the reader, some of which, however, can be found also in the proof of Lemma \ref{rmk:root0} below.
\begin{lemme}
	\label{lem:Phi0AsComplex}
	Under the isomorphism   (\ref{eq:Phi0AsComplex}), the differential $\ad_{Q^{(0)}}$ becomes the total differential of the following bicomplex:	
\begin{equation}
\label{eq:bicomplex2}
 \resizebox{0.88\textwidth}{!}{
 \xymatrix{ & \vdots & \vdots &  \vdots \\ \cdots
	  \ar[r]^{ }  &	\Gamma(S^{n+1}(E^*)_{n+3}  \otimes E_{-3}) \ar[r]^{  {\rm{id}} \otimes \dd } 
	  \ar[u]
	  &\Gamma(S^{n+1}(E^*)_{n+3}  \otimes E_{-2}) \ar[u]
	   \ar[r]^{  {\rm{id}} \otimes \dd }   & \Gamma( S^{n+1}(E^*)_{n+3}  \otimes E_{-1} )  \ar[u] 
	  \ar[r]& 0 \\  \cdots
	  \ar[r]^{ }  &	\Gamma(S^{n+1}(E^*)_{n+2}  \otimes E_{-3}) \ar[r]^{  {\rm{id}} \otimes \dd } 
	  \ar[u]^{  Q^{(0)}  \otimes  {\rm{id}}}  &\Gamma(S^{n+1}(E^*)_{n+2}  \otimes E_{-2}) \ar[u]^{  Q^{(0)}  \otimes  {\rm{id}}}  \ar[r]^{  {\rm{id}} \otimes \dd }   & \Gamma( S^{n+1}(E^*)_{n+2}  \otimes E_{-1} )   \ar[u]^{  Q^{(0)}  \otimes  {\rm{id}} }\ar[r]& 0 \\  \cdots
  \ar[r]^{ }  &	\Gamma( S^{n+1}(E^*)_{n+1} \otimes E_{-3}) \ar[r]^{  {\rm{id}} \otimes \dd } 
   \ar[u]^{  Q^{(0)} \otimes  {\rm{id}}}  &\Gamma(S^{n+1}(E^*)_{n+1} \otimes E_{-2}) \ar[u]^{  Q^{(0)} \otimes  {\rm{id}}}  \ar[r]^{  {\rm{id}} \otimes \dd }   & \Gamma( S^{n+1}(E^*)_{n+1} \otimes E_{-1} )   \ar[u]^{  Q^{(0)}  \otimes  {\rm{id}} } \ar[r]& 0 \\
     &0 \ar[u] & 0 \ar[u]&  0 \ar[u] & }}\end{equation}

\end{lemme}
\isthiswhatyouowant 
For degree reasons $S^{k}(E^*)_{l} = 0$ for $l<k$---the first non-zero row starts at height $n+1$. 

\begin{figure}[ht]
  \centering
  \begin{tikzpicture}[scale=0.50]
    \coordinate (Origin)   at (0,0);
    \coordinate (XAxisMin) at (10,-2);
    \coordinate (XAxisMax) at (0,-2);
    \coordinate (AAxisMin) at (9,-2);
    \coordinate (AAxisMax) at (7,-2);
    \coordinate (YAxisMin) at (10,1);
    \coordinate (YAxisMax) at (10,8);
    \coordinate (ZAxisMin) at (10,-2);
    \coordinate (ZAxisMax) at (10,-1);
    \coordinate (WAxisMin) at (10,-1);
    \coordinate (WAxisMax) at (10,1);
    \draw [ultra thick, black,-latex] (AAxisMax) -- (XAxisMax) node [left] {};
    \draw [ultra thick, black,-latex] (YAxisMin) -- (YAxisMax) node [above] {};
    \draw [ultra thick, black] (ZAxisMin) -- (ZAxisMax);
    \draw [ultra thick, black,dashed] (WAxisMin) -- (WAxisMax);
    \draw [ultra thick, black] (XAxisMin) -- (AAxisMin);
    \draw [ultra thick, black,dashed] (AAxisMin) -- (AAxisMax);

    \clip (-3,-3) rectangle (12cm,9cm); 
    \coordinate (Bone) at (4,2);
    \coordinate (Btwo) at (-6,12);
    \coordinate (B1) at (-7,11);
    \coordinate (B2) at (2,2);
    \coordinate (B3) at (0,4);
    \coordinate (B4) at (-4,8);
    \draw[style=help lines,dashed] (-4,1) grid[step=2cm] (7,9);
    \draw[style=help lines,dashed] (9,2) grid[step=2cm] (10,9);
    \draw (-3,-2) -- (10,-2)[dashed, gray];
    \draw (6,-2) -- (6,-1)[dashed, gray];
    \draw (4,-2) -- (4,-1)[dashed, gray];
    \draw (2,-2) -- (2,-1)[dashed, gray];
    \draw (0,-2) -- (0,-1)[dashed, gray];
    \draw (-2,-2) -- (-2,-1)[dashed, gray];
    \foreach \x in {-4,-3,...,3}{
      \foreach \y in {1,2,...,4}{
        \node[draw,circle,inner sep=1pt,fill] at (2*\x,2*\y) {};
        \node[draw,circle,inner sep=1pt,fill,red] at (Bone){};
        \node[draw,circle,inner sep=1pt,fill,red] at (0,6){};
        \node[draw,circle,inner sep=1pt,fill,red] at (2,4){};
        \node[draw,circle,inner sep=1pt,fill,red] at (-2,8){};
        \node[draw,circle,inner sep=1pt,fill,red] at (-4,10){};
        \node[draw,circle,inner sep=1pt,fill,blue] at (-6,10){};
        \node[draw,circle,inner sep=1pt,fill,blue] at (B2){};
        \node[draw,circle,inner sep=1pt,fill,blue] at (B3){};
        \node[draw,circle,inner sep=1pt,fill,blue] at (B4){};
        \node[draw,circle,inner sep=1pt,fill,blue] at (-2,6){};
      }
    }
    \node [below] at (4,-2)  {$-d+1$};
    \node [below] at (2,-2)  {$-d$};
    \node [above right] at (-2,-2)  {\textbf{depth}};
    \node [right] at (10,2)  {n+1};
    \node [right] at (10,6)  {};
    \node [above left] at (10,8)  {\textbf{height}};
    
    \draw [ultra thick,red] (Btwo)
        -- (Bone);
    \node [above right,red] at (2,4) {\large $[Q^{(0)},Y]$};
    \draw [ultra thick,blue,dashed] (B1)
        -- (B2);
    \node [below left,blue] at (0,4) {\large $Y$};
    
  \end{tikzpicture}
  \caption{\small 
  A vertical vector field $Y$ of arity $n$ and degree $n+1-d$ 
  and its image under the differential $\ad_{Q^{(0)}}\cong Q^{(0)} \otimes  {\rm{id}} +  {\rm{id}} \otimes \dd$}.
  \label{figure1}
\end{figure}
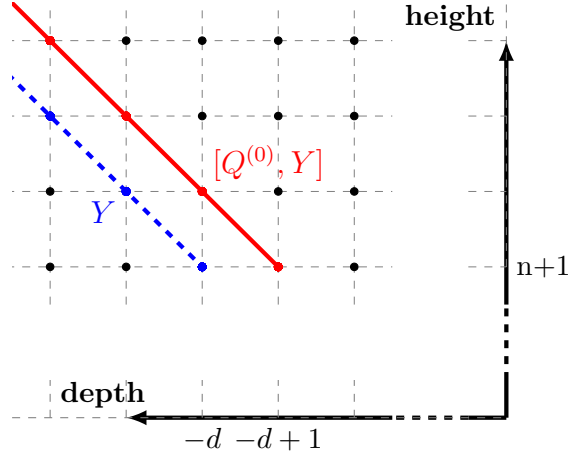

\begin{proof}[Proof of Proposition \ref{prop:fondamental3}] By definition of a geometric resolution, the following is an exact complex of $ {\mathcal O}$-modules:
	$$\xymatrix{	\cdots
		\ar[r]^{ }  &\Gamma( E_{-2}) \ar[r]^{   \dd } 
		& \Gamma(E_{-1}) \ar[r]^{   \rho }   &  {\mathcal F} \ar[r] & 0. 
	}$$
 	Sections of $S^{n+1}(E^*)_k$ form a projective $\mathscr{O}$-module for all $k \geq 0$, thus
	tensoring over $ {\mathcal O}$ with $ \Gamma (S^{n+1}(E^*)_k )  $ preserves exactness:
\begin{equation}
 \label{eq:lines}  \resizebox{0.88\textwidth}{!}{\xymatrix{	\cdots
 	\ar[r]^{ }  &	\Gamma( S^{n+1}(E^*)_{k} \otimes E_{-2}) \ar[r]^{  {\rm{id}} \otimes \dd } 
 	&\Gamma(S^{n+1}(E^*)_{k} \otimes E_{-1}) \ar[r]^{  {\rm{id}} \otimes \rho }   & \Gamma( S^{n+1}(E^*)_{k}) \otimes_{\mathcal O} {\mathcal F} \ar[r] & 0  .}}
\end{equation}
Therefore, all lines in the bicomplex below are exact complexes:
\begin{equation}\label{eq:Bigbicomplex} 
 \resizebox{0.88\textwidth}{!}{\xymatrix{ & \vdots & \vdots &  \vdots & \\ \cdots
	\ar[r]^{ }  &	\Gamma(S^{n+1}(E^*)_{n+3}  \otimes E_{-2}) \ar[r]^{  {\rm{id}} \otimes \dd } 
	\ar[u]
	&\Gamma(S^{n+1}(E^*)_{n+3}  \otimes E_{-1}) \ar[u]
	\ar[r]^{  {\rm{id}} \otimes \rho }   & \Gamma( S^{n+1}(E^*)_{n+3} ) \otimes_{\mathcal O}  {\mathcal F}  \ar[u] 
	\ar[r] & 0
	\\  \cdots
	\ar[r]^{ }  &	\Gamma(S^{n+1}(E^*)_{n+2}  \otimes E_{-2}) \ar[r]^{  {\rm{id}} \otimes \dd } 
	\ar[u]^{  Q^{(0)}  \otimes  {\rm{id}}}  &\Gamma(S^{n+1}(E^*)_{n+2}  \otimes E_{-1}) \ar[u]^{  Q^{(0)}  \otimes  {\rm{id}}}  \ar[r]^{  {\rm{id}} \otimes \rho }   & \Gamma( S^{n+1}(E^*)_{n+2} ) \otimes_{\mathcal O}  {\mathcal F}    \ar[u]^{  Q^{(0)}  \otimes  {\rm{id}} }\ar[r] & 0 \\  \cdots
	\ar[r]^{ }  &	\Gamma( S^{n+1}(E^*)_{n+1} \otimes E_{-2}) \ar[r]^{  {\rm{id}} \otimes \dd } 
	\ar[u]^{  Q^{(0)} \otimes  {\rm{id}}}  &\Gamma(S^{n+1}(E^*)_{n+1} \otimes E_{-1}) \ar[u]^{  Q^{(0)} \otimes  {\rm{id}}}  \ar[r]^{  {\rm{id}} \otimes \rho }   & \Gamma( S^{n+1}(E^*)_{n+1} ) \otimes_{\mathcal O}  {\mathcal F}   \ar[u]^{  Q^{(0)}  \otimes  {\rm{id}} }  \ar[r] & 0  \\ &0 \ar[u] & 0 \ar[u]&  0 \ar[u] & }}
\end{equation}
so that the cohomology of the total differential of this bicomplex is zero. By diagram chasing, this implies that
the sub-bicomplex  obtained by removing the last column on the right is quasi-isomorphic to the complex formed by the last column. The quasi-isomorphism is simply given by the collection of arrows from the penultimate column to the last one,  which, up to the shift in degree, is precisely the map  $ \mathrm{rt}$.

 \isthiswhatyouowant
	Alternatively, this quasi-isomorphism can be seen as follows. There is a short exact sequence of complexes:
	 $$  0 \to \Gamma( S^{n+1}(E^*)_{\bullet} ) \otimes_{\mathcal O}  {\mathcal F} \rightarrow  {\mathcal B}_\bullet \stackrel{\pi}{\to}   \mathfrak{X}_{\mathrm{vert}}^{(n)}(E)_\bullet  \to 0 , $$
	 where  ${\mathcal B}_\bullet$ is the complex associated to the bicomplex \eqref{eq:Bigbicomplex} and  $\pi$  is the projection of the bicomplex \eqref{eq:Bigbicomplex} onto \eqref{eq:bicomplex2} composed with the isomorphism \eqref{eq:Phi0AsComplex}. Since the lines of the bi-complex ${\mathcal B}_\bullet$  have trivial cohomology, the total differential of ${\mathcal B}_\bullet$ has trivial  cohomology. 
	 Taking the associated long exact sequence in cohomology,  we obtain an isomorphism at the level of the connecting map: 
$$	  H^k({\mathfrak{X}}_{\mathrm{vert}}^{(n)}(E),  \ad_{Q^{(0)}} ) \simeq H^{k+1}( \Gamma( S^{n+1}(E^*)) \otimes_{\mathcal O} {\mathcal F} ,  Q^{(0)} \otimes \id )$$
for all $k \in {\mathbb Z}$. This connecting map is easily seen to be induced by ${\mathrm{rt}}$.
	\end{proof}
\begin{lemme}\normalfont\label{rmk:root0}
	Let $ (E,\dd, \rho)$ be a geometric resolution of $ {\mathcal F}$ and  $R$ a vertical vector field of degree $i$ and arity $n$ which is $Q^{(0)}$-closed and whose root is strictly zero.	Then $ R= [Q^{(0)},W]$ for some vertical vector field $W$ of arity $n$ which has no component in $ \Gamma ( S^{n+1} ( E^* )_{i} \otimes E_{-1}) $.
\end{lemme}
	\begin{proof} 
		Decompose $R$ as a sum $ R =\sum_{j \geq 1}  R_j$ with $ R_j \in \Gamma(S^{(n+1)} (E^*)_{j+i}  \otimes E_{-j} )$.
		We need to find $W$ of the form  $ W =\sum_{j \geq 2}  W_j$ with $ W_j \in \Gamma(S^{(n+1)} (E^*)_{j+i-1}  \otimes E_{-j} )$.
		The relation $ R= [Q^{(0)},W]$ amounts to a sequence $  ({\mathcal H_k})_{k \geq 2}$ of conditions:
		$$ \left\{ \begin{array}{crcl} ({\mathcal H_2}) & ({\mathrm{id}} \otimes \dd) \, W_2 &=& R_1 ,\\  ({\mathcal H_k}) & ({\mathrm{id}} \otimes \dd) \, W_k &=& R_{k-1} -  (Q^{(0)} \otimes {\mathrm{id}})W_{k-1} \quad \hbox{ for all $ k \geq 3$}.  \end{array}\right.$$
		Since $ {\mathrm{rt}}(R)= ({\mathrm{id}} \otimes \rho) (R_1)  = 0$, exactness of the $i$th line in (\ref{eq:lines}) implies that there is some $W_2 \in  \Gamma(S^{(n+1)} (E^*)_{i+1} \otimes E_{-2} ) $ satisfying  $({\mathcal H_2})$. 		The remaining components are then constructed by recursion: Assume that we have constructed $W_2, \dots,W_k$ satisfying $  ({\mathcal H_2}), \dots, ({\mathcal H_k})$ for some $k \geq 2$. The condition $[Q^{(0)},R]=0$ translates into the following set of equations
		$$ (Q^{(0)} \otimes {\mathrm{id}}) R_{j-1} = - ({\mathrm{id}} \otimes \dd) R_j \, \quad \hbox{ for all $ j \geq 2$} \, . $$ 
	Applying $(Q^{(0)} \otimes {\mathrm{id}})$ to the recursion assumption $({\mathcal H_k})$, the above equation for $j=k$, together with anticommutativity of the two differentials, shows that $R_{k} - ( Q^{(0)} \otimes {\mathrm{id}}) W_{k} $ is $ ({\mathrm{id}} \otimes \dd)$-closed. Then exactness of the $i$th line in  (\ref{eq:lines}) grants the existence of some $W_{k+1} \in  \Gamma(S^{(n+1)} (E^*)_{i+k} \otimes E_{-k-1})$ which satisfies $ ({\mathcal H}_{k+1})$. \end{proof}

\begin{corollaire} \label{cor:coho} In the case of a geometric resolution $ (E,\dd, \rho)$ one has:
\begin{enumerate}  \item $H^k ({\mathfrak X}_{{\mathrm{vert}}}^{(n)}(E),   \ad_{Q^{(0)}} ) \cong H^{k+1} (\Gamma( S^{n+1}(E^*))\otimes_{\mathcal O} {\mathcal F},  Q^{(0)} \otimes \id  )$
\item $H^2 ({\mathfrak X}_{{\mathrm{vert}}}^{(n)}(E),   \ad_{Q^{(0)}} )=0 $ if  $ n \geq 3$.
\end{enumerate}
\end{corollaire}
\begin{proof} The first item is a trivial consequence of Proposition \ref{prop:fondamental3} after dropping the shift in degree of $\cF$. The second item follows from the first one for $k=2$ and the fact that every element of degree three in $\Gamma( S^{n+1}(E^*))$ is zero if $n>2$. 
\end{proof}

\subsubsection{Extension of an almost Lie algebroid  to a graded almost Lie algebroid on a geometric resolution}
\isthiswhatyouowant
Now we return to our initial problem: We are given a geometric resolution $(E,\dd,\rho)$. This specifies in particular $E$ and $Q^{(0)}$ in the deformation problem outlined in section \ref{sec:aritydefo}. To take care of the given anchor map, we now search for a $Q^{(0)}$-closed $Q^{(1)}$ inducing $\rho$. 

\isthiswhatyouowant
In fact, according to Proposition \ref{prop:Almost}, we know that we can equip $E_{-1}$ with an almost Lie algebroid bracket for the given anchor. More precisely, an anchor and a binary bracket on $E_{-1}$ satisfying the Leibniz identity correspond precisely to a degree one vector field $Q^{(1)}_{E_{-1}}$ of arity one. These data define an almost Lie algebroid structure if in addition Equation \eqref{algebroid2}  is satisfied. It is an easy exercise to verify that this last identity is satisfied iff 
\begin{equation} \label{Q1Q1E1}
[Q^{(1)}_{E_{-1}},Q^{(1)}_{E_{-1}}] \in \mathfrak{X}_{\mathrm{vert}}(E_{-1}) \, . \end{equation}
Thus, to initiate our deformation problem, we need to extend
$Q^{(1)}_{E_{-1}}$ on $E_{-1}$ to a vector field $Q^{(1)}$ defined on all of $E$ which is $\ad_{Q^{(0)}}$-closed. We now first lift $Q^{(1)}_{E_{-1}}$ to a vector field $Q^{(1)}_{E}$ by means of a connection.\footnote{The 
 graded manifolds we discuss here are split and thus we also have an embedding of  $E_{(-1)}$ into $E$. We can choose any extension  of the vector field outside the given graded submanifold, provided only of the given degree and arity. For degree reasons, these two notions of a lift and such an extension are equivalent here.} Given an arbitrary choice for the lift, we search for 
\begin{equation} Q^{(1)}:= Q^{(1)}_{E} + V \label{Q1}\end{equation}
where $V$ is a vertical vector field in $ \oplus_{i \geq 2} \Gamma ( S^2(E^*)_{i+1}  \otimes E_{-i} ) $, \emph{i.e.}~with no component in $ \Gamma ( \wedge^2 E_{-1}^{*}  \otimes E_{-1})$. 

\isthiswhatyouowant
Note that $V$ now also parametrizes different choices for the initial lift $Q^{(1)}_{E}$. One can convince oneself that for every choice of lift and of $V$,  the property \eqref{Q1Q1E1} ensures that  the obtained $Q^{(1)}$ satisfies Equation \eqref{Q1Q1}. We now want to show that there always exists a choice for $V$ such that $Q^{(1)}$ in \eqref{Q1} becomes $\ad_{Q^{(0)}}$-closed. The resulting structure on $E$ given by $Q^{(0)}+Q^{(1)}$ then defines what we call a graded almost Lie algebroid defined over the initial resolution: 
\begin{definition} A \emph{graded almost Lie algebroid} is a complex of vector bundles \eqref{sequencelinear1} equipped with a bracket
$$ [ \cdot, \cdot] \colon \Gamma(E_{-i}) \times \Gamma (E_{-j}) \to \Gamma(E_{-i-j+1})  $$
 which satisfies the following three axioms 
 \begin{eqnarray}
   [x,fy] &=& f[x,y]+ \rho(x)[f] \, y ,  \\ 
    \dd [x,y] &=& [\dd (x),y] + (-1)^i [x, \dd(y)] , \\
    \rho([x,y])&=& [\rho(x),\rho(y)] 
  \end{eqnarray}
  for all $x \in \Gamma(E_{-i})$, $y \in \Gamma(E_{-j})$, and $f \in C^\infty(M)$. Above, it is understood 
  that $\rho(x)=0$ if $x \in \Gamma(E_{-i})$ for all $ i \geq 2$. 
\end{definition}

\begin{proposition}\label{prop:step1} Every geometric resolution $(E,\dd,\rho)$ and every almost Lie algebroid structure on $E_{-1}\subset E$ can be extended to an almost Lie algebroid structure on $E$. \label{prop:haha}
\end{proposition}
\isthiswhatyouowant The proof will use the following three lemmata, the proof of the first of which we leave to the reader.

\begin{lemme} \label{hihi} A graded almost Lie algebroid structure is in one-to-one correspondence with a  graded manifold $E$ equipped with a degree one vector field $Q= Q^{(0)}+Q^{(1)}$ of arity at most one such that  Equations \eqref{jacques00}, \eqref{jacques0}, and  \eqref{Q1Q1} are satisfied.
\end{lemme}

\begin{lemme}
	\label{lem:train}
	The vector field $ [Q^{(0)} , Q^{(1)}]$ with $Q^{(1)}$ as in \eqref{Q1} defines, up to different choices of $V$,  an element in  $H^2 ({\mathfrak{X}}_{\mathrm{vert}}^{(1)}(E),  \ad_{Q^{(0)}} )$.
\end{lemme}
\begin{proof} We first observe that  $Q^{(1)} f = \rho^{\ast}(\dd_{\mathrm{dR}}f)$. 
This yields
$$ [Q^{(0)},Q^{(1)}](f)= ({\rm d}^{(2)})^* \circ \rho^* (\dd_{\mathrm{dR}}f )  $$
and this vanishes due to  property $\rho \circ \dd^{(2)} = 0$ of the resolution. Thus $[Q^{(0)},Q^{(1)}] \in {\mathfrak{X}}_{\mathrm{vert}}$. The rest is obvious then.  
	\end{proof}
\begin{lemme} \label{haha}  For every choice of $V$ in \eqref{Q1}, $\mathrm{rt}([Q^{(0)},Q^{(1)}]) = 0$. \end{lemme}
\begin{proof} We already established above that the self-commutator of $Q^{(1)}$ is vertical, Equation \eqref{Q1Q1}. Thus so is $\big[Q^{(0)},[Q^{(1)},Q^{(1)}]\big]$. For every function $f\in {\mathscr O}$, we then compute:
\begin{equation}
\label{jenaiplusdidee}
0 =\frac{1}{2}\left[Q^{(0)},[Q^{(1)},Q^{(1)}]\right](f)=\left[[Q^{(0)},Q^{(1)}],Q^{(1)}\right](f)=[Q^{(0)},Q^{(1)}] \left( \rho^{\ast}(\dd_{\mathrm{dR}}f) \right).
\end{equation}
Comparison with Equation \eqref{rootdef} then establishes the equation to prove. 
\end{proof}
\begin{remarque} \normalfont The relation  $\mathrm{rt}([Q^{(0)},Q^{(1)}]) = 0$  translates into $
	\rho\big(\{\R{d}(x),y\}_{2}\big)=0 
	$
for every $\ x\in \Gamma(E_{-2}),\ y\in \Gamma(E_{-1})$. The latter equation can be proven also by using first that, in an almost Lie algebroid $E_{-1}$, $\rho$ is a morphism of the brackets and then that $\dd \circ \rho = 0$ by the resolution property.
\end{remarque}

\begin{proof}[Proof of Proposition \ref{prop:haha}] 	
	For every choice of $V$	in \eqref{Q1}, Lemma \ref{lem:train} implies that 
	$[Q^{(0)},Q^{(1)}]$ is a vertical vector field.  Lemmata \ref{haha} and \ref{rmk:root0} imply that there exists a vertical vector field $W$ with $[Q^{(0)},Q^{(1)}]=[Q^{(0)},W]$ whose component in $\wedge^2 E_{-1}^* \otimes E_{-1} $ is zero. This $W$ can now be used to redefine the initially given $V$ above, $V \mapsto V-W$, such that the resulting $Q^{(0)}$ and $Q^{(1)}$ satisfy the required relations of a graded almost Lie algebroid, see Lemma \ref{hihi}.
\end{proof}

 \isthiswhatyouowant
Proposition \ref{prop:step1} together with Proposition \ref{prop:Almost} imply:
 \begin{corollaire} \label{cor:haha}
 	Every geometric resolution $(E,\dd,\rho)$ of a singular foliation $\cF$ admits a graded almost Lie algebroid structure over it, i.e.\ the complex $(E,Q^{(0)})$ permits the definition of a degree and arity one vector field $Q^{(1)}$ such that Equations \eqref{jacques00}, \eqref{jacques0}, and  \eqref{Q1Q1} are satisfied.
 \end{corollaire}

 \begin{remarque} \label{rem:haha} \normalfont
In contrast to Proposition \ref{prop:haha}, Corollary \ref{cor:haha} permits us to deform an initially specified almost Lie algebroid structure  $(E_{-1},Q^{(1)}_{E_{-1}})$. For the proof of this weaker statement, it is sufficient to show that, for every lift $Q^{(1)}_E$ of \emph{some} $Q^{(1)}_{E_{-1}}$, the bracket $[Q^{(0)},Q^{(1)}_E]$ defines a vanishing cohomology class in  \eqref{H2vert} for $n=1$. 
It is interesting to see that while the deformation problem leads to cohomology classes in \eqref{H2vert} for $n \geq 2$, setting its initial data leads to a likewise one 
for $n=1$. 
 \end{remarque}

 \subsubsection{Extension of an almost Lie algebroid  to a  Lie  $\infty$-algebroid on a geometric resolution}\label{preuveexistence}

\isthiswhatyouowant
In this section we extend Proposition \ref{prop:haha} to the following statement:
\begin{proposition} Every graded almost Lie algebroid $(E,Q^{(0)}+Q^{(1)})$ over a geometric resolution $(E,\dd,\rho)$ can be extended to a  Lie $\infty$-algebroid structure on $E$. \label{prop:hahafinal}
\end{proposition} 
\begin{proof}    We are thus given $Q^{(0)}$ and $Q^{(1)}$ satisfying Equations \eqref{jacques00}, \eqref{jacques0}, and  \eqref{Q1Q1} for the deformation problem described in section \ref{sec:aritydefo}. 
 For the extension problem, Equations \eqref{jacques}, to have a solution, vanishing of the cohomology classes in \eqref{H2vert} for all $n \geq 2$ is necessary and sufficient. 
According to Corollary \ref{cor:coho}, all cohomology spaces for $n>2$ are trivial,
$$ H^2(\mathfrak{X}_{\mathrm{vert}}^{(n)}(E), \ad_{Q^{(0)}}) = 0  \qquad \forall n \geq 3. $$
We are thus left with considering only the case $n=2$. By Proposition \ref{prop:haha}, $[Q^{(1)},Q^{(1)}]$
is vertical and $\ad_{Q^{(0)}}$-closed. In this case $H^2(\mathfrak{X}_{\mathrm{vert}}^{(2)}(E), \ad_{Q^{(0)}})$ may be nontrivial, but the following lemma shows that the 2-class of $[Q^{(1)},Q^{(1)}]$ vanishes.
\begin{lemme}\label{lem:forQ2} For the arity one part $Q^{(1)}$ of the odd vector field characterizing a  graded almost Lie algebroid, one has  
\begin{equation} \mathrm{rt}\left([Q^{(1)},Q^{(1)}]\right)=0. \label{rootQ1Q1}\end{equation}
\end{lemme}
\begin{proof}
 The Jacobi identity implies
$	\big[[Q^{(1)},Q^{(1)}],Q^{(1)}\big] =0$. 
	Application to a function $f \in {\mathcal O}$ and the verticality of $[Q^{(1)},Q^{(1)}]$ implies	 $$ 0 = \big[[Q^{(1)},Q^{(1)}],Q^{(1)}\big] (f) = [Q^{(1)},Q^{(1)}](Q^{(1)}
	(f) )  =[Q^{(1)},Q^{(1)}] ( \rho^* ( \dd_{\mathrm{dR}} f )) .$$   Thus, by Equation \eqref{rootdef},  $ {\mathrm{rt}}  ([Q^{(1)},Q^{(1)}])[f] =0$
	 for all functions $f \in {\mathcal O}$. 
\end{proof} \isthiswhatyouowant This lemma together with the  isomorphism in cohomology resulting from Proposition \ref{prop:fondamental3} shows that also at level $n=2$ there is no obstruction. This then
completes the proof of the proposition. 
\end{proof}


\begin{remarque}\label{youhou}\normalfont
Equation \eqref{rootQ1Q1} is equivalent to the following relation \begin{equation}\label{eq:haha}
\forall\ x,y,z\in\Gamma(E_{-1})\hspace{1.3cm}\rho\big(\mathrm{Jac}(x,y,z)\big)=0,
\end{equation}
where $\mathrm{Jac}$ is the Jacobiator of the almost Lie algebroid bracket on $E_{-1}$. The validity of Equation \eqref{eq:haha} now follows immediately from the fact that $\rho$ is a morphism of brackets and the brackets of vector fields form a Lie algebra. This then provides an alternative proof of Lemma \ref{lem:forQ2}.
\end{remarque}

\isthiswhatyouowant Propositions  \ref{prop:haha} and \ref{prop:hahafinal} together imply:
\begin{corollaire} \label{cor:hahafinal0}
 Every geometric resolution $(E,\dd,\rho)$ and every almost Lie algebroid structure on $E_{-1}\subset E$ can be extended to a Lie $\infty$-algebroid structure on $E$.	
	 \end{corollaire}

\isthiswhatyouowant Using in addition Proposition \ref{prop:Almost}, we may also conclude the following:
\begin{corollaire} \label{cor:hahafinal}
 	Every geometric resolution $(E,\dd,\rho)$ of a singular foliation $\cF$ admits a   Lie $\infty$-algebroid structure over it.
 \end{corollaire}

\isthiswhatyouowant This is the content of the first part of Theorem \ref{theo:existe}, which refers to the smooth setting only.

\isthiswhatyouowant In the real analytic or holomorphic case, the first item in Proposition \ref{bonjourprop} ensure the existence of a finite length geometric resolution in the neighborhood of a point. Propositions   \ref{prop:Almost}, \ref{prop:haha}, and \ref{prop:hahafinal} now yield,  in a possibly smaller neighborhood, the second part of Theorem \ref{theo:existe}.

\isthiswhatyouowant We finally remark that Proposition \ref{prop:hahafinal}, as it is written, holds also in the real analytic and  holomorphic context. 




\subsection{Proof of Theorem \ref{theo:onlyOne} about universality}
\label{sec:context}

Throughout this section: 
\begin{itemize}
	\item $(E,Q)$ and $ (E',Q')$ are Lie $\infty$-algebroids over the same base manifold $M$. Their sheaves of functions are denoted by, respectively, $ {\mathcal E}$ and $ {\mathcal E}'$, and their linear parts are, respectively, the complexes $ (E,\dd,\rho)$ and $ (E',\dd',\rho')$.  
	\item The anchor of $E'$ defines a sub-singular foliation ${\mathcal F'}$ of the singular foliation  $ {\mathcal F}$ defined by $\rho$, \emph{i.e.}~$${\mathcal F'} \subset {\mathcal F},$$ where ${\mathcal F} \equiv \rho(\Gamma(E_{-1}))$ and ${\mathcal F'} \equiv \rho'(\Gamma(E_{-1}'))$.
\end{itemize}
Within this section, we work simultaneously in the smooth, real analytic, and holomorphic settings. 
 The proof of Theorem \ref{theo:onlyOne} is related to two deformation problems, which we will first introduce one after the other.

\subsubsection{Lie $\infty$-algebroid morphisms up to arity $n$} 

\label{sec:LieInftyDefor}



	\isthiswhatyouowant
	\begin{definition}
	Let $ n \in {\mathbb N}_0 \cup \{+\infty\}$.
	A \emph{Lie $\infty$-algebroid morphism up to order $n$} is a ${\mathcal O}$-linear graded algebra morphism
	$\Phi\colon  {\mathcal E} \to {\mathcal E}'$ such that  components of arity $i \leq n$ of $ Q' \circ \Phi -  \Phi \circ Q $ vanish, i.e. 
	\begin{equation} \Qtot(\Phi)^{(i)} = 0 \qquad \forall \, 0 \leq i \leq n . \label{conditionMi}
	\end{equation}		
	\end{definition}
	\isthiswhatyouowant
 For $n = + \infty$, we simply recover a standard Lie $\infty$-algebroid morphism.

	\isthiswhatyouowant
	Throughout this subsection, we are given a chain map   
	\begin{equation}
	\label{eq:Moph0}
	\begin{tikzcd}[column sep=0.9cm,row sep=0.6cm]
	\dots  \ar[r,"\dd'"] &E_{-3}'\ar[d,"\phi_{0}"] \ar[r,"\dd'"]&E_{-2}'\ar[d,"\phi_{0}"] \ar[r,"\dd'"] &E_{-1}'\ar[d,"\phi_{0}"] \ar[r,"\rho' "] &TM\ar[d,"\mathrm{id}"]\\
	\dots  \ar[r,"\dd"]  &E_{-3}\ar[r,"\dd"]&E_{-2} \ar[r,"\dd"]  &E_{-1} \ar[r,"\rho "] &TM 
	\end{tikzcd}.
	\end{equation}
	\isthiswhatyouowant
	The dual chain map $\phi_{0}^* \colon \Gamma(E'^*_\bullet) \to \Gamma(E^*_\bullet) $ admits a unique extension to an  $\mathscr{O}$-linear graded algebra morphism $ \Phi^{(0)} \colon {\mathcal E}' \to {\mathcal E}$. 
	Since, by duality, the arity zero components $Q^{(0)}$ of $Q$ and  $ Q'^{(0)}$ of $Q'$ correspond  to the differentials $\dd$ and $\dd'$, respectively, this map is in fact also a morphism of graded differential algebras $   \Phi^{(0)} \colon ({\mathcal E}, Q^{(0)}) \to  ({\mathcal E}', Q'^{(0)})$.
	It is therefore a Lie $\infty$-algebroid morphism up to arity $0$. Note that for being a Lie $\infty$-algebroid morphism up to arity $0$, one only needs the commutativity of the squares in \eqref{eq:Moph0} except for the one to the utmost right. Using  Equation \eqref{boum}, we can also rewrite the conditions \eqref{conditionMi} in terms of the Taylor coefficients of $\Phi$: 

	
	\begin{lemme}
		\label{prop:defMorph}
The $i$-th condition \eqref{conditionMi} restricted to $\Gamma(E^*)$ is equivalent to 			
	\begin{equation}
	\label{eq:defMorph}
\begin{array}{rcll}
	Q'^{(0)} \circ \phi_i^* - \phi_i^*  \circ \dd^* &=  \underset{a+b=i}{\sum_{1\leq a,b\leq i-1}} 
	Q'^{(a)} \circ \phi_b^* - \Phi^{(b)}  \circ Q^{(a)}  \, ,
	\end{array}
	\end{equation}
	where the Taylor coefficients are viewed upon as maps $\phi_k \colon S^{k+1}(E')_\bullet \to E_{\bullet}$. The whole set of equations \eqref{conditionMi} as it stands is equivalent to the set of equations \eqref{eq:defMorph} for all $0 \leq i \leq n$. 
	\end{lemme}

%


\isthiswhatyouowant	Applying Lemma \ref{lem:complexPhideriv} to $\Phi^{(0)}$, we see that the space $ {\mathfrak X}_{\mathrm{vert}}( \!\!{\xymatrix@C-=0.5cm{  {\mathcal E} \ar[r]^{\Phi^{\scalebox{0.5}{(0)}}} & {\mathcal E}'}}\!\!)  $ of ${\mathcal O}$-linear $ \Phi^{(0)}$-derivations is a complex with respect to the   differential:
	$$   \Qtoto (W) := Q'^{(0)} \circ W - (-1)^{|W|} W \circ Q^{(0)}  \hspace{1cm} \hbox{ $\forall  \, W \in   {\mathfrak X}_{\mathrm{vert}}( \!\!{\xymatrix@C-=0.5cm{  {\mathcal E} \ar[r]^{\Phi^{\scalebox{0.5}{(0)}}} & {\mathcal E}'}}\!\!)$}  .$$
	For every $ n \in {\mathbb N}$, the space ${\mathfrak X}^{(n)}_{\mathrm{vert}}(\!{\xymatrix@C-=0.5cm{  {\mathcal E} \ar[r]^{\Phi^{\scalebox{0.5}{(0)}}} & {\mathcal E}'}}\!\!)$ of ${\mathcal O}$-linear $\Phi^{(0)}$-derivations of arity $n$ is a sub-complex of this complex. 
	\isthiswhatyouowant
		Consider a Lie $\infty$-algebroid morphism $\Phi$ up to order $n$
	whose component of arity $0$ is  $ \Phi^{(0)} \colon  {\mathcal E} \to {\mathcal E}'$ or, equivalently, whose zeroeth Taylor coefficient  is $\phi_0$.
		 We say that  Lie $\infty$-algebroid morphism  $\tilde{\Phi} \colon  {\mathcal E} \to {\mathcal E}'$ up to order $n+1$  is an \emph{extension of $ \Phi$} if $ \tilde{\Phi}^{(i)} = \Phi^{(i)}$ for all $ i \leq n$ or, equivalently, if $ \tilde{\Phi} $ and $ \Phi$ have the same Taylor coefficients up to order $n$. Such as $H^2({\mathfrak X}^{(n+1)}_{\mathrm{vert}}( E) , \ad_{Q^{(0)}}  )$ governs the deformation problem in the context of Equations (\ref{jacques}), the following proposition shows that this role is played by 
	 $$  H^1 \left(  {\mathfrak X}^{(n+1)}_{\mathrm{vert}}( \!\!{\xymatrix@C-=0.5cm{  {\mathcal E} \ar[r]^{\Phi^{\scalebox{0.5}{(0)}}} & {\mathcal E}'}}\!\!),  \Qtoto \right) $$
	 for the Equations (\ref{eq:defMorph}) in the search of $\phi_{n+1}$. Note that for each $i$, the l.h.s.\ of Equation \eqref{eq:defMorph}  can be rewritten as $\Qtoto(\phi^*_i):= Q'^{(0)} \circ \phi^*_i -  \phi^*_i \circ Q^{(0)}$, while the r.h.s.\ contains Taylor coefficients of $\Phi$ up to order $i-1$ only. 	
	\begin{proposition}
		\label{prop:obsMorph}
		For every  Lie $\infty$-algebroid morphism  $\Phi\colon  {\mathcal E} \to {\mathcal E}'$ up to order $n$ one has:
		\begin{enumerate}
			\item  $ c_{n+1}(\Phi) := (\Qtot(\Phi))^{(n+1)} $ is an ${\mathcal O}$-linear, $\Qtoto$-closed  $ \Phi^{(0)}$-derivation.
			\item If the class $[c_{n+1}(\Phi)] $ is zero in $ H^1 ({\mathfrak X}_{\mathrm{vert}}^{(n+1)}( \!\!{\xymatrix@C-=0.5cm{  {\mathcal E} \ar[r]^{\Phi^{\scalebox{0.5}{(0)}}} & {\mathcal E}'}}\!\!) , \Qtoto )$, then $\Phi$ permits an extension
			to a Lie $\infty$-algebroid morphism
			up to order $n+1$.
		\end{enumerate}
	\end{proposition}	
	\begin{proof}
		Considering the component of arity $n+1+i+j$ in (\ref{eq:quasider0}),
			for all homogeneous functions $F$ and $G$ of arities $i$ and $j$, respectively, yields:
		\begin{equation}
		\label{quasiPhiDer2}
		{\Qtot (\Phi)}^{(n+1)}(F  \odot G)= {\Qtot (\Phi)}^{(n+1)}(F) \odot \Phi^{(0)} (G)  - (-1)^{|F|} \Phi^{(0)} (F)   \odot {\Qtot (\Phi)}^{(n+1)} (G).
		\end{equation} 
		Since $c_{n+1}(\Phi)  = {\Qtot (\Phi)}^{(n+1)}$, this proves that $c_{n+1}(\Phi) $ is an  
		 ${\mathcal O}$-linear $ \Phi^{(0)}$-derivation.
		 
		 \isthiswhatyouowant Since, by assumption, ${\Qtot (\Phi)}^{(i)}=0$ for all $0 \leq i \leq n$, taking the component of arity $n+1$ within the identity $0=\Qtot^2 (\Phi) = Q' \circ \Qtot (\Phi) + \Qtot (\Phi) \circ Q $, we obtain 
		\begin{equation}  Q'^{(0)} \circ {\Qtot (\Phi)}^{(n+1)} + {\Qtot (\Phi)}^{(n+1)} \circ Q^{(0)}=0 .\end{equation}
		In other words, $c_{n+1}(\Phi)  = {\Qtot (\Phi)}^{(n+1)}$ is $\Qtoto $-closed.
		This concludes the proof of the first item. 	

\isthiswhatyouowant		Assume now that there exists an ${\mathcal O}$-linear $\Phi^{(0)}$-derivation $W$ of degree $0$ and of arity $n+1$ such that $c_{n+1}(\Phi) =  {\Qtot (\Phi)}^{(n+1)} =  -\Qtoto (W).$
		Let $ \tilde{\Phi}$ be as in the following lemma, which is shown easily by use of  Equations \eqref{boum} and \eqref{boum2}:	
	\begin{lemme}
			\label{lem:DeltaPhiExt}
		Let $\Phi\colon {\mathcal E} \to {\mathcal E}'$ be an $\mathscr{O}$-linear graded algebra morphism and $W$  a $\Phi^{(0)}$-derivation of arity $n+1$. 
	Consider an $\mathscr{O}$-linear graded algebra morphism $\tilde{\Phi} \colon{\mathcal E} \to {\mathcal E}'$ whose Taylor coefficients coincide with those of $\Phi$ up to order $n$, $\tilde{\phi}_i = \phi_i$ for all $0 \leq i \leq n$ and whose $n$ plus first Taylor coefficient agrees with $W$, $\tilde{\phi}_{n+1}=W$.  Then 
	 $ \tilde{\Phi}^{(i)} =  {\Phi}^{(i)} $ for $0\leq  i \leq n$ and  $ \tilde{\Phi}^{(n+1)} =  {\Phi}^{(n+1)} + W $. 
	\end{lemme}
\isthiswhatyouowant	 This implies that 
		the components of arity $0,\dots,n$ of $ \Qtot ( {\tilde{\Phi}} )$ coincide with those of 
		  $ \Qtot \left( {{\Phi}} \right)$ and therefore vanish as well. 
		 In addition,
		$$ ( \Qtot ( {\tilde{\Phi}} )) ^{(n+1)}
		=\left( {\Qtot (\Phi)}\right)^{(n+1)}+ Q'^{(0)} \circ W - W \circ Q^{(0)}  =c_{n+1}(\Phi) + \Qtoto (W) . $$
		The right hand side of this equation is equal to zero by definition of $ W$. This proves the second item.
		\end{proof}

		\subsubsection{Homotopy deformations up to arity $n$.}
		
		\label{sec:deformationHomot}
	
		\begin{definition}
			\label{def:defhomoto}
			Let $ n \in {\mathbb N}_0 \cup \{+\infty\}$ and	$ \Phi,\Psi $ two Lie $ \infty$-algebroid morphisms
 from $ (E',Q')$ to $(E,Q)$. A \emph{homotopy deformation of  $ \Psi $ into $ \Phi$ up to arity $n$} is a sequence $(\Psi_i)_{-1 \leq i \leq n}$  of  Lie $ \infty$-algebroid morphisms from $(E',Q')$ to $ (E,Q)$
			such that:
			\begin{enumerate}
				\item[\emph{(i)}] $\Psi_{-1} = \Psi$,
				\item[\emph{(ii)}]  $\Psi_{i}$ and $\Phi$ agree  up to arity $i$ for all $0 \leq i \leq n$,
				\item[\emph{(iii)}] $\Psi_i$ and $\Psi_{i+1}$ are homotopic through a homotopy $(\Psi_\tau, H_\tau)_{ \tau \in [i,i+1]}$ where, for all $\tau \in ]i,i+1[$,  $H_\tau$ is a $\Psi_\tau$-derivation
				whose components of arity less or equal to $i$ vanish.
			\end{enumerate}
		\end{definition}

		\isthiswhatyouowant
		The following statement explains the importance of homotopy deformations:
		
		\begin{proposition}
			\label{prop:secondpart}
			If a homotopy deformation up to arity $+\infty$ of  $ \Psi $ into $ \Phi$ exists, they $\Phi$ and $ \Psi$ are homotopic Lie $\infty$-algebroid morphisms.
		\end{proposition}
		\isthiswhatyouowant This statement is an immediate consequence of the following lemma: 
		
\begin{lemme}			Let  $f\colon[0,1[ \to [-1,+\infty[$ 
			be a strictly increasing surjective smooth function. Then  the pair $\left(\Psi_{f(t)},f'(t) \, H_{f(t)} \right) $ together with $\Psi_{f(1)}:=\Phi$ is a homotopy between $ \Phi$ and $\Psi$. 
\end{lemme}			
			
	\begin{proof}		Assumptions \emph{(ii)} and \emph{(iii)} imply that, for all $\tau \geq n$,  $\Psi_{\tau}^{(n)}$ is equal to $\Phi^{(n)}$. Hence, for each fixed $n$, there exists a neighborhood $U_n \subset [0,1]$ of  $1$ such that for all $t \in U_n$ we have $\Psi_{f(t)}^{(n)}=\Phi^{(n)}$.  In particular,
			$t \mapsto \Psi_{f(t)}^{(n)}$  is continuous and piece-wise-$C^\infty$. Likewise, $f'(t)\, H_{f(t)}^{(n)}$ is piece-wise smooth and vanishes  for all $t \in U_n$.
			
					\isthiswhatyouowant
			By construction, the pair $(\Psi_{f(t)},  f'(t)\, H_{f(t)})$ satisfies 
		the following equation	$$\frac{\dd \Psi_{f(t)}}{\dd t}=\Qtot \left( f'(t)\, H_{f(t)} \right). $$ Comparison with Equation \eqref{Hortense} shows that  this pair provides  a homotopy between $\Psi_{f(0)} = \Psi_{-1}= \Psi$ and $\Psi_{f(1)} =\Phi$. 
		\end{proof}

		\isthiswhatyouowant
		Given a homotopy deformation up to arity $n$, we wish to extend it up to arity $n+1$. As the following proposition shows in parallel to Proposition	\ref{prop:obsMorph}, there is again an obstruction which is cohomological in nature, now living in 
		 $$  H^0 \left(  {\mathfrak X}^{(n+1)}_{\mathrm{vert}}( \!\!{\xymatrix@C-=0.5cm{  {\mathcal E} \ar[r]^{\Phi^{\scalebox{0.5}{(0)}}} & {\mathcal E}'}}\!\!),  \Qtoto \right). $$

		\begin{proposition}
			\label{prop:obsMorphHompt}
			Consider a homotopy deformation $(\Psi_i)_{i =-1, \dots,n}  $  of $\Psi $ into  $\Phi$ up to arity $n \geq 0$. Then:
			\begin{enumerate}
				\item  $ c_{n+1}(\Phi,\Psi_n) := (\Phi-\Psi_n)^{(n+1)} $ is an ${\mathcal O}$-linear $\Qtoto$-closed  $ \Phi^{(0)}$-derivation.
				\item If the class of $ c_{n+1}(\Phi,\Psi_n) $ in $ H^0 ({\mathfrak X}_{\mathrm{vert}}( \!\!{\xymatrix@C-=0.5cm{  {\mathcal E} \ar[r]^{\Phi^{\scalebox{0.5}{(0)}}} & {\mathcal E}'}}\!\!) , \Qtoto )$ vanishes, then there exists a  Lie $\infty$-algebroid morphism $ \Psi_{n+1}$ 
				such that $(\Psi_i)_{i =-1, \dots,n+1}$ provides a homotopy deformation  of $\Psi $ into  $\Phi$ up to order $n+1$.
				
			\end{enumerate}
		\end{proposition}
		\begin{proof} For all $F,G \in {\mathcal E}$, one has:
			\begin{equation} (\Phi-\Psi_n)(F\odot G)= (\Phi-\Psi_n)(F) \odot \Phi(G)+ \Psi_n(F) \odot (\Phi-\Psi_n)(G)  .\end{equation}
Since, by assumption,  $(\Phi-\Psi_n)^{(i)} = 0$ for all $0 \leq i \leq n$, taking $F$ and $G$ of arity $i$ and $j$ in the above equation and considering its component of arity $n+1+i+j$, we obtain:
			\begin{equation}
			(\Phi-\Psi_n)^{(n+1)}(F\odot G) =  (\Phi-\Psi_n)^{(n+1)}(F) \odot  \Phi^{(0)}(G)+ \Phi^{(0)}(F) \odot (\Phi-\Psi_n)^{(n+1)}(G). 
			\end{equation}
			This implies that $(\Phi-\Psi_n)^{(n+1)} \colon {\mathcal E} \to {\mathcal E}'$ is a $\Phi^{(0)}$-derivation.
			Moreover, since both $\Phi$ and $\Psi_n$ are Lie $\infty$-algebroid morphisms, one has:
			\begin{equation} (\Phi-\Psi_n) \circ Q = Q' \circ (\Phi-\Psi_n). \end{equation}
			Restricting this equation to the component of arity $n+1$ yields, for the same reason, $ (\Phi-\Psi_n)^{(n+1)} \circ Q^{(0)} = Q'^{(0)} \circ (\Phi-\Psi_n)^{(n+1)} $. Thus $(\Phi-\Psi_n)^{(n+1)}$ 
			is $\Qtoto$-closed, which proves the first item.

			\isthiswhatyouowant
			Assume now that $ (\Phi-\Psi_n)^{(n+1)}$ is  $\Qtoto$-exact, so that there exists
			a $\Phi^{(0)}$-derivation $W$ of arity $n+1$ and degree  minus one  such that:
			\begin{equation} \label{eq:defdelta} (\Phi-\Psi_n)^{(n+1)}  =  \Qtoto(W).\end{equation}
			In view of Example \ref{ex:homotopies},
			the following differential equation admits a solution on the interval $ [n,n+1]$ with initial condition  at $t=n$:
			\begin{equation}\label{eq:diffHomot} \frac{\diff \Psi_t}{\diff t} =  \Qtot \left( w^{\Psi_t} \right), \qquad    \end{equation}
			where $w$ is the Taylor coefficients of $W$ and, as before, $w^{\Psi_t}$ denotes the ${\Psi_t}$-derivation with this coefficient.
			
			\isthiswhatyouowant 
			For notational simplicity, let $ H_t :=  w^{\Psi_t}$. Since $w $ is a section of $ \Gamma(S^{n+2} (E'^*) \otimes E)$, Equation \eqref{boum2}
			implies:
		\begin{equation} \label{huhu1}
			 H_t^{(k)}=0  \quad \mathrm{and} \qquad (\Qtot ( w^{\Psi_t} ))^{(k)} \equiv \Qtot (H_t )^{(k)}  =0 \qquad \forall \: k \leq n. 
			 \end{equation} Thus according to Equation \eqref{eq:diffHomot} and by an appropriate choice of initial conditions satisfying  our assumptions, one has $$ \Psi_t^{(k)}=\Psi_n^{(k)}=\Phi^{(k)}  \qquad \forall t \in [n,n+1], \: \forall k \leq n , $$	 and thus also 
			 \begin{equation}  \label{huhu3} \Psi_{n+1}^{(k)}=\Phi^{(k)}  \qquad \forall k \leq n. \end{equation} 
			Since, according to Equation \eqref{boum2}, for the critical arity  $n+1$, the morphism $\Psi_t$ does not enter the expression for $H_t \equiv  w^{\Psi_t}$, it also coincides with the $\Phi^{(0)}$-derivation having the same Taylor coefficient $w$:
			  \begin{equation} \label{huhu2}  H_t^{(n+1)}=W . \end{equation}
			By Equations \eqref{eq:diffHomot}, \eqref{huhu1}, and \eqref{huhu2} therefore:
			\begin{eqnarray*}\frac{\diff \Psi_t^{(n+1)}}{\diff t} &=& \Qtot \left(  H_t \right)^{(n+1)}
				\\ &=& Q'^{(0)} \circ H_t^{(n+1)} +
				H_t^{(n+1)} \circ Q^{(0)} \\&=&  Q'^{(0)}  \circ W    +W \circ Q^{(0)}\equiv \Qtoto (W) = (\Phi-\Psi_n)^{(n+1)} .
			\end{eqnarray*}
			This implies
			$ \Psi_t^{(n+1)} = \Psi_n^{(n+1)}+(t-n)(\Phi-\Psi_n)^{(n+1)}$,
			so that $$ \Psi_{n+1}^{(n+1)} 
			= \Phi^{(n+1)}, $$ which, together with Equation \eqref{huhu3}, completes the proof of the second item.
		\end{proof}	
		

		\subsubsection{The cohomology of ${\mathcal O}$-linear $\Phi^{(0)}$-derivations} 

\label{sec:coho2}

\isthiswhatyouowant
Let $\phi_0$ be as in \eqref{eq:Moph0} and  $\Phi^{(0)}\colon {\mathcal E} \to {\mathcal E}'$ its extension as in Section \ref{sec:LieInftyDefor}.
 We now study the complex $(\mathfrak{X}^{(n)}_{\mathrm{vert}}(\!\!{\xymatrix@C-=0.5cm{  {\mathcal E} \ar[r]^{\Phi^{\scalebox{0.5}{(0)}}} & {\mathcal E}'}}\!\!)_\bullet,\Qtoto )$, the cohomology of which hosts the obstruction classes of both previous subsections. 
 
 \isthiswhatyouowant
Assigning to an ${\mathcal O}$-linear $\Phi^{(0)}$-derivation of arity $n$ its Taylor coefficients yields
 a graded ${\mathcal O}$-module isomorphism:
 \begin{equation}\label{eq:Phi0AsComplexPhi}
 \mathfrak{X}^{(n)}_{\mathrm{vert}}(\!\!{\xymatrix@C-=0.5cm{  {\mathcal E} \ar[r]^{\Phi^{\scalebox{0.5}{(0)}}} & {\mathcal E}'}}\!\!)_\bullet \, \simeq \, \Gamma( S^{n+1}(E'^*) \otimes E )_\bullet.
\end{equation}
For $E'=E$ and ${\Phi^{\scalebox{0.5}{(0)}}}=\mathrm{id}$, this reproduces \eqref{eq:Phi0AsComplex}.
Under this isomorphism, the differential $\Qtoto$ becomes the total differential of the following bicomplex:	
\begin{equation}
\label{eq:bicomplexPhi}
 \resizebox{0.9\textwidth}{!}{\xymatrix{ & \vdots & \vdots &  \vdots \\ \cdots
	\ar[r]^{ }  &	\Gamma(S^{n+1}(E'^*)_{n+3}  \otimes E_{-3}) \ar[r]^{  {\rm{id}} \otimes \dd } 
	\ar[u]
	&\Gamma(S^{n+1}(E'^*)_{n+3}  \otimes E_{-2}) \ar[u]
	\ar[r]^{  {\rm{id}} \otimes \dd }   & \Gamma( S^{n+1}(E'^*)_{n+3}  \otimes E_{-1} )  \ar[u] 
	\ar[r]& 0 \\  \cdots
	\ar[r]^{ }  &	\Gamma(S^{n+1}(E'^*)_{n+2}  \otimes E_{-3}) \ar[r]^{  {\rm{id}} \otimes \dd } 
	\ar[u]^{  Q'^{(0)}  \otimes  {\rm{id}}}  &\Gamma(S^{n+1}(E'^*)_{n+2}  \otimes E_{-2}) \ar[u]^{  Q'^{(0)}  \otimes  {\rm{id}}}  \ar[r]^{  {\rm{id}} \otimes \dd }   & \Gamma( S^{n+1}(E'^*)_{n+2}  \otimes E_{-1} )   \ar[u]^{  Q'^{(0)}  \otimes  {\rm{id}} }\ar[r]& 0 \\  \cdots
	\ar[r]^{ }  &	\Gamma( S^{n+1}(E'^*)_{n+1} \otimes E_{-3}) \ar[r]^{  {\rm{id}} \otimes \dd } 
	\ar[u]^{  Q'^{(0)} \otimes  {\rm{id}}}  &\Gamma(S^{n+1}(E'^*)_{n+1} \otimes E_{-2}) \ar[u]^{  Q'^{(0)} \otimes  {\rm{id}}}  \ar[r]^{  {\rm{id}} \otimes \dd }   & \Gamma( S^{n+1}(E'^*)_{n+1} \otimes E_{-1} )   \ar[u]^{  Q'^{(0)}  \otimes  {\rm{id}} } \ar[r]& 0 \\
	&0 \ar[u] & 0 \ar[u]&  0 \ar[u] & }}\end{equation}
Again, we can consider the \emph{root map}:
$$ \mathrm{rt} \colon \mathfrak{X}_{\mathrm{vert}}^{(n)} (\!\!{\xymatrix@C-=0.5cm{  {\mathcal E} \ar[r]^{\Phi^{\scalebox{0.5}{(0)}}} & {\mathcal E}'}}\!\!)_\bullet \longrightarrow \Gamma( S^{n+1}(E'^*))_{\bullet}\otimes_{\mathcal O} {\mathcal F}[-1] $$
which is obtained  by applying $ \id \otimes \rho $ to the component $ S^{n+1} (E'^*)_\bullet \otimes E_{-1} $ of an $\cO$-linear ${\Phi^{\scalebox{0.5}{(0)}}}$-derivation $W$. Thus, if $P \colon \mathfrak{X}_{\mathrm{vert}}^{(n)} (\!\!{\xymatrix@C-=0.5cm{  {\mathcal E} \ar[r]^{\Phi^{\scalebox{0.5}{(0)}}} & {\mathcal E}'}}\!\!)_\bullet \longrightarrow 
\Gamma(S^{n+1} (E'^*)_\bullet \otimes E_{-1})$ denotes the corresponding projector, 
simply 
\begin{equation}  \mathrm{rt} \equiv (\id \otimes \rho) \circ P \, . \label{rootformula}
\end{equation} 
The proof of the following proposition and corollary follows the same line of reasoning as the proof of Proposition \ref{prop:fondamental3} and its Corollary \ref{cor:coho} and is left to the reader.
\begin{proposition}
	\label{prop:fondamental3Def}
	If $ (E,\dd, \rho)$ is a geometric resolution of $ {\mathcal F}$, then 
	$$\mathrm{rt} \colon ({\mathfrak{X}}_{\mathrm{vert}}^{(n)}(\!\!{\xymatrix@C-=0.5cm{  {\mathcal E} \ar[r]^{\Phi^{\scalebox{0.5}{(0)}}} & {\mathcal E}'}}\!\!)_\bullet,  \Qtoto ) \longrightarrow ( \Gamma( S^{n+1}(E'^*))_\bullet \otimes_{\mathcal O} {\mathcal F}[-1] ,  Q'^{(0)} \otimes \id) $$
	is a quasi-isomorphism.
\end{proposition}



\begin{corollaire} \label{cor:cohoPhi} In the case of a geometric resolution $ (E,\dd, \rho)$, one has:
	\begin{enumerate}  
		\item $H^1 \left({\mathfrak X}^{(n)}_{\mathrm{vert}}(\!\!{\xymatrix@C-=0.5cm{  {\mathcal E} \ar[r]^{\Phi^{\scalebox{0.5}{(0)}}} & {\mathcal E}'}}\!\!),   \Qtoto \right)=0 $ if  $ n \geq 2$,
		\item $H^0 \left({\mathfrak X}^{(n)}_{\mathrm{vert}}(\!\!{\xymatrix@C-=0.5cm{  {\mathcal E} \ar[r]^{\Phi^{\scalebox{0.5}{(0)}}} & {\mathcal E}'}}\!\!),   \Qtoto \right)=0 $ if  $ n \geq 1$.
	\end{enumerate}
\end{corollaire}	
\begin{remarque}
	\normalfont
	If $(E',\dd',\rho')=(E, \dd, \rho)$ and $\phi_0=\id$, then ${\mathcal O}$-linear $\Phi^{(0)}$-derivations are vertical vector fields on $E$ 
	and Proposition \ref{prop:fondamental3Def} and Corollary \ref{cor:cohoPhi} reduce to Proposition \ref{prop:fondamental3} and Corollary \ref{cor:coho}, respectively.
\end{remarque}

 \subsubsection{First part of Theorem \ref{theo:onlyOne}: Construction of a morphism as a unobstructed deformation problem}



\begin{proposition} \label{prop:finalMorph}
	Let $(E,Q)$ be  a universal {\LieInftyAlgebroid}  \resolving a singular foliation $ \mathcal{F}$, \emph{i.e.}~$ (E,\dd, \rho)$ is a geometric resolution and let $(E',Q')$ be an arbitrary Lie algebroid over the same base $M$. Every  {\LieInftyAlgebroid} morphism up to order $n\in \mathbb{N}$, $\Phi \colon (E',Q') \to (E,Q)$, can be extended to a (strict)  {\LieInftyAlgebroid} morphism. For $n=0$, this statement is true as well if we add the condition $\rho' = \rho \circ \phi_0$ on $E_{-1}'$. 
\end{proposition} 
\begin{proof} 
	It is sufficient to show that all the obstruction classes $[c_{k+1}]$ in Proposition \ref{prop:obsMorph} for all $k \geq n$ are zero. For $n \in \mathbb{N}$ and all $k>0$, this is automatic due to first item in Corollary \ref{cor:cohoPhi}. We are thus left with the case $k=0$, where we still need to consider the class of	 $$ c_1 (\Phi^{(0)}):= \left( \Phi^{(0)} \circ Q' - Q \circ \Phi^{(0)}\right)^{(1)} =   \Phi^{(0)} \circ Q'^{(1)} - Q^{(1)} \circ \Phi^{(0)}  \, .$$
		Its  component in $S^2(E_{-1}'^*) \otimes E_{-1} $ contracted with $x,y \in \Gamma(E_{-1}')$
		is given by 
		\begin{equation} \label{rhophi}
		\big\{\phi_0(x),\phi_0(y)\big\}-\phi_0\big(\{x,y\}'\big), 
		\end{equation}
		where 
		$ \{\ .\ ,\, .\ \}$ and $\{\ .\ ,\, .\ \}'$ denote the almost Lie algebroid brackets on $E_{-1}$ and $E_{-1}'$ dualizing $Q^{(1)}$ and $Q'^{(1)}$, respectively. According to Equation \eqref{rootformula}, $\mathrm{rt}(c_1 (\Phi^{(0)}))$ is obtained by applying $\rho$ to Equation \eqref{rhophi} 		\begin{equation} \label{rhophi2} \langle \mathrm{rt}(c_1 (\Phi^{(0)})), x \odot y \rangle =\big\{\rho(\phi_0(x)),\rho(\phi_0(y))\big\}-\rho(\phi_0\big(\{x,y\}'\big)) ,
		\end{equation}
		where we used the morphism property of the almost-Lie algebroid $E_{-1}$. The assumption $\rho' = \rho \circ \phi_0$ (which follows automatically for $n >0$ and is added by hand for $n=0$) and  the morphism property of the almost-Lie algebroid $E'_{-1}$ then imply $\mathrm{rt}(c_1 (\Phi^{(0)}))=0$.  Together with Proposition \ref{prop:fondamental3Def}, this shows 
		
		\hspace{5cm}$ [c_1 (\Phi^{(0)})]=0 $. 
\end{proof}	

 \isthiswhatyouowant It is now easy to complete the proof of the first part of Theorem 
  \ref{theo:onlyOne}. Given the assumptions specified at the beginning of this subsection, Section \ref{sec:context}, together with $ (E,\dd, \rho)$ being a geometric resolution, there always exists  a chain map  $\phi_0$
	as in \eqref{eq:Moph0}---globally in the smooth setting (this is a classical result of abelian categories and was cited already as Lemma \ref{lem:existsAChain} above) and locally in the real analytic and holomorphic cases. The corresponding graded algebra morphism from ${\mathcal E}$ to  ${\mathcal E}'$ is also a Lie $\infty$-algebroid morphism up to order zero from $({\mathcal E},Q)$  to  $({\mathcal E}',Q')$: 
	$\Qtot(\Phi)^{(0)} \equiv Q'^{(0)} \circ \Phi^{(0)} - \Phi^{(0)} \circ  Q^{(0)}$ vanishes due to the commutativity of  \eqref{eq:Moph0}. 
	In addition, by the same diagram, also the additional condition for $n=0$  is satisfied in the above proposition. This completes the proof for the first part of the theorem.
 
 \isthiswhatyouowant  We remark in parenthesis that the construction of $\Phi$ out of $\phi_0$  in \eqref{eq:Moph0} can be also viewed as the existence,  in the case that $(E,Q)$ is universal, of a deformation of the Lie $\infty$-algebroid morphism $\Psi:=\Phi^{(0)} \colon ({\mathcal E}, Q^{(0)}) \to  ({\mathcal E}', Q'^{(0)})$ into a Lie $\infty$-algebroid morphism $\Phi \colon ({\mathcal E},Q)  \to ({\mathcal E}',Q')$.  
 

\subsubsection{Second part of Theorem \ref{theo:onlyOne}: Construction of a homotopy as a unobstructed homotopy deformation problem}

	\isthiswhatyouowant
	\begin{proposition} \label{prop:finalMorphHomot}
			Let $\Phi$ and $\Psi$ be Lie $\infty$-algebroid morphisms over the identity of $M$ from  an arbitrary Lie $\infty$-algebroid $(E',Q')$ to  a universal one  $(E,Q)$ \resolving a singular foliation $ \mathcal{F}$.
			Every homotopy deformation $(\Psi_i)_{i =-1, \dots,n}  $  of $\Psi $ into  $\Phi$ up to arity $n \geq 0$ can be extended  up to arity infinity.  
			%
	\end{proposition} 
	
	\isthiswhatyouowant
\begin{proof}
	For a homotopy deformation up to arity $n+1$  to exist, vanishing of the cohomology class $$[c_{n+1}(\Phi,\Psi_n)]\in H^0\left(\mathfrak{X}_{\mathrm{vert}}^{(n+1)}(\!{\xymatrix@C-=0.5cm{  {\mathcal E} \ar[r]^{\Phi^{\scalebox{0.5}{(0)}}} & {\mathcal E}'}}\!\!), \Qtoto\right)$$ is necessary and sufficient, see Proposition \ref{prop:obsMorphHompt}. 
	However, according to the second item in Corollary \ref{cor:cohoPhi}, for $(E,Q)$ universal, these cohomology spaces are already trivial. Therefore, the extension of a homotopy deformation of $ \Psi$ into $ \Phi$ from up to an arity $n\geq 0$ to arity infinity is unobstructed and exists. 
 \end{proof}

\isthiswhatyouowant
By Proposition \ref{prop:secondpart}, if a homotopy deformation $(\Psi_i)_{i =-1}^{+ \infty}  $  of $\Psi $ into  $\Phi$ up to arity infinity exists, then $\Phi$ and $ \Psi$ are  homotopic Lie $\infty$-algebroid morphisms.
In view of the above proposition, Proposition \ref{prop:finalMorphHomot}, to complete the proof of the second part of Theorem  \ref{theo:onlyOne},
it suffices therefore to construct a homotopy deformation of $ \Psi$ into $ \Phi$ up to arity zero---on the manifold $M$ in the smooth case and in a neighborhood of a point in the real analytic or holomorphic cases.

\isthiswhatyouowant
In the smooth case, Lemma \ref{lem:existsAChain} implies that the linear parts $ \phi_0$ and $\psi_0$ of $\Phi $ and $\Psi$ are homotopic through a homotopy $h\colon E'_\bullet \to E_{\bullet-1}$:
\begin{equation}
\label{eq:ThereIshomot}
\begin{tikzcd}[column sep=1.3cm,row sep=0.7cm]
\ldots\ar[r,"\dd"] & E'_{-3} \ar[r,"\dd"]\ar[dd,shift left =0.5ex,"\phi_0"]\ar[dd,shift right =0.5ex,"\psi_0" left]& E'_{-2} \ar[ddl,dashed, "h\ " above left] \ar[dd,shift left =0.5ex,"\phi_0"]\ar[dd,shift right =0.5ex,"\psi_0" left]\ar[r,"\dd"]& E'_{-1} \ar[ddl,dashed,"h\ " above left] \ar[dd,shift left =0.5ex,"\phi_0"]\ar[dd,shift right =0.5ex,"\psi_0" left] \ar[dr,"\rho'"]  & \\
& & & & TM\\
\ldots\ar[r,"\dd'"]& E_{-3} \ar[r,"\dd'"]& E_{-2}\ar[r,"\dd'"]&  E_{-1} \ar[ur,"\rho"] &
\end{tikzcd}
\end{equation}
In the real analytic or holomorphic cases, such a homotopy $h$ may  exist in a neighborhood of a point only. Theorem 
\ref{theo:onlyOne} now follows from the above discussion and the following lemma.
\begin{lemme}
	\label{lem:arity1}
	Assume a homotopy as in \eqref{eq:ThereIshomot} is given.
	Then the Lie $\infty$-algebroid  morphism $\Psi$ is homotopic to a Lie $\infty$-algebroid  morphism whose linear part is the linear part of~$ \Phi$,
	\emph{i.e.}~there exists a homotopy deformation of $\Psi $ into  $\Phi$ up to arity $ 0$.
\end{lemme}
\begin{proof}
	%
	%
	Consider the differential equation:
	$$ \frac{\diff \Psi_t}{\diff t} = \Qtot \left( h^{\Psi_t} \right), $$
	where $h^{\Psi_t}$ is the unique $\Psi_t$-derivation whose unique non-vanishing Taylor coefficient is $h$, \emph{i.e.}~whose restriction to $\Gamma(E^*)$
	is the dual $h^*$ of $h$. 
	According to Example \ref{ex:homotopies}, this differential equation admits a solution  $(\Psi_t,  h^{\Psi_t})_{t \in [0,1]}$ which is a homotopy between $\Psi_0 := \Psi$ and $\Psi_1$.
	Moreover, since the linear part of $h^{\Psi_t}$  is $h^*$ for all $t$, the following differential equation
	is satisfied for all $\alpha \in \Gamma(E^*)$:
	\begin{equation}
	\label{eq:DeriveePremiere}
	\begin{split}\frac{\diff \Psi_t^{(0)}}{\diff t}(\alpha) &=(Q'^{(0)} \circ h^*) (\alpha)+ (h^*   
	\circ Q^{(0)})(\alpha)\\ &= \Phi^{(0)}(\alpha) - \Psi^{(0)}(\alpha),\end{split}
	\end{equation}
	where in the second line we used the homotopy property expressed by diagram \eqref{eq:ThereIshomot}. 
	In view of the initial condition $ \Psi_0^{(0)} = \Psi^{(0)}$, we thus have $\Psi_{t}^{(0)} = \Psi^{(0)}+ t(\Phi^{(0)} - \Psi^{(0)})  $. 
	As a consequence, the component of arity zero of $\Psi_1$ is equal to $\Phi^{(0)} $. This proves the claim.
\end{proof}


\subsection{Examples of  universal Lie \texorpdfstring{$\infty$}{infinity}-algebroid structures of a singular foliation}
\label{sec:examplesLieInfinity}

\isthiswhatyouowant
In this section, we give examples of universal {\LieInftyAlgebroid} structures \resolving a given singular foliation.

\begin{example}\normalfont
For a regular foliation ${\mathcal F}$ on a manifold $M$, the Lie algebroid $T F \subset TM$, whose sections form ${\mathcal F} $,  is a universal {\LieInftyAlgebroid} \resolving ${\mathcal F}$.
\end{example}

\begin{example}\normalfont
	\label{ex:Debord2}
A singular foliation is a Debord foliation if and only if one of its universal {\LieInftyAlgebroid}s  is a Lie algebroid (see Example \ref{ex:debord}).
\end{example}

\begin{example}\normalfont
\label{sl2resolution}
In Example \ref{sl2}, we gave a geometric resolution of length two of the singular foliation coming from the action of $\mathfrak{sl}_{2}$ on $\mathbb{R}^{2}$.
Let us compute now the {\LieInftyAlgebroid} structure on that geometric resolution.
We define the bracket between two constant sections of $E_{-1} \simeq \mathfrak{sl}_{2}$ as being their bracket in $\mathfrak{sl}_{2}$. 
Then we extend this bracket to every section of $E_{-1}$ by the Leibniz identity \eqref{robinson}. 
To define the bracket between sections of $E_{-1}$ and $E_{-2}$, we notice that:
\begin{equation}
\{\widetilde{e},\dd r\}=xy\{\widetilde{e},\widetilde{h}\}+\rho(\widetilde{e})(xy)\widetilde{h}+\rho(\widetilde{e})(y^{2})\widetilde{e}-x^{2}\{\widetilde{e},\widetilde{f}\}=0.
\end{equation}
Since $\dd$ is injective on a dense open subset, this imposes $\{\widetilde{e}, r\}= 0$. 
The same argument also gives $[\widetilde{f} , r ]=[\widetilde{h} , r]=0$.
We then extend these brackets to a bracket between sections of  $E_{-1}$ and $E_{-2}$ by the Leibniz property \eqref{robinson}. 
There is no $k$-ary bracket for $k \geq 3$.

\isthiswhatyouowant
There is an evident Lie algebroid inducing the same foliation, namely the action Lie algebroid $A = M \times \mathfrak{sl}_{2}$. According to Theorem \ref{theo:onlyOne}, there must exist a Lie $\infty$-algebroid morphism $\phi \colon (A,Q_A) \to (E,Q)$, which here can be realized by the obvious inclusion $\mathrm{id} \oplus 0$ of $A$ into $E=E_{-1}\oplus E_{-2}$. 
\end{example}

\begin{example}\normalfont
\label{conjecture2}
The result is similar for the geometric resolution of the foliation $ {\mathcal F}_{ad}$ given by the adjoint action of a semi-simple complex Lie algebra on itself, as in Example \ref{conjecture}.
In that case, $E_{-1} $ is the trivial bundle over  $M= \mathfrak{g} $ with typical fiber $ \mathfrak{g}$, and the bracket of constant sections can be chosen
 to be the bracket of the Lie algebra $\mathfrak{g} $. 
  In fact, $E_{-1}$ by itself is the corresponding action Lie algebroid. 
 The bracket of a constant section of $ E_{-1}$ with a constant section of $E_{-2}$ has  to be zero and there is no $3$-ary bracket.
  \end{example}


\begin{example} \normalfont
\label{ex:vfVanishingAtZero_structure}
Let ${\mathcal F}$ be the singular foliation of all vector fields vanishing at the origin $0\in V$, $V$ a vector space.
Let us consider the geometric resolution given in Example \ref{ex:vfVanishingAtZero}.

\isthiswhatyouowant
A {\LieInftyAlgebroid} structure on that geometric resolution can be described explicitly.
Let all $k$-ary brackets vanish for $k \geq 3$. The $1$-ary bracket is already given by the resolution. It now suffices to define the $2$-ary bracket on constant sections, since they can be extended to all sections by means of the Leibniz property, using the anchor map \eqref{robinson}.
For all $\alpha \in \wedge^{i} V^*$, $\beta \in  \wedge^{j} V^*$, and $u,v \in V$, it is defined as follows: \begin{equation}
 \label{eq:bracketVFVatZero}
 \big\{ \alpha \otimes u , \beta \otimes v \big\}_2 =  \alpha  \wedge {\mathfrak i}_u \beta \otimes v + (-1)^{ij} \beta \wedge {\mathfrak i}_v \alpha \otimes u .
\end{equation}
This bracket is graded symmetric by construction.
The Jacobi identity can be checked by a direct calculation, which one is permitted to perform on constant sections, since the above bracket preserves constant sections. The compatibility with the differential is shown as follows, where first we assume $i,j \ge 2$:
\begin{eqnarray*}  {\rm d}^{(i+j-1)} \big\{ \alpha \otimes u , \beta \otimes v \big\}_2  & = & {\mathfrak i}_e \left( \alpha  \wedge {\mathfrak i}_u \beta  \otimes v  + 
(-1)^{ij} \beta \wedge {\mathfrak i}_v \alpha \otimes u \right) \\
 &=&  ({\mathfrak i}_e  \alpha )  \wedge {\mathfrak i}_u \beta  \otimes v   + (-1)^{ij+j+1} \beta \wedge {\mathfrak i}_v ({\mathfrak i}_e \alpha) \otimes u  \\ & &  + 
(-1)^{i+1}  \alpha  \wedge {\mathfrak i}_u {\mathfrak i}_e \beta  \otimes v  + 
  (-1)^{ij} ({\mathfrak i}_e) \beta \wedge {\mathfrak i}_v \alpha \otimes u  \\
  &= &    \big\{ {\rm d}^{(i)} \left( \alpha \otimes u  \right), \beta \otimes v \big\}_2 - (-1)^i \big\{  \alpha \otimes u , {\rm d}^{(j)} \left(\beta \otimes v\right) \big\}_2
  \end{eqnarray*} 
  For $i=1$ or $j=1$, one replaces $\dd^{(1)}$ by the anchor $\rho$ in the above. 
  
  \isthiswhatyouowant
 Again, in this example, there is an action Lie algebroid giving rise to the same foliation, which can be obtained by restriction to $E_{-1}$. In fact, this Lie algebroid is just $V \times \mathfrak{gl}(V)$. 
\end{example}

\begin{example}
	\normalfont
	\label{ex:vanishingOrder2_v2}
We consider the foliation $ {\mathcal F}$ of vector fields vanishing to order $2$ at $0$ on $ V={\mathbb K}^2$ and  its resolution 
$$ 0 \to E_{-2} \to E_{-1} \to TM 
$$ 	
	constructed in Example \ref{ex:vanishingOrder2}.  The vector bundle $E_{-1} $ is the trivial bundle with fiber $S^2(V^*) \otimes V$. Denote by $  \alpha^*, \beta^* $ the quadratic functions
	  on $V$ associated to  $\alpha, \beta \in S^2 V^* $ and by $\underline{u} ,\underline{v}$  the constant vector fields on $V$ associated to elements $u,v \in V$.
Then the anchor map can be defined by means of 	$\rho ( \alpha \otimes v)  = \alpha^* \, \underline{v} $ and the 2-ary brackets between constant sections by:
	 \begin{equation} \label{eq:Quadra} \{\alpha \otimes u , \beta \otimes v \}_2 =  (\underline{u} \, [\beta^*])  \, (\alpha \otimes v) - 
	   (\underline{v} \, [ \alpha^*]) \, (\beta \otimes u) \, . \end{equation}
The Jacobiator of this bracket is non-zero and thus, in this example, the completion of which we leave to the reader, the $3$-ary bracket is necessarily non-zero as well. In this case, we do not know of a Lie algebroid that generates the same foliation.
\end{example}

 \begin{example}
 \label{ex:Koszul2}
  \normalfont
 We return to Example \ref{ex:Koszul} with its geometric resolution $(E,\dd,\rho)$ where $ \Gamma(E_{-i}) $ is the sheaf of $(i+1)$-multivector fields on ${\mathbb C}^n$.
This can be equipped with the structure of a Lie $\infty$-algebroid by postulating the following $k$-ary brackets:
 \begin{equation}
 \label{eq:Koszul}
 \left\{  \partial_{I_1}, \dots, \partial_{I_k} \right\}_k\quad :=  
 \sum_{i_1 \in I_1, \dots,i_k \in I_k}  \epsilon(i_1, \dots,i_k)  \, \tfrac{\partial^k \varphi}{ \partial x_{i_1} \dots \partial x_{i_k}}\;\: \partial_{I_1^{i_1}\bullet I_2^{i_2}\bullet \ldots \bullet I_k^{i_k} }  .
 \end{equation}
Here for every multi-index $I = (i_1, \dots, i_j)$ of length $j$ of elements in $\{1, \dots,n\}$,  $\partial_I$ is  a shorthand notation for $\frac{\partial}{\partial x_{i_1}} \wedge \dots \wedge \frac{\partial}{\partial x_{i_j}} $. Given a collection of $k$ multi-indices $I_1, I_2, \ldots, I_k$, we can form a new multi-index $I:=I_1\bullet I_2 \bullet \ldots \bullet I_k$ by concatenation. For every $i_1 \in I_1, \dots,i_k \in I_k$, $\epsilon(i_1, \dots,i_k)$ is the signature of the permutation needed inside $I$ to bring them to the first $k$ slots in the given order. For every $i_s \in I$, $I^{i_s}$ denotes the list which results from dropping $i_s$ from the list $I =  (i_1, \dots, i_j)$.

\isthiswhatyouowant
The anchor is defined by 
 \begin{equation} \rho \left( \tfrac{\partial }{\partial x_i}\wedge \tfrac{\partial }{\partial x_j} \right) := \tfrac{\partial \varphi}{ \partial x_{i} } \tfrac{\partial }{\partial x_j} - \tfrac{\partial \varphi}{ \partial x_{i} }\tfrac{\partial }{\partial x_i} \end{equation}
and the brackets, defined in \eqref{eq:Koszul} on constant sections only, extended appropriately.   One may verify, e.g.\ by an adaptation of the techniques in \cite{VoronovTrick,Voronov}, that indeed these data define a {\LieInftyAlgebroid}.
 
 \isthiswhatyouowant
 In this example, in general, the $k$-ary brackets for $k=3, 4 ,\dots$ are not zero and there is no obvious Lie algebroid inducing the same singular foliation. 
 \end{example}

\begin{example}
\normalfont
\label{ex:affine3}
Our construction can be used also to associate three potentially different homotopy classes of Lie $\infty$-algebroids to a given affine variety $W\subset  {\mathbb C}^n$. These result from application of the main theorem to the three different ${\mathscr O}$-modules described in  Example \ref{ex:affine2}.
\end{example}	

\newpage
 
\section{The geometry of a singular foliation through its universal Lie \texorpdfstring{$\infty$}{infinity}-algebroid}
\label{sec:geometry}

\isthiswhatyouowant
In this section we relate the geometry of a singular foliation to its  universal {\LieInftyAlgebroid}.
More precisely, since any two universal {\LieInftyAlgebroid}s  \resolving a singular foliation $ \mathcal{F}$ are homotopy equivalent by Corollary \ref{coro:unique}, every homotopy invariant mathematical notion deduced from it  is in fact associated to the singular foliation $ \mathcal{F}$ itself.

\isthiswhatyouowant
For example, Lemma \ref{lem:alt_sum} shows that the alternate sums of the ranks of the vector bundles underlying a {\LieInftyAlgebroid} is preserved under homotopy equivalence.  Therefore it must have a geometrical meaning associated directly to the singular foliation. And indeed, this was provided in Proposition \ref{prop:Marco}: it coincides with the dimension of regular leaves.

\isthiswhatyouowant
Typical examples of such notions are various cohomologies that {\LieInftyAlgebroid}s are equipped with.
In what follows, we shall first study some global cohomologies and then turn to those attached canonically to every given leaf.

\isthiswhatyouowant We round up this section by relating the fundamental groupoid of the {\LieInftyAlgebroid} 
to the holonomy groupoid of a given singular foliation and conclude with some remarks about an eventual Lie or Leibniz algebroid over it.

\isthiswhatyouowant
Throughout this section, we work in the smooth context. Except for the discussion about the holonomy groupoid, 
the results remain valid in the holomorphic 
and in the real analytic context  in a neighborhood of a point.

\isthiswhatyouowant
Unless otherwise specified, throughout Sections \ref{sec:universal}, \ref{isotropysec}, and \ref{moineau}, $\mathcal{F}$ is a singular foliation on a manifold $M$, $ (E,\dd,\rho)$ is a geometric resolution of finite length,  $(E,Q)$ is a universal {\LieInftyAlgebroid} of  $\mathcal{F}$, and $ {\mathcal E}$ its sheaf of 
functions.

\subsection{Universal foliated cohomology}
\label{sec:universal}

%

\isthiswhatyouowant
Here is a first example of  cohomology associated to a singular foliation
with the help of the universal {\LieInftyAlgebroid}.

\begin{lemme}
\label{cor:cancohom}
Let ${\mathcal F}$ be a singular foliation on $M$. Let $ (E,Q)$ and $(E',Q')$ be two universal
{\LieInftyAlgebroid}s \resolving ${\mathcal F}$ with sheaves of functions $\mathcal{E}$ and $\mathcal{E}'$. The cohomologies of $(\mathcal{E}, Q) $ and $  (\mathcal{E}', Q') $ are canonically isomorphic as graded commutative algebras.
\end{lemme}
\begin{proof}
By Corollary \ref{coro:unique}, 
there exist {\LieInftyAlgebroid} morphisms $\varphi \colon \mathcal{E}'\to\mathcal{E}$ and $\psi \colon \mathcal{E}\to\mathcal{E}'$
whose compositions are homotopic to the identity.
Proposition \ref{prop:HomotopyMeansHomotopy} implies that ${\Phi}:= \varphi \circ \psi$ and ${\Psi}:= \mathrm{id}$ are inverse to one another at the level of cohomology. Two morphisms $\varphi, \varphi'\colon \mathcal{E}'\to\mathcal{E}$ are homotopic to one another, moreover, so that, again at the level of cohomology, they define the same isomorphism.
\end{proof}

\isthiswhatyouowant This justifies the following definition.

\begin{definition} 
	Let ${\mathcal F}$ be a singular foliation on $M$ that admits  a geometric resolution. We call the cohomology of $({\mathcal E}, Q)$, where ${\mathcal E} $ is the sheaf of functions of any universal {\LieInftyAlgebroid} $ (E,Q)$
	of the given ${\mathcal F} $, the \emph{universal foliated cohomology of $ {\mathcal F}$} and denote it by  $H_{\mathscr{U}}(\mathcal F)$.
	 \end{definition}


\begin{remarque}\normalfont By construction,  $H^0_{\mathscr{U}}({\mathcal F})$ is the algebra
 of functions on $M$ 
 constant along the leaves of $ {\mathcal F}$. 
 \end{remarque}

\isthiswhatyouowant
 There is another cohomology of interest. 
Let us call 
 $$ \Omega({\mathcal F}) := \mathrm{Hom}_{\mathscr O}\big(\wedge_{\mathscr O}{\mathcal F},{\mathscr O}\big)= 
 \bigoplus_{k \geq 0} \mathrm{Hom}_{\mathscr O}\big(\wedge_\mathscr{O}^k {\mathcal F},\mathscr{O}\big)$$ the space of \emph{longitudinal forms}. It is 
 equipped with the differential:
\begin{align}
{\dd}_L (\alpha)(X_0, \dots , X_k) &=\sum_{i=0}^k (-1)^{i} X_i\big[ \alpha (X_{0},\ldots,\widehat{X_i},\ldots,X_k)\big] \\&\hspace{1cm}+ \sum_{0\leq i < j\leq k}(-1)^{i+j}\alpha\Big([X_i,X_j],X_0,\ldots, \widehat{X_{i}},\ldots,\widehat{X_{j}},\ldots,X_k\Big) .\nonumber
\end{align}
Here $\alpha \in \Omega^k({\mathcal F}) $, $X_0,\dots,X_k \in {\mathcal F}$,  and $\widehat{X_i}$ means that the term $X_i$ is omitted.
 We call the cohomology of this operator the \emph{longitudinal cohomology of $ {\mathcal F}$} and denote it by $H^\bullet({\mathcal F})$.

\begin{remarque}
	\normalfont
	A $k$-form $\alpha \in  \Omega^k({\mathcal F})$ induces a $k$-form on each 
	regular leaf $L$: for every regular point $m $, it is the unique $k$-form $\alpha_{L_m} $ on the leaf $L_m$ through $m$ such that 
	\begin{equation}
	\label{eq:restrictionKforms}
	\alpha (X_1, \dots, X_k) |_m = \alpha_{L_m}(X_1(m), \dots, X_k(m))
	\end{equation}
for all $X_1, \dots, X_k \in {\mathcal F} $. In contrast,
	a $k$-form satisfying \eqref{eq:restrictionKforms} may not exist on singular leaves.
	For instance, consider the singular foliation on $ {\mathbb R}$ generated by the vector field $ x \tfrac{\partial}{\partial x}$ and the longitudinal $1$-form $ \alpha \colon F x \tfrac{\partial}{\partial x}  \mapsto F$ for all $ F \in {\mathcal O}$. The leaf $ \{0\}$ being of dimension $0$, it does not permit a $1$-form satisfying~\eqref{eq:restrictionKforms}.
\end{remarque}

\isthiswhatyouowant
 There is a natural map $\rho^*$ from $\Omega(\mathcal{F})$ to $ {\mathcal E}$ 
 which associates to each $ \alpha \in \Omega^k(\mathcal{F})$
 the element $\rho^* {\alpha} \in  \Gamma\big(S^k (E_{-1}^*)\big) \subset {\mathcal E}_k $ defined for all $x_1 , \dots,x_k \in \Gamma(E_{-1})$ by:
  \begin{equation}\label{eq:pullbackByAnchor} 
 \rho^* \alpha (x_1, \dots,x_k) = \alpha\big(\rho(x_1), \dots,\rho(x_k)\big) .
  \end{equation}

\begin{lemme}\label{isomorfoliation}
Let $\mathcal{F}$ be a singular foliation on $M$ that admits a universal
{\LieInftyAlgebroid}.
There is a canonical algebra morphism $\rho^*\colon H^\bullet(\mathcal{F})\to  H_{\mathscr{U}}^\bullet(\mathcal{F})$.
\end{lemme} 


  \begin{proof}
   It is routine to check that $\alpha \mapsto \rho^* ({\alpha})$ is a chain map and a graded commutative algebra morphism. It does induces a map in cohomology. Now, 
   let $ (E,Q)$ and $(E',Q')$ be two universal  {\LieInftyAlgebroid}s \resolving $\mathcal{F}$, 
   with sheaves of  functions  ${\mathcal E}$ and ${\mathcal E}'$ respectively.
  Let $\Phi_{{\mathcal E},{\mathcal E}'}:{\mathcal E}' \to {\mathcal E}$ be a {\LieInftyAlgebroid} morphism from  $ (E,Q)$ and $(E',Q')$ as in Corollary \ref{coro:unique}. The relation $ \Phi_{{\mathcal E},{\mathcal E}'} \circ (\rho')^* = \rho^*$ holds and
   proves the commutativity of the following diagram
   \begin{center}
   	\begin{tikzcd}[column sep=1.3cm,row sep=1.4cm]
   		& H_{}^\bullet ({\mathcal F}) \ar[dl,"\rho^*" above left] \ar[dr, "\rho'^*" above right] & \\
   		H^\bullet ({\mathcal E},Q)  \ar[rr, leftrightarrow ,"\simeq " ] &   & H^\bullet({\mathcal E}',Q')
   	\end{tikzcd} 
   \end{center}
   where the horizontal map is the canonical isomorphism given in Corollary \ref{cor:cancohom}. This proves the claim.
  \end{proof}

  \begin{example}
  	\normalfont
  	For Debord foliations,  the universal foliated cohomology is the Lie algebroid Chevalley-Eilenberg cohomology of $(A,[\cdot,\cdot],\rho)$ and $\rho^*\colon H^\bullet(\mathcal{F})\to  H_{\mathscr{U}}^\bullet(\mathcal{F})$ is an isomorphism. 
  \end{example}	
  
  \begin{remarque}
  	\normalfont
  	\label{ex:affine4}
  	Let $W$ be an affine variety and  ${\mathcal X}(W) $ be as in Example \ref{ex:affine2}. The longitudinal cohomology of  ${\mathcal X}(W) $ is the usual de Rham cohomology of an affine variety.
  	 It would be interesting to relate its universal foliated cohomology  to derived de Rham cohomology \cite{Toen}.
  \end{remarque}	
  
\subsection{The isotropy Lie \texorpdfstring{$\infty$}{infinity}-algebra.}
\label{isotropysec}
\isthiswhatyouowant
Androulidakis and Skandalis were able to define 
\cites{AndrouSkandal, AndrouZambis} a Lie algebra associated canonically to each point $m \in M$, if $M$ is equipped a singular foliation. They call it the \emph{isotropy Lie algebra} of ${\mathcal F}$ at $m$. 
It is defined as the quotient
 of the Lie algebra $\mathcal{F}(m)$ of local sections in $\mathcal{F}$ vanishing at $m \in M$ by the Lie ideal $I_m \mathcal{F}$,
where $I_m \subset {\mathcal O}$ are the functions vanishing at $m$.

\isthiswhatyouowant In this section we will apply the isotropy functor introduced in Section \ref{sec:functor} to the universal Lie $\infty$-algebroid of a given foliation, which will permit us to prove Theorem \ref{theo:lasttheorem}.


\subsubsection{The graded space of the Lie \texorpdfstring{$\infty$}{infinity}-algebra. }
\label{sec:isotropy}

\isthiswhatyouowant
Choose  an arbitrary point $m \in M$. Denote by ${\mathfrak i}_m V$ the fiber of a vector bundle $V \to M $ at $m$,
and by ${\mathfrak i}_m \phi $ the restriction to ${\mathfrak i}_m V$ of a vector bundle morphism $\phi \colon V \mapsto V' $.

\isthiswhatyouowant
Recall that although  $ (E,\dd,\rho)$ is a geometric resolution, the complex $ ({\mathfrak i}_m E, {\mathfrak i}_m \dd ,  {\mathfrak i}_m \rho)$ 
\begin{equation}
\label{eq:complexPointx}
\xymatrix{
\ldots \ar[r]^{{\mathfrak i}_m \dd^{(4)}} &{\mathfrak i}_m E_{-3}  \ar[r]^{{\mathfrak i}_m \dd^{(3)}}  &{\mathfrak i}_m E_{-2}\ar[r]^{{\mathfrak i}_m \dd^{(2)}}  &   {\mathrm{Ker}}({\mathfrak i}_m \rho) \ar[r]^{} &0
} 
\end{equation}
 may have cohomology: Exactness at the level of sections does not imply that
the complex (\ref{eq:complexPointx}) is exact at all points.
For instance, the geometric resolutions constructed in Examples \ref{sl2},\ref{ex:vfVanishingAtZero}, and \ref{ex:Koszul} 
have a non-zero cohomology when $m$ is the origin.

\isthiswhatyouowant
We define the graded vector space:
\begin{equation} \label{eq:H} H^{\bullet} ({\mathcal F},m) = \bigoplus_{i \geq 1} H^{-i} ({\mathcal F},m) \end{equation}
 to be the cohomology of the complex (\ref{eq:complexPointx}). Notice that the chosen notation does not make reference to any particular geometric resolution. This is justified by the first assertion of the following lemma:

\begin{lemme}
	\label{lem:canonical}
	Let $\mathcal{F}$ be a singular foliation that admits geometric resolutions on the neighborhood of some $m \in M$.
	\begin{enumerate}
		\item The cohomology of the complex (\ref{eq:complexPointx}) does not depend on the choice of a geometric resolution of $ {\mathcal F}$.
		\item 
		For every geometric resolution $(E,\dd,\rho)$ \resolving $ {\mathcal F}$ which is minimal at $m$ and every $i \geq 2$, the vector space
		$ H^{-i} ({\mathcal F},m) $ is canonically
		isomorphic to $ {\mathfrak i}_m E_{-i}$. In addition, $ H^{-1} ({\mathcal F},m) $ is canonically isomorphic to the kernel of ${\mathfrak i}_m \rho \colon  {\mathfrak i}_m E_{-1} \to T_m M $.
	\end{enumerate}
\end{lemme}
\begin{proof}
According to Lemma \ref{lem:existsAChain2}, between two geometric resolutions
 $( E,\dd , \rho)$ of
  $( E',  \dd' , \rho')$ of the same ${\mathcal O}$-module ${\mathcal F}$ there is  a distinguished homotopy class of homotopy equivalences relating them. By ${\mathcal O}$-linearity, it restricts to a homotopy class of homotopy equivalences between  the complexes $ ({\mathfrak i}_m E, {\mathfrak i}_m \dd ,  {\mathfrak i}_m \rho)$ and   $ ({\mathfrak i}_m E', {\mathfrak i}_m \dd' ,  {\mathfrak i}_m \rho')$.
  This proves the first assertion of the lemma.
 
%

\isthiswhatyouowant
The second statement follows from the obvious fact that for a minimal resolution at $m$, one has ${\mathfrak i}_m \dd = 0$ and thus the complex \eqref{eq:complexPointx} coincides with its cohomology. 
\end{proof}

\begin{remarque}
\normalfont
In terms of abelian categories,  $ H^{\bullet} ({\mathcal F},m)$ is ${\rm Tor}_{\mathcal O}^{\bullet}({\mathcal F}, {\mathbb K})$.
Here, the field $ {\mathbb K}$ is equipped with the $ {\mathcal O}$-module structure defined by $F \cdot \lambda = F(m) \lambda $ for all $F \in {\mathcal O} , \lambda \in {\mathbb K}$.
\end{remarque}

\begin{proposition}
	\label{prop:Debord}
	Let $\mathcal{F}$ be a singular foliation that admits geometric resolutions on the neighborhood of some $m \in M$.
	Then the following items are equivalent:
	\begin{enumerate}
		\item[(i)] There is a neighborhood of $m \in M$ on which $ \mathcal{F}$ is a Debord foliation. 
		\item[(ii)]  $H^{-i}(\mathcal{F},x)=0$ for all $ i\geq 2$
		and for all $x$ in a neighborhood of $m$.
		\item[(iii)]  $H^{-2}(\mathcal{F},m)=0$. 
	\end{enumerate}
\end{proposition}
\begin{proof}
	If a foliation is Debord, it admits a resolution of length one, see Example \ref{ex:debord}. This implies that the
	cohomologies $H^{-i}(\mathcal{F},y)$ are all trivial for all $i \geq 2$ and every point $y$ in that neighborhood. Hence \emph{(i)} implies \emph{(ii)}. 
	It is obvious  that \emph{(ii)} implies \emph{(iii)}. Let us assume  that  \emph{(iii)} holds, so that for every geometric resolution  $(E,\dd, \rho)$ of $\mathcal{F}$,
	the image of ${\mathfrak i}_m \dd^{(3)} $ is the kernel of ${\mathfrak i}_m \dd^{(2)} $.
	Since there exists a neighborhood  of $m$ such that
	 \begin{enumerate}
		\item the dimension of the image of $ \dd^{(3)}$ at  every $x \in U $ has to be greater than or equal to its dimension at $m \in M$ and		\item the dimension of the kernel of
		$ \dd^{(2)}\colon E_{-2} \to  E_{-1}$  at  every $x \in U $ has to be lower
 than or equal to its dimension at $m \in M$,
	\end{enumerate}
	but always $\mathrm{im}\,{\mathfrak i}_x \dd^{(3)} \subset \ker {\mathfrak i}_x \dd^{(2)} $, 
	 these dimensions have to coincide inside $U$. Therefore, $H^{-2}(\mathcal{F},x)=0$ 
	for all $x \in U $. This implies that $E_{-1}':= \left. \left(\bigslant{E_{-1}}{\dd^{(2)}(E_{-2})}\right)\right|_U$ is a vector bundle.  The anchor descends to the quotient  defining a  morphism of ${\mathscr O}$-modules
	$\rho \colon \Gamma(E_{-1}') \to \mathcal{F}$. By construction, it is an isomorphism of $\mathscr{O}$-modules and hence \emph{(iii)} implies \emph{(i)}.  
\end{proof}

\subsubsection{The brackets of the Lie \texorpdfstring{$\infty$}{infinity}-algebra. }
\label{sec:isotropy2}


\isthiswhatyouowant
Applying the isotropy functor $\cI_m$ for a chosen point $m \in M$ to a universal Lie $\infty$-algebroid $(E,Q)$ of a given singular foliation $\cF$, we induce a Lie infinity algebra structure on the complex \eqref{eq:complexPointx}. Since this functor maps homotopy equivalences of Lie $\infty$-algebroids to homotopy equivalences of Lie $\infty$-algebras, we prefer to look at the homotopy isotropy functor \eqref{eq:homotopyFunctor}. 
By Corollary \ref{coro:unique}, different choices of the universal Lie $\infty$-algebroid yield homotopy equivalent Lie $\infty$-algebras, i.e.\ isomorphic Lie $\infty$-algebras inside {\bf hLie-}$\infty${\bf-alg}. 

\isthiswhatyouowant
We now may choose the universal Lie $\infty$-algebroid appropriately, so as to directly induce a Lie $\infty$-algebra structure on \eqref{eq:H}, namely, according to 
Lemma \ref{lem:canonical}, by letting it be minimal at $m$. Every other one is isomorphic to this one and, since  \eqref{eq:H} is a trivial complex, it provides a minimal model of the isotropy Lie $\infty$-algebra.

 \begin{definition}
 	 Let $(E,Q)$ be a  universal {\LieInftyAlgebroid} of $ {\mathcal F}$ which is minimal at $m$. Then $h\cI_m (E,Q)$ is a  Lie $\infty$-algebra structure on $ H^\bullet ({\mathcal F}, m)$, which we denote by  $ (H^\bullet ({\mathcal F}, m),Q_m)$ and
	 call the \emph{isotropy \Linfty-algebra of the singular foliation $ {\mathcal F}$ at $m$}.
 \end{definition}

 \isthiswhatyouowant The first part of the following proposition shows that this definition is independent of any choices. 
\begin{proposition}
 	\label{prop:isotropyStrictIso}
 	Any two isotropy \Linfty-algebras at $m $ of $ {\mathcal F}$, constructed out of two universal {\LieInftyAlgebroid} of $ {\mathcal F}$  minimal at $m$, are isomorphic, through an isomorphism whose linear part is the identity of  $ H^\bullet ({\mathcal F}, m)$.  
 
 \isthiswhatyouowant	
 	In particular, the $2$-ary bracket is a graded Lie algebra bracket on   $ H^\bullet ({\mathcal F}, m)$, which does not depend on any choice made in the construction.
 \end{proposition}

\isthiswhatyouowant	We first need to state a few obvious facts about Lie $\infty$-algebras whose $1$-ary bracket is equal to zero, as is the case for every isotropy Lie $\infty$-algebra defined like above.  

\begin{lemme}
	\label{lem:whenUnaryBraZero}
	Let $(V,Q)$ be a \Linfty-algebra whose $1$-ary bracket is equal to zero:
	\begin{enumerate}
		\item[1.] Its $2$-ary bracket is a graded Lie algebra bracket.  
	\end{enumerate}
	Let $(V',Q')$ be a second \Linfty-algebra whose $1$-ary bracket is equal to zero.
	\begin{enumerate}
		\item[2. ] The linear part of any \Linfty-algebra morphism from  $(V,Q)$ to $(V',Q')$ 
		is a graded Lie algebra morphism of the $2$-ary brackets.
		\item[3. ] The \Linfty-algebras $(V,Q)$ and $(V',Q')$ are isomorphic to one another if and only if they are homotopy equivalent.
	\end{enumerate}
\end{lemme}
\begin{proof}
	The first item  is an easy consequence of the higher Jacobi identity \eqref{superjacobi} for $n=3$.
	The second item is easily derived from the axioms of  \Linfty-algebra morphisms.
	
	\isthiswhatyouowant
	A graded commutative algebra morphism $\Phi \colon S((V')^*) \to S(V^*)$	is invertible if and only if its linear part $\phi \colon V \to V$' is invertible. Evidently, if $\Phi \colon (V,Q) \to (V',Q')$ is part of a homotopy equivalence, so is its 
 linear part $\phi \colon V \to V'$.  But a homotopy equivalence between complexes whose differential is zero has to be invertible.
\end{proof}

\begin{proof}[Proof (of Proposition \ref{prop:isotropyStrictIso})]
By the functorial properties of $h\cI_m$, 
two isotropy \Linfty-algebras at $m \in M$ of $ {\mathcal F}$ which are constructed out of two universal {\LieInftyAlgebroid} of $ {\mathcal F}$  minimal at $m$ are homotopy equivalent. The first assertion from the proposition then follows from item three of Lemma \ref{lem:whenUnaryBraZero}, the second one from items one and two. 
\end{proof}

\isthiswhatyouowant
If we restrict the graded Lie algebra structure on $H^\bullet (\mathcal{F},m)$, see item one in Lemma \ref{lem:whenUnaryBraZero}, its restriction to $H^{-1} (\mathcal{F},m)$ yields an ordinary Lie algebra.

\begin{proposition}
\label{prop:hol=hol}
The isotropy Lie algebra  of the singular foliation $ {\mathcal F}$ at a point $m \in M$, as defined by
Androulidakis and Skandalis, is isomorphic to  the degree minus one component $H^{-1} (\mathcal{F},m)$  of the isotropy Lie $\infty$-algebra of ${\mathcal F} $ at $m$.
\end{proposition}
\begin{proof}
The isotropy Lie algebra at $m$ of  I.\ Androulidakis and G.\ Skandalis \cite{AndrouSkandal} is the quotient $\mathfrak{g}_m := {\mathcal F}(m)/I_m {\mathcal F}(m)$ 
where ${\mathcal F}(m)$ denotes vector fields in ${\mathcal F} $ which vanish at $m$ and $I_m$ functions vanishing at that point. 
	
\isthiswhatyouowant We first define a Lie algebra morphism $\tau \colon  H^{-1} (\mathcal{F},m) \to \mathfrak{g}_m$. 
Let $(E,Q)$ be a universal Lie $\infty$-algebroid of ${\mathcal F}$ which is minimal at $m$.
 Choose $e \in {\mathfrak i}_m  E_{-1}$ in the kernel of $\rho_m$.
Let $\tilde{e}$ be a local section through $e$. Then $\rho(\tilde{e})$ belongs to $ {\mathcal F}(m)$.
Its class $\tau(e) \in \mathcal{F}(m)/I_m \mathcal{F}$ is well-defined, since another choice for $\tilde{e}$ differs from the first
one by a section in $I_m \Gamma(E_{-1})$.
This defines a linear map $\tau \colon H^{-1} (\mathcal{F},m) \to \mathfrak{g}_m$ which is easily checked to be a Lie algebra morphism.


\isthiswhatyouowant
It is clear that $\tau$ is surjective, since every local section of $\mathcal{F}$ vanishing at $m \in M$ is of the
form $\rho(\tilde{e})$ with $\tilde{e}$ a local section of $E_{-1}$ whose value at $m$ is in the kernel of $\rho$.

\isthiswhatyouowant
Now, let us prove injectivity. Let $e \in {\mathfrak i}_m E_{-1}$. Choose a local section $\tilde{e} $ of $E_{-1}$ through $e$.
If $ \tau(e) =0$, then $\rho(\tilde{e})$ is in the ideal  $I_m \mathcal{F}$. As q consequence, it is a finite sum of the form $\sum_{i=1}^r f_i X_i$,  
with $X_i \in \mathcal{F}$ and $f_i \in I_m$ for all $ i=1,  \dots,r$.
This implies $ \rho( \tilde{e} - \sum_{i=1}^r f_i \tilde{e}_i )=0$,
where $\tilde{e}_i$ is, for every $i=1, \dots,n$, a local section of $E_{-1}$ mapped to $X_i$ through $\rho$. Since  $ (E,\dd,\rho)$ is a geometric resolution, there exists a local section $\tilde{h} \in \Gamma(E_{-2})$ such that:
 \begin{equation}  \tilde{e} - \sum_{i=1}^r  f_i \tilde{e}_i = \dd^{(2)} \tilde{h}.\end{equation}
 Evaluated at $ m\in M$, this last relation gives that $e =0$.
 \isthiswhatyouowant This proves that  $\tau$ is injective and completes the proof.
\end{proof}

\subsection{Leaves and isotropy  Lie \texorpdfstring{$\infty$}{infinity}-algebras}
\label{moineau}


\isthiswhatyouowant
The isotropy \Linfty-algebra structure is attached to a leaf rather than to a point.
More precisely, let us consider a leaf $L$ of  ${\mathcal{F}}$.
The following proposition is a simple consequence of the fact that for every two points $m,m'$ in the same leaf $L$
of a singular foliation $ {\mathcal F}$, there exist neighborhoods $U_m,U_{m'}$ of these points
on which the singular foliations are isomorphic \cite{AndrouSkandal}.

\begin{proposition}\label{prop:doesNotDependsPoint}
Let $ {\mathcal F}$ be a singular foliation that admits a universal Lie $\infty$-algebroid in the neighborhood  of every point.
 For any two points in the same leaf of a singular foliation ${\mathcal F} $, 
 the isotropy \Linfty-algebra structures at these points are  isomorphic.
 \end{proposition} 

\begin{remarque}\normalfont
It deserves to be noticed that the isomorphism in Proposition \ref{prop:doesNotDependsPoint}
is not, in general, canonical: It may depend on the choice of an identification between $U_m$ and $U_{m'}$. An explicit Lie $\infty$-algebra isomorphism can be associated to a path staying inside the leaf $L$ which relates $m$ and $m'$. Homotopy equivalent paths yield homotopy equivalent isomorphisms.
\end{remarque}

\isthiswhatyouowant
The isotropy Lie $\infty$-algebra of a singular foliation is of interest only for singular leaves only:
\begin{lemme} \label{trivial}
The isotropy Lie $\infty$-algebra of a regular leaf is identically zero.
\end{lemme}
\begin{proof} Around every regular point $m \in M$, a minimal resolution, is given by $E_{-1}=T[1]F \subset T[1]M$, where $\cF = \Gamma(TF)$ and $E_{-i}=0$ for all $i >1$. In particular $H^\bullet(\cF,m) = 0$, since the anchor map is an inclusion map.  
\end{proof}

\isthiswhatyouowant
The codimension of the leaf gives an obvious restriction about the length of the graded space $H(\mathcal{F},m)$ on which the isotropy \Linfty-algebra is defined.

\begin{proposition}\label{prop:concentration}
Let $L$ be a leaf of a holomorphic, real analytic or smooth locally real analytic singular foliation ${\mathcal F}$.
The isotropy graded Lie algebra  
at a point $m \in L$ is concentrated in degrees $-1, \dots, -{\rm codim}(L)-1$.
\end{proposition}
\begin{proof} 
According to Proposition 1.12 in \cite{AndrouSkandal}, every singular foliation is, in a neighborhood of a point $m$
in a leaf $L$, the trivial product of a singular foliation on a neighborhood of $0$ in $ {\mathbb R}^{n -{\rm dim}(L)}$ 
(called the transverse foliation) with the regular foliation  $T{\mathcal B}$,
with ${\mathcal B} $ an open ball of dimension ${\rm dim}(L)$.
A geometric resolution of $ {\mathcal F}$ is therefore obtained by adding $T{\mathcal B}$ (in degree  minus one ) to a geometric
resolution of the transverse foliation. 
In the real analytic or holomorphic cases, the transverse foliation admits geometric resolutions
of length less or equal to $ {\rm co dim}(L)+1$ by  Proposition \ref{bonjourprop}. This concludes the proof in the holomorphic or real analytic cases.
The second item in Proposition \ref{bonjourprop} concludes the proof in the smooth locally real analytic case. 
\end{proof}

\isthiswhatyouowant

\isthiswhatyouowant 
Proposition \ref{prop:doesNotDependsPoint} implies that for all $i \geq 2$, the  image of the vector bundle morphism  $\dd^{(i+1)} :  {{\mathfrak i}_L}E_{-i-1} \to  {{\mathfrak i}_L}E_{-i}$
has the same dimension at all points of the leaf  $L$. Here ${{\mathfrak i}_L}$ stands for the restriction of a vector bundle to the submanifold $L \subset M$. 
This allows us to truncate the {\LieInftyAlgebroid} restricted to a leaf $L$ at a certain order $i$ in two different ways.  Both 
\begin{equation}  \label{trunc1}\bigslant{ {{\mathfrak i}_L} E_{-i}}{\mathrm{im}\, {\dd}^{(i+1)} } \longrightarrow 
{{\mathfrak i}_L}E_{-i+1} \longrightarrow \dots \longrightarrow {{\mathfrak i}_L} E_{-1} \longrightarrow TL 
\end{equation}
and 
\begin{equation} \label{trunc2}
 \bigslant{ {{\mathfrak i}_L} E_{-i}}{\ker {\dd}^{(i)} } \longrightarrow 
{{\mathfrak i}_L}E_{-i+1} \longrightarrow \dots \longrightarrow {{\mathfrak i}_L} E_{-1} \longrightarrow TL  
\end{equation} 
define a transitive Lie $i$-algebroid structure over $L$, the second one being a quotient of the former one. Here for $i=1$, we need to replace ${\dd}^{(1)}$ by $\rho$ in \eqref{trunc2}. The lowest possible truncations lead to known Lie algebroids over $L$: while \eqref{trunc2} gives simply $TL$, \eqref{trunc1} yields a transitive action Lie algebroid with fiber the isotropy Lie algebra of the leaf.
\begin{proposition}
	\label{prop:holoLieAlg}
	Let $\mathcal{F}$ be a singular foliation equipped with a universal {\LieInftyAlgebroid}  $(E,Q)$, and let $L$ be a leaf of $\mathcal{F}$.
	The  $1$-truncation of the restriction \eqref{trunc1} of   $(E,Q)$ to $L$  coincides with the holonomy Lie algebroid of $L$ defined by Androulidakis and Skandalis in~\cite{AndrouSkandal}. 
\end{proposition}
\begin{proof}
 	We recall the definition of the holonomy Lie algebroid in \cite{AndrouSkandal}.
 	Consider the vector bundle $A_L \to L$ whose fiber over $x \in L$ is  $\mathcal{F}/I_x\mathcal{F}$. Its sections are isomorphic to $\mathcal{F}/I_L\mathcal{F}$, where $I_L \subset {\mathcal O} $ is the ideal of functions vanishing along $L$. The Lie algebra bracket on  $\mathcal{F}/I_L\mathcal{F}$ equips $A_L$ with a Lie algebroid bracket.
The proof is then similar to  the proof of Proposition \ref{prop:hol=hol} and is left to the reader.
\end{proof}


\subsection{Examples of isotropy  Lie \texorpdfstring{$\infty$}{infinity}-algebras}
\label{sec:examplesLieInfinityIso}

\isthiswhatyouowant

\begin{example}
	\normalfont
For regular foliations, the isotropy Lie $\infty$-algebra is trivial at all points $m\in M$
because $H^\bullet ({\mathcal F},m)=0$ (see Lemma \ref{trivial}). 
\end{example}

\begin{example}
	\normalfont
For a Debord foliation, Proposition \ref{prop:Debord} implies that the isotropy Lie $\infty$-algebra at a point $m$ is concentrated in degree minus one, \emph{i.e.~}it is a Lie algebra. \end{example}

\begin{example}
\normalfont
Consider the singular foliation  given by the action of $\mathfrak{sl}_{2}$ on $\mathbb{R}^{2}$. Using the universal Lie $\infty$-algebroid structure of Example
\ref{sl2resolution}, we see that one obtains an isotropy Lie $2$-algebra  on  $H^{\bullet}(\mathcal{F},0)  \simeq \mathbb{R}[2]\oplus\mathfrak{sl}_{2}[1]$.
The restriction of the $2$-ary bracket to  $H^{-1}(\mathcal{F},0) \simeq \mathfrak{sl}_{2}$ is the usual Lie algebra bracket and the $2$-ary bracket  of an element of degree  minus one  with an element of degree minus two vanishes.  There are no $k$-ary brackets for $k \geq 3$.

\isthiswhatyouowant
The appearance of the apparently unnecessary abelian factor $\mathbb{R}[2]$ in the previous example can be seen also as a consequence of the third item in Proposition \ref{prop:Debord}.
\end{example}

\begin{example}
\normalfont
Consider the singular foliation given by all vector fields on a vector space $V$ vanishing at the origin.
Using the universal Lie $\infty$-algebroid structure of  Example \ref{ex:vfVanishingAtZero_structure}, we see that the isotropy Lie $\infty$-algebra at the origin reduces to the graded Lie algebra $\bigoplus_{i \geq 1}^n \wedge^{i} V^* \otimes V $
equipped with the (graded symmetric) Lie bracket defined as in (\ref{eq:bracketVFVatZero}).
There is no $k$-ary bracket for $k \neq 2$.
\end{example}

\begin{example}
\normalfont
\label{conjectural}
Consider the singular foliation ${\mathcal F}_{ad}$ for the adjoint action of a semi-simple complex Lie algebra  ${\mathfrak g} $ on itself, studied in Examples \ref{conjecture} and \ref{conjecture2}. 
The isotropy Lie $\infty$-algebra at $0$ is defined on the graded vector space 
$ H^\bullet({\mathcal F}_{ad},m) = {\mathbb C}^l[2] \oplus \mathfrak{g}[1]$, where $l$ is the rank of $\mathfrak{g}$. The only non-vanishing bracket is the $2$-ary bracket between two elements of  degree  minus one; it coincides with the Lie bracket on  ${\mathfrak g} $.
\end{example}

\begin{example}
	\normalfont
	\label{ex:vanishingOrder2_v3}
	Consider the foliation $ {\mathcal F}$ of vector fields vanishing to order $2$ at $0\in {\mathbb K}^2$, studied in Example \ref{ex:vanishingOrder2_v2}. The isotropy Lie $\infty$-algebra at the origin is a Lie $2$-algebra whose $n$-ary brackets are equal to zero for all $n$, with underlying graded vector space ${\mathbb K}^2[2] \oplus {\mathbb K}^6[1]$.  More precisely, the $2$-ary bracket is zero by Equation \eqref{eq:Quadra}.
		 This equation  also implies that the Jacobiator ${\rm Jac}(X,Y,Z)$  of three elements belongs to $I_0^2 \, \Gamma(E_{-1}) $. Since the $3$-ary bracket satisfies $\dd^{(2)} \circ \{\cdot,\cdot,\cdot\}_3 = {\rm Jac}\, (\cdot,\cdot,\cdot)$ and since $\dd^{(2)}$ is injective except at the origin, the $3$-ary bracket has to be valued in $I_0 \Gamma(E_{-2})$. 
\end{example}

 \begin{example}
 \label{ex:Koszul3}
  \normalfont
 Consider the
  singular foliation ${\mathcal F}_\varphi$ of all
 vector fields $X$  on $M={\mathbb C}^n$ such that $X[\varphi]=0$,
 with $\varphi$ a weight homogeneous function 
 with isolated singularities. 
  Example \ref{ex:Koszul2}  describes a universal {\LieInftyAlgebroid} for this singular foliation.

\isthiswhatyouowant
The origin $0$ of ${\mathbb C}^n$ is a leaf. 
Since all partial derivatives of $\varphi$ vanish at the origin $0$, 
the Koszul  resolution, see Example \ref{ex:Koszul}, is a minimal geometric resolution at $0$.
Hence $H^{-k}({{\mathcal F}_{\varphi}},0) \simeq \wedge^{k+1} {\mathbb C}^n$.
For all $k \geq 2$, the $k$-ary brackets of the universal {\LieInftyAlgebroid}   given by Equation (\ref{eq:Koszul}) restrict  as follows:
\begin{equation}
 \label{eq:Koszul2}
 \left\{  \partial_{I_1}, \dots, \partial_{I_k} \right\}_k\quad :=  
 \sum_{i_1 \in I_1, \dots,i_k \in I_k}  \epsilon(i_1, \dots,i_k)  \, \tfrac{\partial^k \varphi}{ \partial x_{i_1} \dots \partial x_{i_k}}(0)\;\: \partial_{I_1^{i_1}\bullet I_2^{i_2}\bullet \ldots \bullet I_k^{i_k} } 
 \end{equation}
 where notations are as in Example \ref{ex:Koszul2}.
 \end{example}

\subsection{On the existence of a Lie algebroid defining a singular foliation}

\subsubsection{Lie algebroids of minimal rank}

\label{sec:3aryNotZero}

\isthiswhatyouowant
In this section we exploit the cohomologies of the isotropy Lie $\infty$-algebra at a point and define a class in the third Chevalley-Eilenberg cohomology of the Androulidalis-Skandalis isotropy Lie algebra. This class was later on generalized  for an arbitrary  {\Linfty}-algebroid by Ricardo Campos in \cite{ricardo}.


\isthiswhatyouowant
\begin{proposition}
	\label{propdef:NMRLA}
	Let ${\mathcal F}$ be a singular foliation that admits a geometric resolution of finite length
	in a neighborhood of $ m \in M$.
 Equip $  H^\bullet ({\mathcal F},m) = \bigoplus_{ i\geq 1}   H^{-i} ({\mathcal F},m)$ with the isotropy Lie $\infty$-algebra brackets $( \{\cdots \}_k) _{k \geq 2}$ constructed out of some universal Lie $\infty$-algebroid $(E,Q)$ minimal at $m$. 
	\begin{enumerate}
		\item The restriction of the $2$-ary bracket $$ \{\cdot, \cdot \}_2 \colon H^{-1} ({\mathcal F},m) \otimes H^{-2} ({\mathcal F},m) \longrightarrow H^{-2} ({\mathcal F},m)$$ makes $H^{-2} ({\mathcal F},m)$ a module over the isotropy Lie algebra $H^{-1} ({\mathcal F},m)$, which does not depend on the choice of $(E,Q)$.
		\item The restriction  of the $3$-ary bracket
		$$ \{\cdot, \cdot, \cdot \}_3 \colon \Lambda^3  H^{-1} ({\mathcal F},m)  \longrightarrow H^{-2} ({\mathcal F},m)$$
		 is a $3$-cocycle for the Chevalley-Eilenberg complex of 
		$ H^{-1} ({\mathcal F},m)$ valued in  $H^{-2} ({\mathcal F},m)$.
		\item The cohomology class of this cocycle does not depend on the choice of $(E,Q)$.
	\end{enumerate}
	\end{proposition}
\begin{proof}
	The first item follows from Proposition \ref{prop:isotropyStrictIso}.
	The second item is an easy consequence of the higher Jacobi identity \eqref{superjacobi} applied to four elements of degree minus one. 
	The third item can be obtained as follows: Two different choices made in the construction of
	the universal {\LieInftyAlgebroid} give isotropy \Linfty-algebra structures which are strictly isomorphic by Proposition \ref{prop:isotropyStrictIso}, through isomorphisms
	whose linear parts are the identity. The second Taylor coefficient of this isomorphism has a component which is a
	map $\tilde{\theta} \colon \Lambda^2 \big(H^{-1}({\mathcal F},m)\big)\to H^{-2}({\mathcal F},m)$.
	Writing explicitly the definition of \Linfty-algebra morphisms, applied to three elements in $ H^{-1}({\mathcal F},m)$, one obtains that
	the $3$-ary bracket is a multiple of the Chevalley-Eilenberg differential of~$ \tilde{\theta}$.
\end{proof}
\begin{definition} We call the 3-cohomology class of Proposition \ref{propdef:NMRLA} the No-Minimal-Rank-Lie-Algebroid class or the  \emph{NMRLA class} for short.
\end{definition}

\isthiswhatyouowant
Recall that the rank of $ {\mathcal F}$ at $m \in M$ is the minimal number of generators of 
$ {\mathcal F}$ in a neighborhood of that point.

\begin{proposition}
	\label{prop:NMRLA}
	Let ${\mathcal F}$ be a singular foliation on a manifold $M$ that admits a geometric resolution of finite length.
    Let $r$ be the rank of $\mathcal{F}$ at the point $m$.

\isthiswhatyouowant	If the NMRLA $3$-class does not vanish, then it is not possible to find a Lie algebroid $(A, [\cdot,\cdot],\rho)$
	defined in a neighborhood $U_m$ of $m$ which  satisfies the two following conditions: \begin{enumerate}
		\item the rank of the vector bundle $A$ is $r$ and 
		\item  $\rho\big(\Gamma(A)\big) = {\mathcal F}\vert_{U_m} $.
	\end{enumerate}
\end{proposition}

\isthiswhatyouowant The proof requires some preparation.
	
	\begin{lemme}
		\label{lem:rank}
		For every geometric resolution $(E,\dd,\rho)$ of $ {\mathcal F}$ which is minimal at $m$, 
		the rank of the vector bundle $E_{-1}$ is equal to the rank $r$ of $ {\mathcal F}$ at $m$.
	\end{lemme}
	\begin{proof}
		Let $r_E$ and $r_{m}$ be the respective ranks of the vector bundle $E_{-1}$ and of the singular foliation $ {\mathcal F}$ at $m$. Let $ e_1, \dots, e_{r_E}$ be a local trivialization of  $E_{-1}$.
		Since $ \rho(\Gamma(E_{-1})) = {\mathcal F}$, the family  $(\rho ( e_i))_{i=1, \dots,r_E} $ generates $ {\mathcal F}$ as an $ {\mathcal O}$-module. Hence $ r \leq r_E$. 
		If $  r < r_E$, then one of  these generators is a linear combination of the others.
		Without loss of generality, we can assume that it is $e_1$.
		 Therefore,  there exist functions $ f_2, \dots, f_{r_E} \in {\mathcal O}$ such that
		 $$  \rho(e_1) = \sum_{i=2}^{r_E}\, f_i \, \rho(e_i) .$$
		 This implies that $  e_1 = \sum_{i=2}^{r_E} f_i e_i + \dd^{(2)} g $ for some  $g \in \Gamma(E_{-2})$. Since   $(E,\dd= (\dd^{(i)})_{i \geq 2},\rho)$ is minimal at $m$, this relation implies upon evaluation at this point: 
		 $ e_1(m) = \sum_{i=2}^{r_E} f_i (m) \, e_i(m)$. This contradicts the assumption that   $ e_1, \dots, e_{r_E}$ is a local trivialization. 
		 Hence $ r_E = r $.
		\end{proof}

	\begin{proof}[Proof of Proposition \ref{prop:NMRLA}]
 Assume that a Lie algebroid $(A,[\cdot, \cdot] , \rho_A) $ satisfying $\rho_A\big(\Gamma(A)\big) ={\mathcal F}$ exists. 
		Let $(E,Q)$ be a universal {\LieInftyAlgebroid} \resolving ${\mathcal F}$  in a neighborhood of $m$, built on a geometric resolution $(E,\dd, \rho)$ minimal at $m$.
		By Theorem \ref{theo:onlyOne}, a {\LieInftyAlgebroid} morphism $\Phi$ over the identity of $M$ from $A$ to $(E,Q)$ exists. We need a first step: 
		
		\begin{lemme} \label{lem:onto} Upon restriction to $m$, the vector bundle morphism $\Phi_1 \colon A \to E_{-1}$ becomes a surjective linear map.
			\end{lemme} 

\begin{proof}[Proof of Lemma \ref{lem:onto}]
    Let $ \tilde{e} \in \Gamma( E_{-1})$ be a section through a given $e \in {\mathfrak i}_m E_{-1} $. 	 Then, by the assumption on $A$, there exists a local section $a \in \Gamma(A) $ such that $ \rho_A (a) = \rho (\tilde{e}) $.  By the morphism property, this yields $ \rho (\tilde{e} - \Phi_1(a)) =0 $. Because $(E,\dd,\rho) $ is a geometric resolution of $  {\mathcal F}$, this implies that there exists some $ g \in \Gamma(E_{-2})$ such that $ \tilde{e} - \Phi_1(a)  = \dd^{(2)} g$. Since $ (E,\dd,\rho)$ is minimal at $m$, when evaluated at this point, the last relation gives
     $  e =  {\mathfrak i}_m \Phi_1  ( a(m)) $. Hence  $ {\mathfrak i}_m \Phi_1 \colon  {\mathfrak i}_m A \to  {\mathfrak i}_m E_{-1} $ is surjective. 
     \end{proof}

\isthiswhatyouowant
We complete the proof of Proposition \ref{prop:NMRLA}.
Assume now that the rank of the vector bundle $A$ is $r$, so that $A$ and $E_{-1}$ are vector bundles of the same rank by Lemma \ref{lem:rank}. By Lemma \ref{lem:onto}, upon restriction at $m$, the vector bundle morphism $\Phi_1 \colon A \to E_{-1}$ becomes a bijective linear map. 
	%
	By definition of a Lie $\infty$-algebroid morphism, the linear, quadratic and cubic terms in the Taylor coefficients of $\Phi$
	satisfy for all $a,b,c \in \Gamma(A)$:
	\begin{equation}
	\label{cobound}
	\big\{ \Phi_1 (a) , \Phi_2 (b,c) \big\}_2 - \Phi_2\big(\{a,b\}_2,c\big)    + \CricArrowRight{\hbox{\tiny{$abc$}}} = 
	\big\{\Phi_1(a),\Phi_1(b),\Phi_1(c)\big\}_3  -  {\rm d}^{(2)} \, \Phi_3(a,b,c).
	\end{equation}  
	If we now evaluate this equation at the point $m$, the last term on the r.h.s.\ vanishes by minimality. 	 By the above bijectivity, moreover, we can assume that  ${\mathfrak i}_m A ={\mathfrak i}_m E_{-1}$ and that, in particular, the first order Taylor coefficient ${\mathfrak i}_m \Phi_1 \colon {\mathfrak i}_m A \to {\mathfrak i}_m E_{-1}$ of the Taylor coefficient of $\Phi$ at $m$ is the identity map. 
	If we choose the sections ${\mathfrak i}_m a,{\mathfrak i}_m b,{\mathfrak i}_m c$ in the kernel of ${\mathfrak i}_m \rho$, the equation resulting from Equation \eqref{cobound} in this way shows that $\{.\,,.\,,.\}_3$ is a Chevalley-Eilenberg coboundary.
\end{proof}


\begin{example}
	\label{ex:NMLRA}
	\normalfont Let $n \geq 4$ and $ \varphi \colon  {\mathbb C}^n \to {\mathbb C}, (x_1, \ldots,x_n) \mapsto \sum_{i=1}^n x_i^3 $. 
	Consider the singular foliation ${\mathcal F}_\varphi$ 
	of all vector fields $X$ on $ {\mathbb C}^n $ satisfying $ X[\varphi]=0$, which, according to \eqref{forthecounting}, has rank $ n(n-1)/2$. 
	The function $\varphi$ has an isolated singularity at the origin and therefore  satisfies the requirements of Example \ref{ex:Koszul3}.
	It follows from Equation (\ref{eq:Koszul2}) that, at the origin $0$, one has: 
	\begin{enumerate}
		\item The $2$-ary bracket of the isotropy \Linfty-algebra vanishes. Thus  the isotropy Lie algebra is abelian and its action on $ H^{-2}({\mathcal F},0)$ is trivial.
		\item  The $3$-ary bracket is non-zero  since (see the conventions of Example \ref{ex:Koszul3}):
		$$  \{\partial_{12},\partial_{13},\partial_{14}\}_3 := \partial_{234} .$$
	\end{enumerate}
	This implies that the NMRLA $3$-class of ${\mathcal F}_\varphi$ at the origin is non-zero.
	Therefore, the foliation ${\mathcal F}_\varphi$ cannot be induced by a Lie algebroid of rank $ n(n-1)/2$, around the origin even not locally.
\end{example}

\isthiswhatyouowant
Example \ref{ex:NMLRA} and Proposition \ref{prop:NMRLA} imply the following result.

\begin{proposition}
	\label{prop:nonexist}
	There exist singular foliations of rank $r$ that, even locally, can not be induced by  a Lie algebroid of rank $r$. 
\end{proposition}

\isthiswhatyouowant
Let us give an even more explicit interpretation of Proposition \ref{prop:nonexist}.
Let $ {\mathcal F}$ be a singular foliation generated by vector fields $X_1, \dots,X_r$.
There may be \emph{relations} between these vector fields, i.e.~$r$-tuples $f_1, \dots,f_r \in {\mathcal O}$ such that $\sum_{i=1}^r f_i X_i=0$. In fact, this always happens except  if, and only if, $ {\mathcal F}$ is Debord, see Example \ref{ex:debord}. If there are such relations, the functions $c_{ij}^k \equiv - c_{ji}^k\in {\mathcal O}$ in 
\begin{equation} [X_i,X_j]= \sum_{k=1}^r c_{ij}^k X_k \label{cijk} \end{equation}
are not unique.
On the other hand, an easy computation gives
\begin{equation} [[X_i,X_j],X_k] +\CricArrowRight{\hbox{\tiny{$ijk$}}}  = \sum_{l=1}^r J_{ijk}^l X_l \label{cycl} \, ,
\end{equation} where we defined the symbols  $J_{ijk}^l\in {\mathcal O}$ by 
$$ J_{ijk}^l :=\left(  \sum_{m=1}^r c_{ij}^m c_{mk}^l  - X_k[ c_{ij}^l] \right  ) +   
\CricArrowRight{\hbox{\tiny{$ijk$}}} .$$
Since vector fields on a manifold satisfy the Jacobi identity, the left hand side of Equation \eqref{cycl} is identically zero.
This implies that the family $J_{ijk}^{1}, \dots, J_{ijk}^{r} $ constitutes a relation between the vector fields $X_1, \dots,X_r$,
but, in general, these functions may be different from zero.
A natural question is: 
\vspace{1mm}

{\centering "Given a fixed singular foliation $ {\mathcal F}$ on a manifold $M$ generated by $X_1, \ldots , X_r$, \\can we choose functions $c_{ij}^k \equiv - c_{ji}^k$ subject to  Equation \eqref{cijk} such that $J_{ijk}^{l}\equiv 0$?" \\}

\vspace{3mm}
\isthiswhatyouowant Suppose the answer to this question is positive. 
Consider the bundle $A:=M \times \mathbb{R}^r$ and denote its constant standard basis of sections by $e_1, \ldots , e_r$. Define an anchor by means of the map $\rho \colon A \to TM, e_i \mapsto X_i$ and use it to extend the 2-bracket $[e_i,e_j] := \sum_{k=1}^r c_{ij}^k e_k$ to a bracket between arbitrary sections of $A$. It is now easy to verify that this turns $A$ into a Lie algebroid. 
Thus, by Proposition \ref{prop:nonexist}, the answer to the above question is negative in general.

\subsubsection{The Leibniz algebroid of a singular foliation}
\label{sec:leibnizoid}

\isthiswhatyouowant
As already mentioned in the introduction, it is not easy to know if a  singular foliation is, locally,
the image of a Lie algebroid under the anchor map. It is known not to be the case globally, cf. \cite{AndrouZambis},
while in a neighborhood of a point the question is open.
We are not able to give a positive or negative answer to decide this question,
but it is an easy consequence of Theorem \ref{theo:existe} that a Leibniz algebroid exists. Had Theorem \ref{theo:existe} not be proven, this would be a far from obvious result.

\begin{definition}
	\cite{Kotov}
	Let $L$ be a vector bundle over $M$. A \emph{Leibniz algebroid structure on $L$} is a bilinear assignment $[\ .\ ,\, .\ ]_L : \Gamma(L) \otimes \Gamma(L) \to \Gamma(L)$
	and a vector bundle morphism $\rho:L\to TM$,
	satisfying the \emph{Loday-Jacobi condition}:
	\begin{equation}\label{loday}
	\big[x,[y,z]_L\big]_L=\big[[x,y]_L,z\big]_L+\big[y,[x,z]_L\big]_L
	\end{equation}
	for all $x,y,z\in \Gamma(L)$, and the \emph{Leibniz identity}:
	\begin{equation}
	[x,fy]_L=f[x,y]_L+\rho(x)[f]\, y
	\end{equation}
	for every $x,y\in L$ and $f\in\mathscr{O}$.
\end{definition}



\isthiswhatyouowant
%
Note that the skew-symmetrization of a Leibniz algebroid bracket does not turn a Leibniz Lie algebroid into an almost Lie algebroid in general. Adapting Proposition 5.4 item 1 and Lemma 5.5 in \cite{Melchior}  to our case, we easily derive  the following result.

\begin{proposition}
	\label{prop:LeibnizReallyExists}
	Let $\mathcal{F}$ be a singular foliation that admits a universal {\LieInftyAlgebroid} $(E,Q)$ with anchor $\rho$. Assume that its associated geometric resolution
	is of finite length. Then  $L=\big(S(E^*) \otimes E\big)\big|_{-1} $ is a vector bundle of finite rank and comes 
	with a Leibniz algebroid structure, when equipped with:
	\begin{enumerate}
		\item the Leibniz bracket defined by: $$ [X,Y]_L := \big[[Q,X],Y\big]  $$ for all $X,Y \in \Gamma(L)$ (identified with vertical 
		vector fields $\partial_{X}$ and $\partial_{Y}$ of degree  minus one  on the graded manifold $E$),
		\item the anchor given by the composition:
		\begin{center}
			\begin{tikzcd}
				&E_{-1}\oplus\big(\bigoplus_{k \geq 1} S^k(E^*) \otimes E\big)\big|_{-1} \ar[r]& E_{-1 } \ar[r,"\rho"]& TM .
			\end{tikzcd}
		\end{center}
	\end{enumerate}
\end{proposition} 
\begin{proof}
	For every graded Lie algebra, ${\mathfrak g}:=\bigoplus_{ i\in {\mathbb Z}}{\mathfrak g}_i $ and every homological element $Q\in\mathfrak{g}$ of degree $+1$,
	$ {\mathfrak g}_{-1}$ is a graded Leibniz algebra when equipped with the bracket $(X,Y) \mapsto \big[[Q,X],Y\big]$, cf. \cite{YKS}. Applied to the graded Lie algebra of
	derivations of functions ${\mathcal E}$ 
	on the {\LieInftyAlgebroid} $(E,Q)$ (that is, vector fields on the $N$-manifold $E$) and to the vector field $Q$,
	the bracket given as above  induces a Leibniz algebra bracket on vector fields of degree  minus one. 
	Now, every vertical vector field of degree  minus one  is ${\mathscr O}$-linear derivations of $ {\mathcal E}$.   Vertical vector field of degree  minus one
	can therefore be identified with sections of the vector bundle  $L=\big(S(E^*) \otimes E\big)|_{-1}$, \emph{i.e.}~the degree  minus one  component of $S(E^*) \otimes E$.
	Also, since sections of $S^{k}(E^*) \otimes E $ are of non-negative degree for $k \geq n$, $L$  is a vector bundle 
	of finite rank over $M$. 
	One checks directly that the anchor $\eta$ is given as in item 2.
	By construction, $\eta\big(\Gamma(L)\big)= \rho\big(\Gamma(E_{-1})\big)= {\mathcal F}$.
	This proves the  proposition.\end{proof}

\isthiswhatyouowant
The following proposition is an immediate consequence of Proposition \ref{prop:LeibnizReallyExists} and Theorem~\ref{theo:existe}.

\begin{proposition}
	\label{eq:LeibnizExists}
	Let $\mathcal{F}$ be a singular foliation that admits a geometric resolution of finite length. 
	Then  there exists a Leibniz algebroid structure whose induced singular foliation is  $\mathcal{F}$.
\end{proposition}

\isthiswhatyouowant
Indeed, one could imitate even further the construction in \cite{Melchior} and obtain a Vinogradov algebroid, by adding vector fields of degree $-2$ into the picture.



\subsection{The holonomy Lie groupoid as the fundamental  groupoid of the universal Lie $\infty$-algebroid}
\label{sec:sub-grou}

\isthiswhatyouowant
Let us define  the fundamental or \v Severa groupoid of a {\LieInftyAlgebroid} $(E,Q)$ over a manifold $M$.  It is a suitable generalization of the fundamental groupoid of a manifold to the world of Q-manifolds: essentially everywhere the interval $I := [0,1]$ is replaced by its tangent Lie algebroid $(T[1]I,\dd_{dR})$. To our knowledge this construction was first proposed by \v Severa (cf.\ letter 8 of \cite{Severaletters} and \cite{Severahomotopy}). 
It reproduces the previously found symplectic groupoid of  Cattaneo and Felder  \cite{Cattaneo} when applied to the Lie algebroid $T^*M$ associated to a Poisson manifold $(M,\pi)$, where it was found in terms of the reduced phase space of the Hamiltonian formulation of the Poisson sigma model \cite{Ikeda,Schaller_Strobl}. 
Most famously, it was studied for its smoothness properties by Cranic and Fernandes \cite{CrainicFernandes}, under the name of the Weinstein groupoid, where the necessary and sufficient conditions for a (smooth) integration of a Lie algebroid were found.\footnote{In \cite{Severaletters,Severahomotopy} a generalization of this to Lie $n$-algebroids was proposed so as to yield a (not necessarily smooth) integration to $n$-groupoids. Here, however, we will apply the fundamental groupoid construction also to {\LieInftyAlgebroid}s: It is clear that the fundamental groupoid depends only on the components of degree minus one and minus two of a Lie $\infty$-algebroid. }

\isthiswhatyouowant In more detail: 
\begin{enumerate}
	\item We call \emph{$E$-path} a morphism of {\LieInftyAlgebroid}s from $(TI,\dd_{dR})$ to 
	$(E,Q)$. Since $I$ is just one-dimensional, the morphism property only captures information about the anchor. In particular, 
	$E$-paths are in one-to-one correspondence with paths $a \colon I \to E_{-1}$ covering a path $\gamma \colon I \to M$ such that:
	\begin{equation} \frac{{\mathrm d} \gamma(t)}{{\mathrm  d}t} = \rho \big( a \, (t)\big).\end{equation}
	 An $E$-path is said to be \emph{trivial} if, for all $t\in I$, $\gamma(t)$ is a constant path,  $\gamma(t)=x$ for some $x \in M$, and in addition $a(t)=0_x$. 
	\item 
	A \emph{homotopy between two $E$-paths} $a_0$ and $a_1$ is a
	{\LieInftyAlgebroid} morphism from $(TI^2,\dd_{dR})$ to $(E,Q)$ whose restrictions to $\{0\} \times I$ and
	$\{1\} \times I$  are $a_0$ and $a_1$ respectively, while the restrictions to 
	$ I \times \{0\}$ and $ I \times \{1\}$ are trivial $E$-paths.  
\end{enumerate}
  Homotopy defines an equivalence relation on $E$-paths.  Concatenation of paths\footnote{The concatenation of $E$-paths may not be smooth, but every $E$-path is homotopy equivalent to an $E$-path which is trivial in neighborhoods of $0$ and $1$, and the concatenation of such $E$-paths is smooth.} is compatible with this equivalence relation and turns the quotient set into a groupoid over $M$ that we call the  \emph{fundamental groupoid} or \emph{\v Severa groupoid} of $ (E,Q)$.
 %
 The definition is justified because homotopy equivalent {\LieInftyAlgebroid}s obviously have isomorphic  fundamental groupoids. Here is the main result of this section.

\begin{proposition}
	\label{prop:recover_AS}
Let $(E,Q)$ be a universal {\LieInftyAlgebroid} \resolving a singular foliation $\mathcal{F}$.
The  fundamental groupoid of $(E,Q)$ is the universal cover of the holonomy groupoid
described by Androulidakis and Skandalis in \cite{AndrouSkandal}.
\end{proposition}

\isthiswhatyouowant
The proof of this proposition uses the following lemma.
\begin{lemme}
	\label{lem:leafByleaf}
	The  restriction of the fundamental groupoid of any universal Lie $\infty$-algebroid $(E,Q)$ of $ {\mathcal F}$ to a leaf $L$ coincides with the fundamental groupoid of the
	holonomy Lie algebroid of $L$.
\end{lemme}
\begin{proof}
	Let $L$ be a leaf, $A_L$ its holonomy Lie algebroid, and $(E_L,Q)$  the restriction of the Lie $\infty$-algebroid $(E,Q)$ to $L$. 
	
	\isthiswhatyouowant
It follows from Proposition \ref{prop:holoLieAlg} that the following sequence of vector bundles is exact:
\begin{equation}
\label{exact}
 0 \to {\dd}^{(2)} \big( {{\mathfrak i}_L} E_{2}\big) \to  {{\mathfrak i}_L} E_{-1} \stackrel{\pi}{\to} A_L \to 0 .
\end{equation} 
Here the map $\pi \colon {{\mathfrak i}_L} E_{-1}\to A_L$ is a morphism of almost Lie algebroids, thus preserving the anchor and the 2-brackets. Therefore 
	the  projection ${\mathfrak i}_L E _{-1}  \to A_L $ 
	maps an $E_L$-path to an $A_L$-path 
	and maps homotopies of $E_L$-paths to homotopies of $A_L$-paths.
	Hence, the  fundamental groupoid of  $(E,Q)$ maps to the fundamental groupoid of $A_L$.
	
	\isthiswhatyouowant
	Let us check that this map is bijective. Surjectivity is obvious:  every section $ \sigma \colon A_L \to {\mathfrak i}_L E_{-1}$ of $\pi$ in (\ref{exact}) lifts an $A_L$ path to an $E_L$-path.
	
	\isthiswhatyouowant
	Let us check injectivity.
	Let $ \alpha \colon TI^2 \to A_L$ be an algebroid morphism whose restriction
	to the boundaries satisfies the usual requirements of homotopies relating two $A_L$-paths $a_1$ and $a_2$.
	It is easy to check that there exists 
	a vector bundle morphism $\tilde{\alpha}  \colon TI^2 \to {{\mathfrak i}_L} E_{-1}$ that satisfies the requirements of a homotopy of $E_L$-paths when restricted to boundaries,
	that relates arbitrary lifts of $a_1$ and $a_2$, and that satisfies $ \pi \circ \tilde{\alpha} = \alpha$. 
	
	\isthiswhatyouowant
	The vector bundle morphism $\tilde{\alpha}$ may not be a {\LieInftyAlgebroid} morphism, \emph{i.e.}~$\Psi : = \tilde{\alpha} ^* \circ Q - {\diff}_{\mathrm{dR}}\circ \tilde{\alpha} ^*$ may not be zero.
	
	\isthiswhatyouowant
	However,  since $\alpha$ is a Lie algebroid morphism, $\Psi \colon \Gamma\big({{\mathfrak i}_L}  E_{-1}^*\big) \to \Omega^2 (I^2) $  is linear over functions
	and vanishes  on the image of  $\pi^* \colon \Gamma ( A^*_L) \to \Gamma\big({{\mathfrak i}_L}  E_{-1}^*\big)$. By the linearity, it is
	 the dual of a vector bundle morphism $ \kappa  \colon \Lambda^2 TI \to {\mathfrak i}_L  E_{-1}$ and, by the second property, $\kappa$ is valued in the kernel of $ \pi$.
	By exactness of the complex (\ref{exact}), there exists a vector bundle morphism
	$\tilde{\kappa} \colon  \Lambda^2 TI \to  {\mathfrak i}_L  E_{-2} $ such that $ \dd^{(2)} \circ \tilde{\kappa} = \kappa$.
	The pair $ (\tilde{\alpha}^*, \tilde{\kappa}^*)$ defines a graded algebra morphism
	from $ {\mathcal E}$ to $ \Omega(I^2)$ which is easily checked to be a Lie $\infty$-algberoid morphism. 
\end{proof} 
\begin{remarque} \normalfont%
	A close look at its proof may convince the reader that  Lemma \ref{lem:leafByleaf} is valid for any Lie $\infty$-algebroid, universal or not,  inducing $ {\mathcal F}$ upon replacing the holonomy Lie algebroid of $L$ with the
	$1$-truncation Lie algebroid $\bigslant{{{\mathfrak i}_L} E_{-1}}{{\dd^{(2)}} ({{\mathfrak i}_L}E_{-2}})$ as in (\ref{trunc1}). 
\end{remarque}

\begin{proof}[Proof of Proposition \ref{prop:recover_AS}]
Given a singular foliation $\mathcal{F}$, the holonomy groupoid of $\mathcal{F}$ described in \cite{AndrouSkandal} is a topological groupoid, whose leaves are the leaves of $\mathcal{F}$.
Moreover, according to \cite{AndrouSkandal},  its restriction to every leaf $L$ of $\mathcal{F}$ is a smooth groupoid
integrating the holonomy Lie algebroid $A_L$ of that leaf. 
Since the fundamental groupoid of $A_L$ is simply connected, it is the universal cover of every source-connected Lie groupoid integrating $A_L$, and the Proposition follows from Lemma \ref{lem:leafByleaf}
\end{proof}


\bibliography{mabibliographie1}

\end{document}